\documentclass[12pt]{article}

\usepackage{amsmath,amsthm,amsfonts,amssymb,color}

\newtheorem{theorem}{Theorem}[section]
\newtheorem{corollary}[theorem]{Corollary}

\newtheorem{proposition}[theorem]{Proposition}
\newtheorem{lemma}[theorem]{Lemma}
\newtheorem{definition}[theorem]{Definition}
\newtheorem{remark}[theorem]{Remark}

\def\cA{\mathcal{A}}
\def\cB{\mathcal{B}}
\def\cD{\mathcal{D}}
\def\cE{\mathcal{E}}
\def\cF{\mathcal{F}}
\def\cH{\mathcal{H}}
\def\cP{\mathcal{P}}
\def\cS{\mathcal{S}}
\def\cU{\mathcal{U}}

\def\bP{\mathbb{P}}
\def\bC{\mathbb{C}}
\def\bD{\mathbb{D}}
\def\E{\mathbb{E}}
\def\bR{\mathbb{R}}

\begin{document}

\title{
 Existence of density for the stochastic wave equation with
space-time homogeneous Gaussian noise}

\author{Raluca M. Balan\footnote{University of Ottawa, Department of Mathematics and Statistics, STEM Building, 150 Louis-Pasteur Private, Ottawa, ON, K1N 6N5, Canada. E-mail address:
rbalan@uottawa.ca. Research supported by a grant from the
Natural Sciences and Engineering Research Council of Canada.} \and Llu\'is
Quer-Sardanyons\footnote{Departament de Matem\`atiques, Universitat
Aut\`{o}noma de Barcelona, 08193 Bellaterra (Barcelona), Catalonia,
Spain. E-mail address: quer@mat.uab.cat. Research supported by the
grant MTM2015-67802P.}
 \and Jian Song\footnote{Corresponding author. School of Mathematics, Shandong University, Jinan, Shandong 250100, P. R.
China. E-mail address: txjsong@hotmail.com}}


\date{December 18, 2018}
\maketitle

\begin{abstract}
\noindent In this article, we consider the stochastic wave equation
on $\bR_{+} \times \bR$,
driven by a linear multiplicative space-time
homogeneous Gaussian noise whose temporal and spatial covariance
structures are given by locally integrable functions $\gamma$ (in
time) and $f$ (in space), which are the Fourier transforms of
tempered measures $\nu$ on $\bR$, respectively $\mu$ on $\bR$. Our
main result shows that the law of the solution $u(t,x)$ of this
equation is absolutely continuous with respect to the Lebesgue
measure.
\end{abstract}


\vspace{1mm}



\section{Introduction}

In this article, we study the absolute continuity of the law of the solution to the stochastic wave equation with linear multiplicative noise, in spatial dimension $d=1$:
\begin{equation}
\left\{\begin{array}{rcl}
\displaystyle \frac{\partial^2 u}{\partial t^2}(t,x) & = & \displaystyle \frac{\partial^2 u}{\partial x^2}(t,x)+u(t,x)\dot{W}(t,x), \quad t>0, x \in \bR, \\[2ex]
\displaystyle u(0,x) & = & 1, \quad x \in \bR,\\[1ex]
\displaystyle \frac{\partial u}{\partial t}(0,x) & = & 0, \quad x \in \bR.
\end{array}\right. \label{wave} 
\end{equation}
where $W$ is the {\em spatially-homogeneous Gaussian noise} considered in \cite{balan-song17,balan-chen}, whose precise definition is given in Section \ref{section-char} below.

The goal of this article is to show that the law of the solution to equation \eqref{wave} has a density with respect to the Lebesgue measure. This problem has a long history which we outline below.
Before this, we recall from \cite{balan-song17} the definition of the solution, referring to Section \ref{section-Malliavin} below for the precise definition of the Skorohod integral.
We denote by $G$ the fundamental solution of the wave equation on $\bR_{+} \times \bR$:
$$G(t,x)=\frac{1}{2}1_{\{|x| < t\}}.$$

\begin{definition}
{\rm We say that a process $u=\{u(t,x); t \geq 0,x \in \bR\}$ with
$u(0,x)=1$ for all $x \in \bR$ is a (mild Skorohod) {\em solution}
of equation \eqref{wave} if $u$ has a measurable modification
(denoted also by $u$) such that $\sup_{(t,x)\in [0,T] \times \bR}
\E|u(t,x)|^2<\infty$ for all $T>0$, and for any $t>0$ and $x \in
\bR$, the following equality holds in $L^2(\Omega)$:
\begin{equation}
\label{def-sol}
u(t,x)=1+\int_{0}^t \int_{\bR}G(t-s,x-y)u(s,y) W(\delta s, \delta y),
\end{equation}
where the stochastic integral is understood in the Skorohod sense, and the process $v^{(t,x)}=\{v^{(t,x)}(s,y)=1_{[0,t]}(s)G(t-s,x-y)u(s,y);s \geq 0,y \in \bR\}$ is Skorohod integrable.
}
\end{definition}

The existence of the solution to the wave equation with general initial condition and space-time homogeneous Gaussian noise has not been studied in the literature so far. Article \cite{balan-chen} examined the heat equation with initial condition given by a measure, and the same noise $W$ as in the present paper.

In the case of constant initial condition, it was proved in \cite{balan-song17} that the wave equation has
a unique solution in any spatial dimension $d$, provided that the
spatial spectral measure $\mu$ of the noise satisfies {\em Dalang's
condition}:
\begin{equation}
\label{Dalang-cond}
\int_{\bR^d}\frac{1}{1+|\xi|^2}\mu(d\xi)<\infty.
\end{equation}

Condition \eqref{Dalang-cond} was introduced simultaneously in
articles \cite{dalang99} and \cite{PZ00}, and plays a crucial role
in the study of stochastic partial differential equations (SPDEs)
with spatially homogeneous Gaussian noise (see for instance
\cite{dalang99, peszat02, NQ07, huang-le-nualart17} for a sample of
relevant references). Owing to Remark 10(b) in
\cite{dalang99}, note that \eqref{Dalang-cond} holds automatically
if $d=1$.

The problem of absolute continuity of the law and smoothness of the density for the solution of an SPDE goes back to article \cite{carmona-nualart88}, in which the authors studied the equation:
\begin{equation}
\label{SPDE}
Lu(t,x)=\sigma(u(t,x))\dot{W}(t,x)+b(u(t,x))
\end{equation}
with space-time white noise $W$, smooth functions $\sigma$ and $b$,
and $L$ the wave operator on $\bR_{+} \times I$, for an interval $I
\subset \bR$ which could be bounded, semi-bounded, or even $\bR$,
with Dirichlet boundary conditions when $I$ has a finite endpoint.
These authors showed that the {\em mild solution} (defined similarly
to \eqref{def-sol} using the Green function of the wave operator)
coincides with the {\em weak solution} (defined using integration
against test functions) and proved that this solution has a smooth
density. In \cite{pardoux-zhang93}, \cite{bally-pardoux98} and
\cite{mn}, it was shown that the same property
holds for equation \eqref{SPDE} in which $L$ is replaced by the heat
operator on $\bR_{+} \times [0,1]$, with Dirichlet (respectively
Neumann) boundary conditions. (In this case, the fact that the mild
and weak solutions coincide was known from Walsh' lecture notes
\cite{walsh86}.)

In the recent years, several authors revisited the problem of existence and smoothness of density for the mild solution of an SPDE of form \eqref{SPDE} on the entire space $\bR^d$, with Lipschitz functions $\sigma$ and $b$, driven by the more general Gaussian noise introduced in \cite{dalang-frangos98, dalang99}. This noise is spatially homogeneous (with spatial spectral measure $\mu$ as above), but white in time, i.e. has temporal covariance structure given formally by the Dirac distribution at $0$. The function $\sigma$ is assumed to satisfy the condition:
\begin{equation}
\label{non-deg}
|\sigma(x)| \geq c>0 \quad \mbox{for all} \ x \in \bR.
\end{equation}
This non-degeneracy condition guarantees the genuine stochastic
nature of the solution $u$, since the noise term will never vanish,
and turns out to be crucial in order to prove conditions
\eqref{norm-D-positive} and \eqref{negative-moments} below.
Nevertheless, in the case of the stochastic heat equation, condition
\eqref{non-deg} can be substantially weaken (see
\cite{mn,pardoux-zhang93}).

We refer the reader to \cite{millet-sanz99} for the wave equation in
spatial dimension $d=2$, \cite{NQ07} for the heat equation in any
dimension $d$ and the wave equation in dimension $d=1,2,3$
(see also \cite{mms,qs1,qs2}), and
\cite{sanz-sus13, sanz-sus15} for the wave equation in dimension $d\geq 4$.

In the case of the space-time Gaussian noise which is colored in
time (i.e. has temporal covariance structure given by a
non-negative-definite locally integrable function), all references
related to the problem of existence and smoothness of the density of
the law of the solution focus on the stochastic heat equation with
{\em linear} multiplicative noise (i.e. $\sigma(x)=x$ and $b=0$):
\begin{equation}
\label{SHE} \frac{\partial u}{\partial t}(t,x)=\frac{1}{2}\Delta
u(t,x)+u(t,x)\dot{W}(t,x), \quad t>0,x \in \bR^d,
\end{equation}
with initial condition $u(0,x)=u_0(x)$, where $u_0$ is a continuous
bounded function. In this case, the noise is not a martingale in
time, and the techniques of It\^o stochastic integration cannot be
used. This leads to several types of solution: the {\em mild
Skorohod solution} defined similarly to \eqref{def-sol} using the
Skorohod integral involving the Green function $G_h$ of the heat
operator, the {\em mild Stratonovich solution}  defined using a
Stratonovich-type integral of the same term as in \eqref{def-sol}
involving $G_h$, and the {\em weak solution} defined using
Stratonovich integration against test functions.  According to some
well-known criteria from Malliavin calculus, to show that $u(t,x)$
has a density it is enough to prove that
\begin{equation}
\label{norm-D-positive}
\|D u(t,x)\|_{\cH}>0 \quad \mbox{a.s.}
\end{equation}
and this density is smooth if $u(t,x)$ is infinitely differentiable
in the Malliavin sense and
\begin{equation}
\label{negative-moments} \E\|Du(t,x)\|_{\cH}^{-2p}<\infty \quad
\mbox{for all} \quad p>0.
\end{equation}
The notation $\cH$ stands for the Hilbert space associated to the space-time correlation of the noise (see Section \ref{section-char} for the precise definition).

In \cite{hu-nualart-song11}, it was proved that relation \eqref{negative-moments} holds for the weak solution of \eqref{SHE}, if the covariance functions of the noise are given by $\gamma(t)=\rho_H(t):=H(2H-1)|t|^{2H-2}$ and $f(x)=\prod_{i=1}^{d}\rho_{H_i}(x_i)$ with parameters $H,H_1, \ldots,H_d \in (1/2,1)$. This result was obtained using the Feynman-Kac (FK) representation of the weak solution, which holds only when the parameters of the noise satisfy the condition $2H+\sum_{i=1}^{d}H_i>d+1$. Under the same condition, it might be possible to prove that the mild Skorohod solution of equation \eqref{SHE} satisfies \eqref{negative-moments}, using the FK formula of this solution given in \cite{hu-nualart-song11}.
In the case of general covariance functions $\gamma$ and $f$, the authors of \cite{HHNT} established
the FK formula for the mild Stratonivich solution to equation \eqref{SHE}, assuming that $0 \leq \gamma(t) \leq C_{\beta}|t|^{-(1-\beta)}$ and
\begin{equation}
\label{eq:beta}
\int_{\bR^d}\left(\frac{1}{1+|\xi|^2} \right)^{\beta}\mu(d\xi)<\infty,
\end{equation}
for some $\beta \in (0,1)$ and $C_{\beta}>0$. Under this assumption,
it may be possible to show that the mild Stratonovich solution
satisfies relation \eqref{negative-moments}. In the recent preprint
\cite{hu-le18}, it was proved that the mild Skorohod solution
$u(t,x)$ to equation \eqref{SHE} satisfies \eqref{negative-moments}
by deriving estimates for the small ball probability ${\mathbb
P}(\|Du(t,x)\|_{\cH} \leq a)$ as $a \to 0+$. This was proved using
an FK formula for the regularization of the solution, and assuming
that the noise is white in either space or time, or $c_1
t^{\alpha_0} \leq \gamma(t) \leq c_2 t^{-\alpha_0}$ for all $t \in
\bR$, for some $c_1, c_2>0$ and $\alpha_0 \in [0,1)$, and $f$
satisfies the scaling property $f(cx) \leq c^{-\alpha} f(x)$ for all
$c>0$ and $x \in \bR^d$, for some $\alpha \in (0,2)$. In addition,
these authors assume that $\gamma=\gamma_0*\gamma_0$ and $f=f_0*f_0$
for some functions $\gamma_0$ and $f_0$.

In the present article, we will prove the absolute continuity of the
law of the solution to the wave equation \eqref{wave} with
space-time homogeneous Gaussian noise, under
Assumption A given below. For the wave equation, it is not known if
there is a FK formula for the solution. Our method for proving
\eqref{norm-D-positive} is similar to the one used in \cite{NQ07}
for the white noise in time. Since $\sigma(x)=x$ fails to satisfy
condition \eqref{non-deg}, we will prove \eqref{norm-D-positive} by
localizing on the event $\Omega_m=\{|u(t,x)|>m^{-1}\}$ and let $m
\to \infty$. Note that, compared to parabolic type equations, the
solution of the stochastic wave equation does not verify a
comparison principle and therefore can hit zero with positive
probability.

 A closer look at the inner product in $\cH$ reveals
that it is enough to prove that (see Corollary \ref{new-cor} below):
\begin{equation}
\label{norm-0-positive}
\int_0^t \int_{\bR} |D_{r,z}u(t,x)|^2dz dr>0 \quad \mbox{a.s. on} \quad \Omega_m.
\end{equation}
For this, we will show that the process $\{D_{r,z} u(t,x);r \in
[0,t],z \in \bR\}$ has a jointly measurable modification and
satisfies a certain integral equation involving a Hilbert-space
valued Skorohod integral. As shown in Sections
\ref{section-Malliavin}, \ref{section-D2} and \ref{section-eq-D} below, the proofs of these facts contain some
of the main technical difficulties encountered throughout the
article, which are mainly due to the time-space correlation of the
underlying noise.

The recent preprints \cite{huang-nualart-viitasaari,delgato-nualart-zheng} contain some estimates which are related to the present article. More precisely, Lemma 5.1 of \cite{huang-nualart-viitasaari} shows that if $u$ is the solution of the stochastic heat equation with space-time white noise, then for any $T>0$ and $(t,x) \in [0,T] \times \bR$,
\begin{equation}
\label{Nualart-estimate}
\E|D_{r,z}u(t,x)|^2 \leq C_T G^2(t-r,x-z), \quad \mbox{for all} \ (r,z) \in [0,t] \times \bR,
\end{equation}
where $G(t,x)=(2\pi t)^{-1/2}\exp(-x^2/(2t))$ is the heat kernel and $C_T>0$ is a constant depending on $T$. In the case of the wave equation in dimension $d=1$ with Gaussian noise which is white in time and fractional in space (with Hurst index $H \geq 1/2$), Lemma 2.2 in \cite{delgato-nualart-zheng} gives an estimate similar to \eqref{Nualart-estimate}, except that it holds for {\em almost all} $(r,z) \in [0,t] \times \bR$.
In the present article, we could not obtain such a nice estimate for the solution of the wave equation in dimension $d=1$, with general space-time homogeneous Gaussian noise. But we prove that $\E\|D_{r,\cdot}u(t,x)\|_{L^2(\bR)}^2$ is uniformly bounded in $r \in [0,t]$ and $(t,x) \in [0,T] \times \bR$, which implies that $D_{r,\cdot}u(t,x)$ belongs to $L^2(\bR)$ a.s. (see Theorem \ref{th-D-in-L2} below).

\medskip

We suppose that the following assumption holds, which is important for describing the space $\cH$ and its inner product, as shown by Theorem \ref{jolis-th} below.

\vspace{2mm}

\noindent {\bf Assumption A.} $\mu(d\xi)=(2\pi)^{-d}g(\xi)d\xi$, $\nu(d\tau)=(2\pi)^{-1}h(\tau)d\tau$ and $1/(hg)1_{\{hg>0\}}$ is a slow growth (or tempered) function, i.e. there exists an integer $k \geq 1$ such that
$$\int_{\{hg>0\}}\left(\frac{1}{1+\tau^2+|\xi|^2} \right)^k  \frac{1}{h(\tau)g(\xi)}d\tau d\xi<\infty.$$

\vspace{2mm}

Our basic example is $\gamma(t)=H(2H-1)|t|^{2H-2}$ with $1/2<H<1$ and $f$ is the Riesz kernel of order $\alpha$, i.e. $f(x)=|x|^{-1+\alpha}$ with $0<\alpha<1$ .
In this case, $h(\tau)=|\tau|^{2H-1}$  and $g(\xi)=c_{\alpha}|\xi|^{-\alpha}$ for some constant $c_{\alpha}>0$ depending on $\alpha$. Assumption A is satisfied for this example.

The following theorem is the main result of the present article.

\begin{theorem}
\label{thm:density} Let $u$ be the solution of equation
\eqref{wave}. If Assumption A holds, then the restriction of the law
of the random variable $u(t,x)1_{\{u(t,x) \not=0\}}$ to the set $\bR
\verb2\2 \{0\}$ is absolutely continuous with respect to the
Lebesgue measure on $\bR \verb2\2 \{0\}$.
\end{theorem}

This article is organized as follows. In Section \ref{section-char},
we describe our space-time homogeneous Gaussian noise and  we
characterize the space $\cH$ and its inner product. In Section
\ref{section-Malliavin}, we review some basic elements of Malliavin
calculus, we prove that the solution to equation \eqref{wave} is
infinitely differentiable in the Malliavin sense, and examine some
properties of its Malliavin derivative. In Section
\ref{section-D2} we examine the second order Malliavin derivative of
the solution. In Section \ref{section-eq-D}, we prove that the
Malliavin derivative satisfies a certain integral equation. In
Section \ref{sec:density}, we give the proof of Theorem
\ref{thm:density}. In Appendix A, we discuss a Parseval-type
identity, while in Appendix B we give a criterion for the existence
of a measurable modification of a random field. Both these results
are used in the present article.

\section{Characterization of the space $\cH$}
\label{section-char}

In this section, we define the space-time homogeneous Gaussian noise
$W$, we provide an alternative definition of the inner product in
$\cH$ in terms of the Fourier transform in the space and time variables, and
we give a characterization of the space $\cH$ (which is due
essentially to \cite{BGP12}). This characterization plays an
important role in the present article, because it allows us to focus
on \eqref{norm-0-positive}, instead of \eqref{norm-D-positive}. The
results in the present section are valid for any space dimension
$d\geq 1$.


The noise $W$ is given by a zero-mean Gaussian process
$\{W(\varphi);\varphi \in \cD(\bR^{d+1})\}$ defined on a complete
probability space $(\Omega,\cF,\bP)$, with covariance
$$\E[W(\varphi_1)W(\varphi_2)]=\int_{\bR^2 \times \bR^{2d}}\gamma(t-s)f(x-y)\varphi_1(t,x)\varphi_2(s,y)dxdydtds=:J(\varphi_1,\varphi_2),$$
where $\gamma:\bR \to [0,\infty]$ and $f:\bR^d \to [0,\infty]$ are continuous, symmetric, locally integrable functions, such that
$$\gamma(t) <\infty \quad \mbox{if and only if} \quad t\not=0;$$
$$f(x) <\infty \quad \mbox{if and only if} \quad x\not=0.$$
Here $\cD(\bR^{d+1})$ is the space of $C^{\infty}$-functions on $\bR^{d+1}$ with compact support.
We denote by $\cH$  the completion of $\cD(\bR^{d+1})$ with respect
to $\langle \cdot,\cdot\rangle_{\cH}$ defined by
$\langle \varphi_1,\varphi_2\rangle_{\cH}=J(\varphi_1,\varphi_2)$.

We assume that $f$ is non-negative-definite (in the sense of distributions), i.e.
$$\int_{\bR^d}(\varphi*\widetilde{\varphi})(x)f(x)dx \geq 0, \quad \mbox{for all}  \quad \varphi \in \cS(\bR^d),$$
where $\widetilde \varphi(x)=\varphi(-x)$ and $\cS(\bR^d)$ is the
space of rapidly decreasing $C^{\infty}$-functions on $\bR^d$. By
the Bochner-Schwartz theorem, there exists a tempered measure $\mu$
on $\bR^d$ such that $f=\cF \mu$, where $\cF \mu$ denotes the
Fourier transform of $\mu$ in the space $\cS_{\bC}'(\bR^d)$ of
$\bC$-valued tempered distributions on $\bR^d$. We emphasize that
this does not mean that $f(x)=\int_{\bR^d} e^{-i \xi \cdot x}
\mu(d\xi)$ for all $x \in \bR^d$, since $\mu$ may be an infinite
measure. It means that
$$\int_{\bR^d} f(x)\varphi(x)dx=\int_{\bR} \cF \varphi(\xi) \mu(d\xi) \quad \mbox{for all} \quad \varphi \in \cS_{\bC}(\bR^d),$$
where $\cS_{\bC}(\bR^d)$ is the space of $\bC$-valued rapidly decreasing $C^{\infty}$-functions on $\bR^d$, and $\cF \varphi(x)=\int_{\bR^d} e^{-i \xi \cdot x} \varphi(x)dx$ is the Fourier transform of $\varphi$. Here $\xi \cdot x$ denotes the scalar product in $\bR^d$.
Similarly, we assume that $\gamma$ is non-negative-definite (in the sense of distributions), and so there exists a tempered measure $\nu$ on $\bR^d$ such that $\gamma=\cF \nu$ in $\cS_{\bC}'(\bR^d)$.

Note that, by the definition of the measure $\mu$, for any
$\varphi_1,\varphi_2 \in \cS_{\bC}(\bR^d)$
\begin{equation}
\label{Parseval-x}
\int_{\bR^d} \int_{\bR^d}f(x-y)\varphi_1(x)\overline{\varphi_2(y)}dxdy=\int_{\bR^d}\cF \varphi_1(\xi)\overline{\cF \varphi_2(\xi)}\mu(d\xi).
\end{equation}
Similarly, for any $\phi_1,\phi_2 \in \cS_{\bC}(\bR^d)$
\begin{equation}
\label{Parseval-t}\int_{\bR} \int_{\bR}\gamma(t-s)\phi_1(t)\overline{\phi_2(s)}dtds=\int_{\bR}\cF \phi_1(\tau)\overline{\cF \phi_2(\tau)}\nu(d\tau).
\end{equation}

Building upon a remarkable result borrowed from \cite{KX09}, in
Appendix \ref{appendix-Parseval}, we show that relation
\eqref{Parseval-x} holds for any functions $\varphi_1,\varphi_2 \in
L^1_{\bC}(\bR^d)$ whose absolute values have ``finite energy'' with
respect to the kernel $f$. This fact is used in the proof of Lemma
\ref{max-principle} below.

By Lemma 2.1 of \cite{balan-song17}, for any $\varphi_1, \varphi_2
\in \cD(\bR^{d+1})$,
\begin{equation}
\label{alt-expres-J} \langle \varphi_1,\varphi_2
\rangle_{\cH}=\int_{\bR^{d+1}} \cF \varphi_1(\tau,\xi)\overline{\cF
\varphi_2(\tau,\xi)}\nu(d\tau)\mu(d\xi),
\end{equation}
where $\cF$ denotes the Fourier transform in both variables $t$ and $x$.

Similarly to \cite{jolis10, BGP12}, we consider the space:
$$\cU=\{S \in \cS'(\bR^{d+1}); \cF S \ \mbox{is a (measurable) function and} \int_{\bR^{d+1}} |\cF S(\tau,\xi)|^2 \nu(d\tau)\mu(d\xi)<\infty\}.$$
The space $\cU$ is endowed with the inner product
$$\langle S_1,S_2\rangle_\cU := \int_{\bR^{d+1}} \cF S_1(\tau,\xi) \overline{\cF S_2(\tau,\xi)} \nu(d\tau)\mu(d\xi).$$
By \eqref{alt-expres-J}, $\langle \varphi_1,\varphi_2
\rangle_{\cH}=\langle \varphi_1,\varphi_2 \rangle_{\cU}$ for any
$\varphi_1,\varphi_2 \in \cD(\bR^{d+1})$. We denote
$\|S\|_{\cU}^2=\langle S,S \rangle_{\cU}$. Note that if $S\in
\cS'(\bR^{d+1})$ is such that $\cF S$ is a function, then
\begin{equation}
\label{eq:0}
 \cF S(-\tau,-\xi)=\overline{\cF S(\tau,\xi)} \quad \mbox{for almost all} \ (\tau,\xi) \in \bR \times \bR^d.
 \end{equation}
The proof of this fact is very similar to that of Lemma 3.1 of \cite{jolis10}. This property and
the symmetry of the measures $\nu$ and $\mu$ imply that
$\langle\cdot,\cdot\rangle_\cU$ is a well-defined $\bR$-valued inner
product, provided that we identify two elements $S_1$ and $S_2$ such
that $\|S_1-S_2\|_\cU=0$.

We let $L^2_\bC (\nu\times
\mu)$ be the space of functions $v:\bR^{d+1}\rightarrow \bC$
such that
$$\int_{\bR^{d+1}} |v(\tau,\xi)|^2 \nu(d\tau)\mu(d\xi)<\infty,$$
and $\widetilde{L}^2_{\bC}(\nu \times \mu)$ be the subset of $L^2_\bC (\nu\times
\mu)$ consisting of functions $v$ such that
$v(-\tau,-\xi)=\overline{v(\tau,\xi)}$ for almost all $(\tau,\xi) \in \bR \times \bR^d$.

The next result generalizes Theorem 3.2 of \cite{jolis10} to higher
dimensions. Assumption A is not needed for this result.

\begin{theorem}
\label{th-D-dense}
The space $\cD(\bR^{d+1})$ is dense in $\cU$, with respect to
 $\langle \cdot, \cdot \rangle_{\cU}$,
and hence $\cU$ is included in $\cH$. Moreover, $\langle
S_1,S_2\rangle_\cU=\langle S_1,S_2\rangle_\cH$ for any $S_1,S_2 \in
\cU$.
\end{theorem}

\begin{proof} We only need to prove that $\cD(\bR^{d+1})$ is dense in $\cU$,
with respect to $\langle \cdot, \cdot \rangle_{\cU}$, i.e. for any
$S \in \cU$, there exists a sequence $(\varphi_n)_{n\geq 1}\subset
\cD(\bR^{d+1})$ such that $\|\varphi_n-S\|_\cU\rightarrow 0$ as $n
\to \infty$. For this, it suffices to prove that
\begin{equation}
\label{F-dense-tilde}
\cF\big(\cD(\bR^{d+1})\big) \quad \mbox{is dense in} \ \widetilde{L}^2_{\bC}(\nu \times \mu),
\end{equation}
where $\cF\big(\cD(\bR^{d+1})\big)$ is the image of $\cD(\bR^{d+1})$ under the Fourier transform, and $\widetilde{L}^2_{\bC}(\nu \times \mu)$
is endowed with the topology of $L^2_\bC(\nu\times \mu)$.
(To see this, let $S\in \cU$ be arbitrary. Since $\cF S\in \widetilde{L}^2_{\bC}(\nu \times \mu)$, there exists a sequence $(\varphi_n)_{n \geq 1} \subset \cD(\bR^{d+1})$ such that
$\|\cF \varphi_n-\cF S\|_{L_{\bC}^2(\nu \times \mu)}=\|\varphi_n-S\|_{\cU} \to 0$ as $n \to \infty$.)

It remains prove \eqref{F-dense-tilde}. First, we claim that
\begin{equation}
\label{F-dense}
\cF\big(\cD_\bC(\bR^{d+1})\big) \quad \mbox{is dense in} \ L^2_\bC(\nu\times \mu).
\end{equation}
Indeed, this is an extension of Theorem 4.1 in \cite{Ito}, which is proved as follows.
First, $\cF\big(\cD_\bC(\bR^{d+1})\big)$ is dense in
$\cS_\bC(\bR^{d+1})$, because the Fourier transform defines an
homeomorphism from $\cS_\bC(\bR^{d+1})$ onto itself. Secondly, using
that $\nu$ and $\mu$ are tempered measures, one obtains that
$\cS_\bC(\bR^{d+1})\subset L^2_\bC(\nu\times \mu)$ and that
convergence in $\cS_\bC(\bR^{d+1})$ implies convergence in
$L^2_\bC(\nu\times \mu)$. These two facts imply the third one,
namely $\cF\big(\cD_\bC(\bR^{d+1})\big)$ is dense in
$\cS_\bC(\bR^{d+1})$ with respect to the topology of
$L^2_\bC(\nu\times \mu)$. Finally, the conclusion follows by
observing that $\cS_\bC(\bR^{d+1})$ is dense in $L^2_\bC(\nu\times
\mu)$, since $\cD_\bC(\bR^{d+1})$ is so. This proves \eqref{F-dense}.

We now prove \eqref{F-dense-tilde}. Let $v \in \widetilde{L}^2_{\bC}(\nu \times \mu)$ be arbitrary.
By \eqref{F-dense}, there exists a sequence $(\varphi_n)_{n\geq 1}$ in $\cD_\bC(\bR^{d+1})$ such that
$\cF \varphi_n \rightarrow v$ in $L^2_\bC(\nu\times \mu)$. Note that
this implies that the following limits hold in $L^2(\nu\times \mu)$:
\[
 \lim_{n\rightarrow \infty} {\rm Re}(\cF \varphi_n)={\rm Re}(v) \quad \text{and} \quad
 \lim_{n\rightarrow \infty} {\rm Im}(\cF \varphi_n)={\rm Im}(v).
\]
Let us construct a sequence $(\psi_n)_{n\geq 1}$  in
$\cD(\bR^{d+1})$ such that $\cF \psi_n \rightarrow v$ in
$L^2_\bC(\nu\times \mu)$. For this, we will use the following
notation. Namely, for any $\bC$-valued function $\kappa$, ${\rm
e}(\kappa):=\frac12 (\kappa+\widetilde{\kappa})$ and ${\rm
o}(\kappa):=\frac12 (\kappa-\widetilde{\kappa})$ denote the even and
odd parts of $\kappa$, respectively.

Since $v$ belongs to $\widetilde{L}^2_{\bC}(\nu \times \mu)$, we have that ${\rm Re}(v)$ is a even
function and ${\rm Im}(v)$ is an odd function. Here, we say that a function $g$ is {\em even} if $g(-x)=g(x)$ for almost all $x \in \bR^{d+1}$ and is {\em odd} if $g(-x)=-g(x)$ for almost all $x \in \bR^{d+1}$.
These properties,
together with the symmetry of the measures $\nu$ and $\mu$, imply
that the following limits hold in $L^2(\nu\times \mu)$:
\[
 \lim_{n\rightarrow \infty} {\rm e}({\rm Re}(\cF \varphi_n))={\rm Re}(v)
 \quad \text{and} \quad  \lim_{n\rightarrow \infty} {\rm o}({\rm Im}(\cF \varphi_n))={\rm
 Im}(v).
\]
Define $\psi_n:=\cF^{-1} \big( {\rm e}({\rm Re}(\cF \varphi_n)) + i
{\rm o}({\rm Im}(\cF \varphi_n))\big)$. Then, $\psi_n$ is a real
function, belongs to $\cD(\bR^{d+1})$ and, by construction, $\cF
\psi_n \rightarrow v$ in $L^2_\bC(\nu\times \mu)$.
\end{proof}

The following result not only generalizes Theorem 3.4 of \cite{jolis10} to higher dimensions, but also specifies the necessary and sufficient condition for the completeness of $\cU$. This result is an immediate consequence of Theorem 3.5.(2) of \cite{BGP12}, applied to the space $\bR^{d+1}$ and the measure $F=\nu \times \mu$.

\begin{theorem}
\label{th-3-5-BGP}
Assume that $\mu(\xi)=(2\pi)^{-d}g(\xi)d\xi$ and $\nu(d\tau)=(2\pi)^{-1}h(\tau)d\tau$. Then
$\cU$ is complete if and only if for any function $\varphi \in L_{\bC}^2(\nu \times \mu)$ there exists an integer $k \geq 1$ such that
$$\int_{\{h>0,g>0\}} \left(\frac{1}{1+\tau^2+|\xi|^2} \right)^k  |\varphi(\tau,\xi)| \, d\tau d\xi<\infty.$$
In particular, $\cU$ is a complete if $1/(hg) 1_{\{hg>0\}}$ is a slow growth function.
\end{theorem}

Combining Theorems \ref{th-D-dense} and \ref{th-3-5-BGP}, we obtain the following result, which can be viewed as a generalization of Theorem 3.5 of \cite{jolis10} to higher dimensions.

\begin{theorem}
\label{jolis-th}
If Assumption A holds, then $\cH$ coincides with the space $\cU$.
Moreover, for any $S_1,S_2 \in \cH$,
$$\langle S_1, S_2\rangle_{\cH}=\int_{\bR^{d+1}}
\cF S_1(\tau,\xi) \overline{\cF S_2(\tau,\xi)}\nu(d\tau)\mu(d\xi).$$
\end{theorem}

\begin{corollary}
\label{new-cor}
If Assumption A holds and $S$ is a measurable function on
$\bR_+\times \bR$ such that $S\in \cH$ and
\begin{equation}
\label{eq-S}
\int_0^\infty \int_{\bR^d} |S(t,x)|^2 dx dt >0,
\end{equation}
then $\|S\|_\cH >0$.
\end{corollary}

\begin{proof}
Suppose that $\|S\|_{\cH}=0$.
By Theorem \ref{jolis-th}, $\cF S(\tau,\xi)=0$ for almost all $(\tau,\xi) \in \bR \times \bR^d$. Hence $S(t,x)=0$ for almost all $t>0$ and $x \in \bR^d$, which contradicts \eqref{eq-S}.
\end{proof}

Despite its complicated nature, the space $\cH$ contains a nice space of functions, called $|\cH|$, which is defined as the set of all measurable functions $\varphi:\bR_{+} \times \bR^d \to \bR$ such that
$$\|\varphi\|_{|\cH|}^2:=\int_{(\bR_+ \times \bR^d)^2} |\varphi(t,x)||\varphi(s,y)|\gamma(t-s)f(x-y)dtdx dsdy<\infty.$$
Note that $|\cH|$ is a Banach space equipped with the norm $\|\cdot\|_{|\cH|}$, and
$\|\varphi\|_{\cH} \leq \|\varphi\|_{|\cH|}$ for any $\varphi \in |\cH|$. By Lemma \ref{L1-lemma} (Appendix \ref{appendix-Parseval}), for any integrable functions $\varphi,\psi \in |\cH|$,
$$\langle \varphi,\psi \rangle_{\cH}=\int_{(\bR_+ \times \bR^d)^2}\varphi(t,x)\psi(s,y)\gamma(t-s)f(x-y)dtdx dsdy.$$

\section{Malliavin derivative of the solution}
\label{section-Malliavin}

In this section, we examine some properties of the Malliavin derivative $Du(t,x)$ of the solution $u(t,x)$ to equation \eqref{wave}.

We recall some basic elements of Malliavin
calculus. We refer the reader to \cite{nualart06} for more details.
Any random variable $F \in L^2(\Omega)$ which is measurable with
respect to the $\sigma$-field generated by $\{W(\varphi);\varphi \in
\cH\}$ admits the representation $F=\sum_{n\geq 0}J_n F$ where $J_n
F$ is the projection of $F$ on the $n$-th Wiener chaos space $\cH_n$
for $n\geq 1$, and $J_0 F=\E(F)$. We denote by $I_n$ the multiple
integral of order $n$ with respect to $W$, which is a linear
continuous operator from $\cH^{\otimes n}$ onto $\cH_n$.

Let $\cS$ be the class of ``smooth'' random variables, i.e. random variables of the form
$F=f(W(\varphi_1),\ldots,W(\varphi_n))$, where $n\geq 1$, $\varphi_1, \ldots,\varphi_n \in \cH$ and $f$ is in the set  $C_b^{\infty}(\bR^n)$ of bounded infinitely differentiable functions on $\bR^n$ whose partial derivatives are bounded. The Malliavin derivative of a random variable $F$ of this form is the $\cH$-valued random variable given by:
$$DF=\sum_{i=1}^{n}\frac{\partial f}{\partial x_i}(W(\varphi_1),\ldots,W(\varphi_n))\varphi_i.$$
We let $\bD^{1,2}$ be the completion of $\cS$ with respect to the
norm $\|F\|_{1,2}=\big(\E|F|^2+\E\|DF\|_{\cH}^2\big)^{1/2}$.

Similarly, the iterated derivative $D^k F$ can be defined as a $\cH^{\otimes k}$-valued random variable, for any natural number $k \geq 1$. For any $p>1$, let $\bD^{k,p}$ be the completion of $\cS$ with respect to the norm
$$\|F\|_{k,p}=\Big(\E|F|^p+\sum_{j=1}^{k}\E\|D^jF\|_{\cH^{\otimes j}}^p\Big)^{1/p}.$$

It can be proved that: (see p. 28 of \cite{nualart06})
$$F \in \bD^{k,2} \quad \mbox{if and only if} \quad \sum_{n\geq 1} n^k \E|J_n F|^2<\infty.$$
For $p>1$, it can be shown that
\begin{equation}
\label{F-in-D} F \in \bD^{k,p} \quad \mbox{if} \quad \sum_{n\geq 1}
n^{k/2}(p-1)^{n/2} (\E|J_n F|^2)^{1/2}<\infty
\end{equation}
(see p. 28 of \cite{balan12}).

The divergence operator $\delta$ is defined as the adjoint of the
operator $D$. The domain of $\delta$, denoted by $\mbox{Dom} \
\delta$, is the set of $u \in L^2(\Omega;\cH )$ such that
\begin{equation}
\label{def-delta} |\E \langle DF,u \rangle_{\cH}| \leq c
(\E|F|^2)^{1/2}, \quad \forall F \in \bD^{1,2},
\end{equation}
where $c$ is a constant depending on $u$. If $u \in {\rm Dom} \
\delta$, then $\delta(u)$ is the element of $L^2(\Omega)$
characterized by the following duality relation:
\begin{equation}
\label{duality} \E(F \delta(u))=\E\langle DF,u \rangle_{\cH}, \quad
\forall F \in \bD^{1,2}.
 \end{equation}

If $u \in \mbox{Dom} \ \delta$, we will use the notation
$$\delta(u)=\int_0^{\infty} \int_{\bR}u(t,x) W(\delta t, \delta x),$$
even if $u$ is not a function in $(t,x)$, and we say that
$\delta(u)$ is the Skorohod integral of $u$ with respect to $W$.

We return now to the solution $u$ of equation \eqref{wave}.
From \cite{balan-song17}, we know that $u(t,x)$ has the Wiener chaos expansion
$$u(t,x)=1+ \sum_{n\geq 1}I_n(f_n(\cdot,t,x)),$$ where
$I_n$ is the multiple Wiener integral of order $n$ with respect to $W$, and
$$f_n(t_1,x_1,\ldots,t_n,x_n,t,x)=G(t-t_n,x-x_n)\cdots G(t_2-t_1,x_2-x_1) 1_{\{0<t_1<\ldots<t_n<t\}}.$$
It follows that
$$\E|u(t,x)|^2 =\sum_{n\geq 0}\frac{1}{n!}\alpha_n(t),$$
where $\alpha_n(t)=(n!)^2\|\widetilde{f}_n(\cdot,t,x)\|_{\cH^{\otimes n}}^{2}$ and $\widetilde{f}_n(\cdot,t,x)$ is the symmetrization of $f_n(\cdot,t,x)$.

In Section 6 of \cite{balan12}, it was proved that if $f(x)=|x|^{-\alpha}$ with $\alpha \in (0,d)$ and $\gamma(t)=H(2H-1)|t|^{2H-2}$ with $1/2<H<1$, then for any integer $k \geq 1$ and for any $p>1$,
$$u(t,x) \in \bD^{k,p}.$$ We will now extend this result to general functions $f$ and $\gamma$
(in the case $d=1$).

We begin with a maximum principle, which is a refinement of Lemma 4.2 of \cite{balan-song17}, specific to the cases $d=1$ and $d=2$. Its proof is based on a Parseval-type identity given in Appendix \ref{appendix-Parseval}.

\begin{lemma}
\label{max-principle}
If $G$ is the fundamental solution of the wave equation in spatial dimension $d=1$ or $d=2$, then
$$\sup_{\eta\in\bR^d}\int_{\bR^d}|\cF G(t,\cdot)(\xi+\eta)|^2 \mu(d\xi) =  \int_{\bR^d}|\cF G(t,\cdot)(\xi)|^2\mu(d\xi).$$
\end{lemma}

\begin{proof} We denote by $\cE_{f}(\varphi)$ the energy of $\varphi$ with respect to $f$ given by
\eqref{identity1} (Appendix \ref{appendix-Parseval}).
Note that $G(t,\cdot)$ is a non-negative integrable function with $\int_{\bR^d}G(t,x)dx=t$ and
$$\cE_{f}(G(t,\cdot))=\int_{\bR^d}\big(G(t,\cdot)* G(t,\cdot) \big)(x)f(x)dx=\int_{\bR^d}|\cF G(t,\cdot)(\xi)|^2 \mu(d\xi)<\infty.$$
 For any $\eta \in \bR^d$, we let $G_\eta(t,x)=e^{-ix\cdot \eta}G(t, x)$. Then $G_{\eta}(t,\cdot) \in L_{\bC}^1(\bR^d)$ and
\begin{align*}
\cE_{f}(|G_{\eta}(t,\cdot)|)&=\int_{\bR^d}\big(|G_{\eta}(t,\cdot)|*|G_{\eta}(t,\cdot)| \big)(x)f(x)dx\\
& \leq \int_{\bR^d}\big(G(t,\cdot)* G(t,\cdot) \big)(x)f(x)dx=\cE_{f}(G(t,\cdot))<\infty.
\end{align*}
Moreover, $\cF G_{\eta}(t,\cdot)(\xi)=\cF G(t,\cdot)(\xi+\eta)$ for any $\xi \in \bR^d$. It follows that
\begin{align*}
0& \leq \int_{\bR^d}|\cF G(t,\cdot)(\xi+\eta)|^2 \mu(d\xi) =  \int_{\bR^d} |\cF G_{\eta}(t,\cdot)(\xi)|^2\mu(d\xi)\\
&= \int_{\bR^d}  (G_\eta(t,\cdot)*G_\eta(t,\cdot))(x) f(x) dx=\left|\int_{\bR^d}  (G_\eta(t,\cdot)*G_\eta(t,\cdot))(x) f(x) dx\right|\\
& \le  \int_{\bR^d}  \big(|G_{\eta}(t,\cdot)|*|G_{\eta}(t,\cdot)|\big)(x) f(x) dx \leq \cE_{f}(G(t,\cdot)),
\end{align*}
where for the second equality above, we used Lemma \ref{L1C-lemma} (Appendix \ref{appendix-Parseval}).
\end{proof}

The next result gives a stronger form of relation (4.10) of \cite{balan-song17} (with a simplified proof). Its proof is based on Lemma \ref{max-principle}.

\begin{lemma}
\label{sum-alpha-finite}
For any $k \geq 0$,
$$\sum_{n \geq 0}\frac{n^k}{n!}\alpha_n(t)<\infty.$$
\end{lemma}

\begin{proof}
We will borrow notations from \cite[Theorem 4.4]{balan-song17}. Recall that
\begin{equation}
\label{def-alpha-n}
\alpha_n(t)=\int_{[0,t]^{2n}}\prod_{j=1}^{n}\gamma(t_j-s_j)\psi_n({\bf t},{\bf s})d{\bf t}d{\bf s},
\end{equation}
where ${\bf t}=(t_1,\dots,t_n)$ and ${\bf s}=(s_1,\dots,s_n)$,
\begin{eqnarray*}
\psi_n({\bf t},{\bf s})&=&\int_{\bR^{n}}\cF g_{\bf t}^{(n)}(\cdot,t,x)(\xi_1,\ldots,\xi_n)\overline{\cF g_{\bf s}^{(n)}(\cdot,t,x)(\xi_1,\ldots,\xi_n)}\mu(d\xi_1)\ldots \mu(d\xi_n)
\end{eqnarray*}
and
\begin{equation}
\label{gtn}
g_{\bf t}^{(n)}(\cdot,t,x)=n!\widetilde{f}_n(t_1,\cdot,\ldots,t_n,\cdot,t,x).
\end{equation}

If the permutation $\rho$ of $\{1, \ldots,n\}$ is chosen such that
$t_{\rho(1)}<\ldots<t_{\rho(n)}$,
then
\begin{align}
\nonumber
& \cF g_{\bf t}^{(n)}(\cdot,t,x)(\xi_1,\ldots,\xi_n) =e^{-i \sum_{j=1}^{n}\xi_j\cdot x} \overline{\cF G(t_{\rho(2)}-t_{\rho(1)},\cdot)(\xi_{\rho(1)})} \\
\label{Fourier-gn}
& \quad \quad \quad \overline{\cF G(t_{\rho(3)}-t_{\rho(2)},\cdot)(\xi_{\rho(1)}+\xi_{\rho(2)})}
\ldots \overline{\cF G(t-t_{\rho(n)},\cdot)(\xi_{\rho(1)}+\ldots+\xi_{\rho(n)})}.
\end{align}

By relation (4.15) of \cite{balan-song17},
$$\alpha_n(t) \leq  \Gamma_t^n \int_{[0,t]^n}\psi_n({\bf t},{\bf t})d{\bf t},$$
where $\Gamma_t:=2\int_0^t \gamma(s)ds$.

To estimate $\psi_n({\bf t},{\bf t})$ we use the first inequality in (4.16) of \cite{balan-song17}.  We denote $u_j=t_{\rho(j+1)}-t_{\rho(j)}$ for $j=1,\ldots,n$ and $t_{\rho(n+1)}=t$. We have:
\begin{eqnarray*}
\psi_n({\bf t},{\bf t}) & \leq &   \prod_{j=1}^{n}\left(\sup_{\eta \in \bR} \int_{\bR} |\cF G(u_j,\cdot)(\xi_j+\eta)|^2 \mu(d\xi_j)\right)\\
& \leq &  \prod_{j=1}^{n} \int_{\bR} \frac{\sin^2(u_j|\xi_j|)}{|\xi_j|^2} \mu(d\xi_j),
\end{eqnarray*}
where for the last equality we used Lemma \ref{max-principle}.
Therefore,
$$\alpha_n(t) \leq \Gamma_t^n n! \int_{\{0<t_1<\ldots<t_n<t\}} \int_{\bR^{n}} \frac{\sin^2((t_2-t_1)|\xi_1|)}{|\xi_1|^2}\ldots \frac{\sin^2((t-t_n)|\xi_n|)}{|\xi_n|^2}\mu(d\xi_1)
\ldots \mu(d\xi_n)d{\bf t}.$$

From Lemma 2.2 of \cite{balan-song17b}, we know that for any $\beta>0$ and $\xi \in \bR$,
$$I_{\beta}^w(\xi):=\int_{0}^{\infty}e^{-\beta t}\frac{\sin^2(t|\xi|)}{|\xi|^2} dt=\frac{2}{\beta} \cdot \frac{1}{\beta^2+4|\xi|^2}.$$
Using the change of variables $t_2-t_1=u_1, \ldots t-t_n=u_n$, we obtain:
\begin{align*}
& \int_{\{0<t_1<\ldots<t_n<t\}} \frac{\sin^2((t_2-t_1)|\xi_1|)}{|\xi_1|^2}\ldots \frac{\sin^2((t-t_n)|\xi_n|)}{|\xi_n|^2}dt_1 \ldots dt_n\\
& \quad  =\int_{\{(u_1,\ldots,u_n)\in [0,t]^n;\sum_{j=1}^{n}u_j<t\}} e^{-Mu_1}\frac{\sin^2(u_1|\xi_1|)}{|\xi_1|^2}\ldots e^{-Mu_n}\frac{\sin^2(u_n|\xi_n|)}{|\xi_n|^2}e^{M(u_1+\ldots+u_n)}du_1 \ldots du_n\\
&\quad  \leq e^{Mt} \prod_{j=1}^{n}\left(\int_0^{t} e^{-Mu_j} \frac{\sin^2(u_j |\xi_j|)}{|\xi_j|^2} du_j\right) \leq e^{Mt} \prod_{j=1}^{n}I_{M}^w(\xi_j)=e^{Mt} \left( \frac{2}{M}\right)^n \prod_{j=1}^{n}\frac{1}{M^2+4|\xi_j|^2}.
\end{align*}
It follows that
\begin{equation}
\label{alpha-bound} \alpha_n(t) \le e^{Mt}  n!
\left(\frac{2\Gamma_t}{M}\right)^n \left( \int_{\bR}
\frac{1}{M^2+4|\xi|^2}\mu(d\xi)  \right)^n= e^{Mt}  n!
\left(\frac{2\Gamma_t}{M} K_M\right)^n,
\end{equation}
where $K_M= \int_{\bR} \frac{1}{M^2+4|\xi|^2}\mu(d\xi)$. By the dominated convergence theorem and \eqref{Dalang-cond}, $K_M \to 0$ as $M \to \infty$. Hence, using the fact that $n \leq e^n$ for any $n \geq 0$, we have:
$$\sum_{n \geq 0}\frac{n^k}{n!}\alpha_n(t)\le e^{Mt} \sum_{n \geq 0} n^k \left(\frac{2\Gamma_t}{M} K_M\right)^n\leq e^{Mt} \sum_{n \geq 0} \left( e^k \frac{2\Gamma_t}{M} K_M\right)^n<\infty,$$
if we choose $M$ sufficiently large.
\end{proof}

\medskip

\begin{lemma}
\label{lemma-u-D}
Let $u$ be the solution of equation \eqref{wave} and fix $(t,x)\in \bR_+\times \bR$. Then $u(t,x) \in \bD^{k,p}$ for any $k \geq 1$ and $p>1$.
\end{lemma}

\begin{proof}
We apply \eqref{F-in-D} to the variable $F=u(t,x)$. Let
$J_n(t,x):=J_n u(t,x)=I_n(f_n(\cdot,t,x))$. Then $\E|J_n
(t,x)|^2=n! \|\widetilde{f}_n(\cdot,t,x)\|_{\cH^{\otimes
n}}^2=\frac{1}{n!} \alpha_n(t)$. By \eqref{alpha-bound}, we have:
$$\sum_{n\geq 1} n^{k/2} (p-1)^{n/2} \left( \frac{1}{n!}\alpha_n(t)\right)^{1/2} \leq e^{Mt/2} \sum_{n\geq 1}
e^{nk/2} (p-1)^{n/2} \left( \frac{2\Gamma_t}{M}K_M\right)^{n/2}<\infty,$$
if we choose $M$ large enough.
\end{proof}

\begin{remark}
{\rm Lemmas \ref{sum-alpha-finite} and \ref{lemma-u-D} remain valid in the case $d=2$.}
\end{remark}

We begin now to examine the Malliavin derivative of the variable $u(t,x)$. Recall that the solution $u$ has the chaos expansion: $$u(t,x)=1+\sum_{n\geq 1}I_n(f_n(\cdot,t,x))=1+\sum_{n\geq 1}I_n(\widetilde{f}_n(\cdot,t,x)),$$
where
$\widetilde{f}_n(\cdot,t,x)$ is the symmetrization of $f_n(\cdot,t,x)$. By Proposition 2.4 of \cite{balan12}, the Malliavin derivative of $u(t,x)$ has the chaos expansion:
\begin{equation}
\label{series-D}
Du(t,x)=\sum_{n\geq 1}nI_{n-1}(\widetilde{f}_n(\cdot,*,t,x)),
\end{equation}
where $*$ denotes the missing $(r,z)$-variable and the convergence of the series is in $L^2(\Omega;\cH)$.
We emphasize the presence of the {\em symmetrized} kernel $\widetilde{f}_n(\cdot,t,x)$ in this relation.

We now examine the series $\sum_{n\geq 1}nI_{n-1}(\widetilde{f}_n(\cdot,r,z,t,x))$ for fixed $(r,z)\in [0,t] \times \bR$. Note that $\widetilde{f}_n(\cdot,r,z,t,x)$ is the sum of $n$ terms corresponding to the possible positions of $(r,z)$ among the $n$ variables of $\widetilde{f}_n(\cdot,t,x)$:
\begin{equation}
\label{tilde-fn-sum}
\widetilde{f}_n(\cdot,r,z,t,x)=\frac{1}{n}\sum_{j=1}^{n}h_j^{(n)}(\cdot,r,z,t,x),
\end{equation}
where $h_j^{(n)}(\cdot,r,z,t,x) \in \cH^{\otimes (n-1)}$ is a symmetric function given by:
\begin{align}
\label{def-h-jn}
& h_j^{(n)}(t_1,x_1,\ldots,t_{n-1},x_{n-1},r,z,t,x)=\\
\nonumber
 & \frac{1}{(n-1)!}\sum_{\rho \in S_{n-1}}f_{n}(t_{\rho(1)},x_{\rho(1)}, \ldots, t_{\rho(j-1)},x_{\rho(j-1)},r,z,t_{\rho(j)},x_{\rho(j)},\ldots,t_{\rho(n-1)},x_{\rho(n-1)},t,x).
\end{align}
Clearly, $h_{j}^{(n)}(\cdot,r,z,t,x)$ is the symmetrization of the function $f_j^{(n)}(\cdot,r,z,t,x)$ given by:
\begin{align}
\label{def-f-jn}
& f_{j}^{(n)}(t_1,x_1,\ldots,t_{n-1},x_{n-1},r,z,t,x) \\
\nonumber
& \quad =f_n(t_1,x_1,\ldots,t_{j-1},x_{j-1},r,z,t_j,x_j,\ldots,
t_{n-1},x_{n-1},t,x) \\
\nonumber
& \quad =G(t-t_{n-1},x-x_{n-1}) \ldots G(t_j-r,x_j-z)G(r-t_{j-1},z-x_{j-1}) \ldots G(t_2-t_1,x_2-x_1) \\
\nonumber
& \quad \quad \quad 1_{\{0<t_1<\ldots<t_{j-1}<r<t_j<\ldots<t_n<t\}}.
\end{align}

We have the following result, whose proof uses in
an essential way the form of the fundamental solution $G$ of the
wave equation for $d=1$. In its proof, we will use the space
$\cP_0$, which is defined as the completion of $\cD(\bR)$ with
respect to the inner product
\begin{equation}
\label{def-P0-product} \langle \varphi,\psi\rangle_0 =
\int_{\bR}\int_{\bR} \varphi(x)\psi(y) f(x-y)dx dy.
\end{equation}
Similarly to Theorem \ref{jolis-th}, it can be proved that $\cP_0$
coincides with the space of distributions $S \in \cS'(\bR)$ whose
Fourier transform $\cF S$ is a measurable function such that
$\|S\|_0^2:=\int_{\bR}|\cF S(\xi)|^2 \mu(d\xi)<\infty$. But $\cP_0$ includes
the space $|\cP_0|$ of all measurable functions $\varphi: \bR \to
\bR$ such that $\|\varphi\|_{+}:=\int_{\bR}\int_{\bR}
|\varphi(x)||\varphi(y)|f(x-y) dx dy<\infty$. Note that
$\|\varphi\|_{0} \leq \|\varphi\|_{+}$ for all $ \varphi \in	
|\cP_0|$. By Lemma \ref{L1-lemma} (Appendix
\ref{appendix-Parseval}), for any $\varphi,\psi \in L^1(\bR) \cap
|\cP_0|$, 
 $$\langle \varphi,\psi\rangle_0=\int_{\bR} \int_{\bR} \varphi(x)\psi(y)f(x-y)dxdy.$$

\begin{lemma}
\label{An-series}
For any $(r,z) \in [0,t] \times \bR$, $\widetilde{f}_n(\cdot,r,z,t,x) \in \cH^{\otimes (n-1)}$. Moreover, the series
$$\sum_{n\geq 1}nI_{n-1}(\widetilde{f}_n(\cdot,r,z,t,x)) \quad \mbox{converges in $L^2(\Omega)$},$$
uniformly in $(r,z) \in [0,t] \times \bR$.
\end{lemma}

\begin{proof}
Fix $(r,z) \in [0,t] \times \bR$. Due to \eqref{tilde-fn-sum}, to prove that  $\widetilde{f}_n(\cdot,r,z,t,x) \in \cH^{\otimes (n-1)}$, it is enough to
 show that $f_j^{(n)}(\cdot,r,z,t,x) \in \cH^{\otimes (n-1)}$ for any $j=1,\ldots,n$.
If $j=n$, this is clear since $f_n^{(n)}(\cdot,r,z,t,x)=G(t-r,x-z)f_{n-1}(\cdot,r,z)$. So we let $j \leq n-1$.
Since $f_j^{(n)}(\cdot,r,z,t,x) \geq 0$, and similar to the proof of the fact that $|\cH| \subset \cH$, it is enough to show that the following integral is finite:
\begin{align*}
& \|f_j^{(n)}(\cdot,r,z,t,x)\|_{\cH^{\otimes (n-1)}}^2=\int_{([0,t] \times \bR)^{2(n-1)}}\prod_{i=1}^{n-1}\gamma(t_k-s_k)\prod_{k=1}^{n-1}f(x_k-y_k)\\
& \quad \quad f_j^{(n)}(t_1,x_1,\ldots,t_{n-1},x_{n-1},r,z,t,x)
f_j^{(n)}(s_1,y_1,\ldots,s_{n-1},y_{n-1},r,z,t,x) d{\bf x} d{\bf y} d{\bf t} d{\bf s},
\end{align*}
where here ${\bf x}=(x_1,\dots,x_{n-1})$, ${\bf t}=(t_1,\dots,t_{n-1})$ and similarly for ${\bf y}$ and ${\bf s}$.
The problem in this integral is caused by the terms $G(t_j-r,x_j-z)$ and $G(s_j-r,y_j-z)$ for which there is no corresponding integral $dr dz$.
Due to the form of $G$ in dimension $d=1$, we can bound these terms by $1/2$. The remaining terms can be
separated into two integrals, one on $([0,r] \times \bR)^{2(j-1)}$ and the other on $([r,t] \times \bR)^{2(n-j)}$.
The second integral can be written as an integral on $([0,t-r] \times \bR)^{2(n-j)}$, using a change of variables. It follows that
$$ \|f_j^{(n)}(\cdot,r,z,t,x)\|_{\cH^{\otimes (n-1)}}^2 \leq \frac{1}{4}\|f_j(\cdot,r,z)\|_{\cH^{\otimes (j-1)}}^2 \|f_{n-j}(\cdot,t-r,x) \|_{\cH^{\otimes (n-j)}}^{2}<\infty.$$

Next, we treat the summability of the sum. We denote
\begin{equation}
\label{def-An}A_n(r,z,t,x):=nI_{n-1}(\widetilde{f}_n(\cdot,r,z,t,x))=
\sum_{j=1}^{n}I_{n-1}(h_j^{(n)}(\cdot,r,z,t,x)).
\end{equation}
Using the inequality $(\sum_{j=1}^n a_j)^2 \leq n \sum_{j=1}^n a_j^2$, we have:
\begin{align}
\nonumber
\E|A_n(r,z,t,x)|^2 & \leq n \sum_{j=1}^{n} \E|I_{n-1}(h_j^{(n)}(\cdot,r,z,t,x))|^2\\
& = n \sum_{j=1}^{n}(n-1)! \|h_j^{(n)}(\cdot,r,z,t,x))\|_{\cH^{\otimes (n-1)}}^2 \nonumber \\
\label{bound-An}
&=n \sum_{j=1}^{n}\frac{1}{(n-1)!} \|(n-1)!h_j^{(n)}(\cdot,r,z,t,x))\|_{\cH^{\otimes (n-1)}}^2.
\end{align}
To evaluate $\|(n-1)!h_j^{(n)}(\cdot,r,z,t,x))\|_{\cH^{\otimes (n-1)}}^2$, we proceed as in Lemma \ref{sum-alpha-finite}. For ${\bf t}=(t_1,\ldots,t_{n-1}) \in [0,t]^{n-1}$, we denote
$$g_{{\bf t},j,r,z}^{(n)}(x_1,\ldots,x_{n-1},t,x)=(n-1)!h_j^{(n)}(t_1,x_1,\dots,t_{n-1},x_{n-1},r,z,t,x).$$
Then
\begin{align*}
\|(n-1)!h_j^{(n)}(\cdot,r,z,t,x))\|_{\cH^{\otimes (n-1)}}^2 &=\int_{[0,t]^{2(n-1)}} \prod_{k=1}^{n-1} \gamma(t_k-s_k) \psi_{j,r,z}^{(n)}({\bf t},{\bf s}) d{\bf t} d{\bf s},
\end{align*}
where $\psi_{j,r,z}^{(n)}({\bf t},{\bf s})=\langle g_{{\bf t},j,r,z}^{(n)}(\cdot,t,x), g_{{\bf s},j,r,z}^{(n)}(\cdot,t,x) \rangle_{\cP_0^{\otimes (n-1)}}$. Using Cauchy-Schwarz inequality and Lemma 4.3 in \cite{balan-song17}, it follows that
\begin{align*}
& \|(n-1)!h_j^{(n)}(\cdot,r,z,t,x))\|_{\cH^{\otimes (n-1)}}^2 \leq \Gamma_t^{n-1} \int_{[0,t]^{n-1}} \|g_{{\bf t},j,r,z}^{(n)}(\cdot,t,x)\|_{\cP_0^{\otimes (n-1)}}^2 d{\bf t}\\
& \quad =\Gamma_t^{n-1} (n-1)! \int_{0<t_1<\ldots<t_{j-1}<r<t_j<\ldots<t_{n-1}<t} \int_{\bR^{2(n-1)}} \prod_{i=1}^{n-1}f(x_i-y_i)\\
& \quad \quad \qquad  \times f_n(t_1,x_1,\ldots,t_{j-1},x_{j-1},r,z,t_j,x_j,\ldots,t_{n-1},x_{n-1},t,x)\\
& \quad \quad \qquad \times
f_n(t_1,y_1,\ldots,t_{j-1},y_{j-1},r,z,t_j,y_j,\ldots,t_{n-1},y_{n-1},t,x)d{\bf
x} d{\bf y}d{\bf t}.
\end{align*}
Similarly as above, we bound the terms $G(t_j-r,x_j-z)$ and $G(s_j-r,y_j-z)$ by $1/2$. The remaining
terms can be separated into two integrals, one on $\{t_1<\ldots<t_{j-1}<r\} \times \bR^{2(j-1)}$
which is bounded by $e^{Mr} \left(\frac{2}{M} \right)^{i-1}K_M^{j-1}$, and the other
on $\{r<t_j<\ldots<t_{n-1}<t\} \times \bR^{2(n-j)}$ which is bounded by $e^{M(t-r)} \left(\frac{2}{M} \right)^{n-j}K_M^{n-j}$, for any $M>0$. Hence, for any $j=1, \ldots,n$,
$$\|(n-1)!h_j^{(n)}(\cdot,r,z,t,x))\|_{\cH^{\otimes (n-1)}}^2 \leq \frac{1}{4}\Gamma_t^{n-1}(n-1)!\,e^{Mt} \left(\frac{2}{M} \right)^{n-1}K_M^{n-1}.$$
Coming back to \eqref{bound-An}, we obtain that for any $(r,z) \in [0,t] \in \bR$
$$\E|A_n(r,z,t,x)|^2 \leq \frac{1}{4}  n^2 \Gamma_t^{n-1}  e^{Mt} \left(\frac{2}{M} \right)^{n-1}K_M^{n-1}.$$
By choosing $M=M_t$ large enough, and using the orthogonality of the Wiener chaos spaces, we see that $$\sup_{(r,z) \in [0,t] \times \bR} \E\left|\sum_{k=n}^{m} A_k(r,z,t,x)\right|^2 \to 0 \quad \mbox{as} \quad n,m \to \infty.$$
\end{proof}

\begin{remark} (The case $d=2$)
{\rm The proof of Lemma \ref{An-series} does not work for the case $d=2$.
Unfortunately, we could not find another argument to prove that $f_{j}^{(n)}(\cdot,r,z,t,x) \in \cH^{\otimes (n-1)}$ for $j<n$, when $d=2$. To see where the problem appears, consider  $n=2$ and $j=1$: $$f_1^{(2)}(t_1,x_1,r,z,t,x)=G(t-t_1,x-x_1)G(t_1-r,x_1-z)1_{\{r<t_1<t\}}$$
Even in this case, we could not show that the function $f_{1}^{(2)}(\cdot,r,z,t,x)$ is in $\cH$ when $d=2$.}
\end{remark}

Our next goal is to establish the measurability of the function $(\omega,r,z) \mapsto D_{r,z}u(t,x)(\omega)$, which will be needed for the application of Corollary \ref{new-cor} (in the proof of Theorem \ref{thm:density}).

We recall the following definitions.

\begin{definition}
{\rm A random field $\{X(t,x);t \geq 0,x \in \bR\}$ defined on a probability space $(\Omega,\cF,\bP)$
is {\em measurable} if the map $(\omega,t,x) \mapsto X(\omega,t,x)$ is measurable with respect to
$\cF \times \cB(\bR_{+}) \times \cB(\bR)$.}
\end{definition}

\begin{definition}
{\rm
The random fields $\{X(t,x);t \geq 0, x \in \bR\}$ and $\{Y(t,x);t \geq 0, x \in \bR\}$ are {\em modifications}
of each other if $X(t,x)=Y(t,x)$ a.s. for all $t \geq 0$ and $x \in \bR$.}
\end{definition}

\begin{theorem}
\label{D-measurable-th}
The process $\{D_{r,z}u(t,x);r \in [0,t],z \in \bR\}$ has a measurable modification.
\end{theorem}

\begin{proof}
We use definition \eqref{series-D} of $D_{r,z}u(t,x)$, and relation \eqref{def-An} in which we separate the term corresponding to $j=n$. More precisely, we write:
\begin{equation}
\label{decompose-D}
D_{r,z}u(t,x)=\sum_{n\geq 1}A_n^*(r,z)+\sum_{n\geq 1}A_n^{(n)}(r,z)=:T(r,z)+T'(r,z),
\end{equation}
with $A_n^*(r,z)=\sum_{j=1}^{n-1}A_j^{(n-1)}(r,z)$ and $A_j^{(n)}(r,z)=I_{n-1}(f_j^{(n)}(\cdot,r,z,t,x))$ for $j=1,\ldots,n$.
We treat separately the two terms on the right-hand side of \eqref{decompose-D}.

{\em Step 1. (The first term)} We will show that the random field $T=\{T(r,z);r \in [0,t],z \in \bR\}$ has a measurable modification. By Proposition \ref{exist-meas}.a) (Appendix \ref{appendix-meas}), it suffices to show that $T$ is $L^2(\Omega)$-continuous. Let $T_n(r,z)=\sum_{k=1}^n A_k^*(r,z)$. The same argument as in the proof of Lemma \ref{An-series} shows that $T_n(r.z)$ converges to $T(r,z)$ in $L^2(\Omega)$, uniformly in $(r,z) \in [0,t] \times \bR^d$. To prove that $T$ is $L^2(\Omega)$-continuous, it suffices to show that
$T_n$ is $L^2(\Omega)$-continuous for any $n \geq 1$. For this, it is enough to prove that  $A_n^*$ is $L^2(\Omega)$-continuous for any $n \geq 1$.

We prove that $A_j^{(n)}$ is $L^2(\Omega)$-continuous, for any $n \geq 1$ and $j=1,\ldots,n-1$.


We fix $j=1,\ldots n-1$. We will prove two things:
\begin{itemize}
 \item[(i)] as $h \to 0$,
$$\E|A_j^{(n)}(r+h,z)-A_j^{(n)}(r,z)|^2=(n-1)! \|h_j^{(n)}(\cdot,r+h,z)-h_j^{(n)}(\cdot,r,z)\|_{\cH^{\otimes (n-1)}} \to 0,$$
uniformly in $z \in \bR$. \item [(ii)] as $|k| \to 0$,
$$\E|A_j^{(n)}(r,z+k)-A_j^{(n)}(r,z)|^2=(n-1)! \|h_j^{(n)}(\cdot,r,z+k)-h_j^{(n)}(\cdot,r,z)\|_{\cH^{\otimes (n-1)}} \to 0,$$
for every $r \in [0,t]$ fixed.
\end{itemize}

We sketch the proof of (i), the continuity in time. The proof of
(ii) can be performed in a similar way. As in the proof of Lemma
\ref{An-series}, this involves a difference of the form
$$f_n(t_1,x_1,\ldots,t_{j-1},x_{j-1},r+h,z,t_j,x_j,\ldots,t_{n-1},x_{n-1},t,x)-$$
$$f_n(t_1,x_1,\ldots,t_{j-1},x_{j-1},r,z,t_j,x_j,\ldots,t_{n-1},x_{n-1},t,x),$$
which is equal to
\begin{align*}
& G(t-t_{n-1},x-x_{n-1}) \cdots
[G(t_j-r-h,x_j-z)G(r+h-t_{j-1},z-x_{j-1})\\
& \qquad \qquad -G(t_j-r,x_j-z)G(r-t_{j-1},z-x_{j-1})] \cdots
G(t_2-t_1,x_2-x_1).
\end{align*}
We examine the difference $[...]$ above.
By inserting plus and minus the mixed term $G(t_j-r-h,x_j-z)G(r-t_{j-1},z-x_{j-1})$, the modulus of this difference is smaller than
\begin{align*}
&G(t_j-r-h,x_j-z)|G(r+h-t_{j-1},z-x_{j-1})-G(r-t_{j-1},z-x_{j-1})|\\
& \qquad
+|G(t_j-r-h,x_j-z)-G(t_{j}-r,x_{j}-z)|G(r-t_{j-1},z-x_{j-1})=:A+B.
\end{align*}
When we take the norm in $\cH^{\otimes (n-1)}$ of the term involving
$A$, we proceed as follows, First, we bound the term
$G(t_j-r-h,x_j-z)$ by $\frac12$ (indeed, we bound two terms of this
kind, because we have a norm in $|\cP_0|$). The remaining expression
splits into a product of two integrals, one of which on
$\{r<t_j<\ldots<t_{n-1}<t\} \times \bR^{2(n-j)}$ and is bounded by
$e^{M(t-r)} \left(\frac{2}{M} \right)^{n-j}K_M^{n-j}$, for any
$M>0$. Regarding the other term, using the arguments in the proof of
Lemma \ref{sum-alpha-finite} it can be bounded, up to some constant,
by
\begin{align}
& \int_0^t \int_{\bR^2} f(x_j-x'_j) |G(r+h-t_{j-1},z-x_{j-1})-
G(r-t_{j-1},z-x_{j-1})| \nonumber \\
& \qquad \qquad \quad \times |G(r+h-t_{j-1},z-x'_{j-1})-
G(r-t_{j-1},z-x'_{j-1})| dx_{j-1} dx'_{j-1} dt_{j-1}\nonumber \\
& = \int_0^t \int_{\bR^2} f(x_j-x'_j) |G(r+h-t_{j-1},x_{j-1})-
G(r-t_{j-1},x_{j-1})| \nonumber \\
& \qquad \qquad \quad \times |G(r+h-t_{j-1},x'_{j-1})-
G(r-t_{j-1},x'_{j-1})| dx_{j-1} dx'_{j-1} dt_{j-1}, \label{eq:156}
\end{align}
where we have applied a change of variables in order to get rid of
$z$. Finally, in \eqref{eq:156} we apply the bounded convergence
theorem thanks to the continuity of the map $t \mapsto G(t,x)$ for
fixed $x$ and taking into account that, assuming $0<h<1$,
\begin{align*}
|G(r+h-t_{j-1},x_{j-1})- G(r-t_{j-1},x_{j-1})| & \leq
G(r+h-t_{j-1},x_{j-1}) \leq G(T,x_{j-1}),
\end{align*}
for some $T>0$.

We now deal with the norm in $\cH^{\otimes (n-1)}$ of the term
involving $B$. In this case, this norm can be directly written as a
product of two terms, one of which is an integral on
$\{t_1<\ldots<t_{j-1}<r\} \times \bR^{2(j-1)}$ that is bounded by
$e^{Mr} \left(\frac{2}{M} \right)^{i-1}K_M^{j-1}$, for all $M>0$.
The other term is given by
\begin{align*}
&\Gamma_t^{n-j} \int_{[0,t]^{n-j}} \int_{(\bR^2)^{n-j}}
f(x_j-x_j')\cdots f(x_{n-1}-x_{n-1}')
|G(t_j-r-h,x_j-z)-G(t_j-r,x_j-z)| \\
& \; \times G(t_{j+1}-t_j,x_{j+1}-x_j) \cdots G(t-t_{n-1},x-x_{n-1})
 |G(t_j-r-h,x'_j-z)-G(t_j-r,x'_j-z)| \\
& \; \times G(t_{j+1}-t_j,x'_{j+1}-x'_j) \cdots
G(t-t_{n-1},x-x'_{n-1})\, dx_j \cdots dx_{n-1} dx'_j \cdots
dx'_{n-1} dt_j\cdots dt_{n-1}.
\end{align*}
In order to get rid of the dependence on $z$, we make a series of
changes of variable and move $z$ from the terms involving $x_j$ and
$x'_j$ to those involving $x_{n-1}$ and $x'_{n-1}$. Another change
of variable allows us eventually remove the variable $z$. In the
remaining term, one can easily apply the bounded convergence
theorem, because the differences in absolute value can be bounded by
a constant. This would let us conclude the proof of item (i).



{\em Step 2. (The second term)} Note that
$T'(r,z)=\sum_{n\geq 1}A_n^{(n)}(r,z)=G(t-r,x-z)F(r,z)$, where $$F(r,z)=\sum_{n\geq 1}n I_{n-1}(f_{n-1}(\cdot,r,z))=\sum_{n\geq 0}(n+1)J_n(r,z),$$
and $J_n(r,z)=I_n(f_n(\cdot,r,z))$.
Since the map $(r,z) \mapsto G(t-r,x-z)$ is measurable on $[0,t] \times \bR$, it suffices to show that $F$ has a measurable modification. This will follow by Proposition \ref{exist-meas}.a) (Appendix \ref{appendix-meas}), once we show that $F$ is $L^p(\Omega)$-continuous, for any $p\geq 2$. For this, we use the same argument as in the proof of Theorem 7.1 of \cite{balan-song17}. We denote by $\|\cdot\|_p$ the $L^p(\Omega)$-norm.

Using the equivalence of norms $\|\cdot\|_p$ for random variables in the same Wiener chaos space (see last line of p.62 of \cite{nualart06}) and \eqref{alpha-bound}, for any $r \in [0,t]$ and $z \in \bR^d$, we have:
\begin{align*}
\|J_n(r,z)\|_p & \leq (p-1)^{n/2} \|J_n(r,z)\|_2= (p-1)^{n/2} \left( \frac{1}{n!}\alpha_n(r)\right)^{1/2}\\
& \leq  (p-1)^{n/2} e^{Mr/2} \left( \frac{2\Gamma_r}{M}K_M\right)^{n/2}.
\end{align*}
Using the fact that $n+1 \leq e^n$ and choosing $M$ sufficiently large, it follows that
$$\sum_{n \geq 0} (n+1) \sup_{(r,z) \in [0,t] \times \bR^d} \|J_n(r,z)\|_p \leq e^{Mt/2}  \sum_{n\geq 0} e^n (p-1)^{n/2} \left( \frac{2\Gamma_t}{M}K_M\right)^{n/2}
<\infty.$$
Hence, the sequence $\{F_n(r,z)=\sum_{k=0}^{n}(k+1)J_k(r,z);n \geq 1\}$ converges to $F(r,z)$ in $L^p(\Omega)$, uniformly in $(r,z) \in [0,t] \times \bR^d$. Note that $F_n$ is $L^p(\Omega)$-continuous for any $n$,
since $J_n$ is $L^p(\Omega)$-continuous for any $n$ (by Lemma 7.1 of \cite{balan-song17}). Therefore, $F$ is
$L^p(\Omega)$-continuous.

\end{proof}

\begin{remark}
\label{D-measurable-rem}
{\rm For the remaining part of the article, we fix $t>0$ and $x \in \bR$, and we work with the measurable modification given by Theorem \ref{D-measurable-th}, which will be denoted also by $\{D_{r,z}u(t,x);r \in [0,t],z \in \bR\}$.}
\end{remark}

We aim to prove that for any fixed $r\in [0,t]$, the Malliavin
derivative $D_{r,\cdot}u(t,x)$ satisfies an equation in $L^2(\Omega;L^2(\bR))$. For this, we first need to prove that $D_{r,\cdot}u(t,x)$ belongs to $L^2(\bR)$ a.s. This fact is established by the following result.

\begin{theorem}
\label{th-D-in-L2}
For any $T>0$,
\[
\sup_{(t,x)\in [0,T]\times \bR} \, \sup_{r\in [0,t]}\E \int_\bR
|D_{r,z} u(t,x)|^2 dz  \leq C_T,
\]
where $C_T>0$ is a constant depending on $T$ (which may not be increasing in $T$).
\end{theorem}

\begin{proof}
Using the orthogonality of Wiener chaos spaces, we have
\begin{align}
\E \int_\bR |D_{r,z} u(t,x)|^2 dz & = \int_\bR \E \Big|\sum_{n\geq
1}nI_{n-1}(\widetilde{f}_n(\cdot,r,z,t,x))\Big|^2 dz \nonumber \\
& = \sum_{n\geq 1} \int_\bR \E
\big|nI_{n-1}(\widetilde{f}_n(\cdot,r,z,t,x))\big|^2 dz.
\label{eq:99}
\end{align}
The proof of Lemma \ref{An-series} shows that
$f_j^{(n)}(\cdot,r,z,t,x)\in \cH^{\otimes (n-1)}$ and
\begin{align*}
& \E \big|nI_{n-1}(\widetilde{f}_n(\cdot,r,z,t,x))\big|^2 \\
 & \; \leq
n \sum_{j=1}^{n}\frac{1}{(n-1)!} \,
\|(n-1)!h_j^{(n)}(\cdot,r,z,t,x))\|_{\cH^{\otimes (n-1)}}^2 \\
 & \; \leq n \sum_{j=1}^{n}\frac{1}{(n-1)!} \,
 \Gamma_t^{n-1} (n-1)! \int_{0<t_1<\ldots<t_{j-1}<r<t_j<\ldots<t_{n-1}<t} \int_{\bR^{2(n-1)}} \prod_{i=1}^{n-1}f(x_i-y_i)\\
& \qquad \qquad \times f_n(t_1,x_1,\ldots,t_{j-1},x_{j-1},r,z,t_j,x_j,\ldots,t_{n-1},x_{n-1},t,x)\\
& \qquad \qquad \times
f_n(t_1,y_1,\ldots,t_{j-1},y_{j-1},r,z,t_j,y_j,\ldots,t_{n-1},y_{n-1},t,x)d{\bf
x} d{\bf y}d{\bf t}.
\end{align*}
Taking the integral $dz$ and applying Fubini's theorem, we have
\begin{align}
& \int_{\bR}\E |nI_{n-1}(\widetilde{f}_n(\cdot,r,z,t,x))|^2 dz \leq n \Gamma_t^{n-1} \sum_{j=1}^n
\int_{0<t_1<\ldots<t_{j-1}<r<t_j<\ldots<t_{n-1}<t} \nonumber \\
& \; \qquad  \times \left\{ \int_\bR \int_{\bR^{2(n-1)}}
\prod_{i=1}^{n-1}f(x_i-y_i)
 f_n(t_1,x_1,\ldots,t_{j-1},x_{j-1},r,z,t_j,x_j,\ldots,t_{n-1},x_{n-1},t,x)\right. \nonumber \\
& \; \qquad  \times
f_n(t_1,y_1,\ldots,t_{j-1},y_{j-1},r,z,t_j,y_j,\ldots,t_{n-1},y_{n-1},t,x)d{\bf
x} d{\bf y} dz \Bigg\} d{\bf t}. \label{eq:98}
\end{align}
Observe that the expression inside $\{...\}$ above has the same
structure as $\psi_n({\bf t},{\bf t})$ in the proof of Lemma \ref{sum-alpha-finite}, except that one of the covariances $f$ is replaced by a
Dirac delta (that is a white covariance), whose spectral measure is
the Lebesgue measure on $\bR$. Equivalently, expression inside
$\{...\}$ can be seen as a norm in $|\cP_0|^{\otimes (j-1)}\otimes
L^2(\bR)\otimes |\cP_0|^{\otimes (n-j)}$ (since everything is
non-negative). Hence, using inequality (4.16) of
\cite{balan-song17}, we obtain that $\{...\}$ in \eqref{eq:98} is
bounded by
\begin{align*}
& \left(\sup_{\eta\in \bR} \int_\bR |\cF G(t_2-t_1)(\xi_1+\eta)|^2
\mu(d\xi_1)\right)\times \cdots \times \left(\sup_{\eta\in \bR}
\int_\bR |\cF G(r-t_{j-1})(\xi_{j-1}+\eta)|^2 \mu(d\xi_{j-1})\right)
\\
&\; \times \left(\sup_{\eta\in \bR} \int_\bR |\cF
G(t_j-r)(\xi+\eta)|^2 d\xi \right) \times \left(\sup_{\eta\in \bR}
\int_\bR |\cF G(t_{j+1}-t_j)(\xi_j+\eta)|^2
\mu(d\xi_j)\right) \\
&\; \times \cdots \times \left(\sup_{\eta\in \bR} \int_\bR |\cF
G(t-t_{n-1})(\xi_{n-1}+\eta)|^2 \mu(d\xi_{n-1})\right).
\end{align*}
Moreover, applying a change of variable and Plancherel's formula, it
holds that
\begin{align}
\nonumber
\int_\bR |\cF G(t_j-r,\cdot)(\xi+\eta)|^2 d\xi & = \int_\bR |\cF
G(t_j-r,\cdot)(\xi)|^2 d\xi= 2\pi \int_\bR
|G(t_j-r,x)|^2 dx\\
\label{bound-L2-F}
& = \frac{\pi}{2} \int_\bR 1_{\{|x|<t_j - r\}} dx = \pi (t_j-r)\leq \pi t.
\end{align}
Therefore, using Lemma \ref{max-principle} we have proved that:
\begin{equation}
\label{I-n-1-L2-estimate}
\int_{\bR}\E |nI_{n-1}(\widetilde{f}_n(\cdot,r,z,t,x))|^2 dz   \leq \pi t  n
\Gamma_t^{n-1} \sum_{j=1}^n H_j^{(n)},
\end{equation}
where
\begin{align}
\nonumber
H_j^{(n)}&:=
\int_{0<t_1<\ldots<t_{j-1}<r<t_j<\ldots<t_{n-1}<t} \int_{\bR^{n-1}}\frac{\sin^2((t_2-t_1)|\xi_1|)}{|\xi_1|^2} \ldots
\frac{\sin^2((r-t_{j-1})|\xi_{j-1}|)}{|\xi_{j-1}|^2} \\
\label{def-Hjn}
& \quad \frac{\sin^2((t_{j+1}-t_{j})|\xi_{j}|)}{|\xi_{j}|^2}  \ldots\frac{\sin^2((t-t_{n-1})|\xi_{n-1}|)}{|\xi_{n-1}|^2} \mu(d\xi_1)\ldots \mu(d\xi_n) dt_1 \ldots dt_{n-1}.
\end{align}

Next, we make the change of variables $$u_1=t_2-t_1,\ldots,u_{j-1}=r-t_j,u_j=t_{j+1}-t_j,\ldots,u_{n-1}=t-t_{n-1}.$$ Note that
$\sum_{k=1}^{j-1}u_k=r-t_1$, $\sum_{k=j}^{n-1} u_k=t-t_j$, and hence $\sum_{k=1}^{n-1}u_k<2t$. In the previous integral, after we change the variables, we insert the terms $e^{-Mu_k}$ for $k=1,\ldots,n-1$, and we use the fact that $e^{M(u_1+\ldots+u_{n-1})}\leq e^{2Mt}$. Then,
\begin{align}
\nonumber
H_j^{(n)} &  \leq
\int_{[0,t]^{n-1}} 1_{\{\sum_{k=1}^{n-1}u_k<2t\}}  e^{M(\sum_{k=1}^{n-1}u_k)}
\int_{\bR^{n-1}} e^{-Mu_1}\frac{\sin^2
(u_1|\xi_1|)}{|\xi_1|^2}  \cdots e^{-Mu_{j-1}} \frac{\sin^2
(u_{j-1}|\xi_{j-1}|)}{|\xi_{j-1}|^2} \\
\nonumber
& \quad e^{-Mu_j} \frac{\sin^2 (u_j|\xi_{j}|)}{|\xi_{j}|^2}
\cdots e^{-Mu_{n-1}}\frac{\sin^2 (u_{n-1}|\xi_{n-1}|)}{|\xi_{n-1}|^2}
 \mu(d\xi_1)\cdots \mu(d\xi_{n-1}) du_1 \cdots
du_{n-1}\\
\nonumber
& \leq  e^{2Mt} \left(\frac2M\right)^{n-1} \int_{\bR^{n-1}}  \prod_{k=1}^{n-1} \frac{1}{M^2+4|\xi_k|^2}
\mu(d\xi_1)\cdots \mu(d\xi_{n-1})\\
\nonumber
& =  e^{2Mt} \left(\frac2M\right)^{n-1} \left(\int_{\bR} \frac{1}{M^2+4|\xi|^2} \mu(d\xi)\right)^{n-1}\\
\label{estimate-Hjn}
& = e^{2Mt} \left(\frac2M\right)^{n-1} K_M^{n-1},
\end{align}
where we have applied the same arguments used in the proof of
Lemma \ref{sum-alpha-finite} (and the same notations). Note that the above estimate provides an upper bound for $H_j^{(n)}$ which does not depend on $j$.

Using \eqref{I-n-1-L2-estimate}, it follows that
\begin{equation}
\label{2-moment-I}
\int_{\bR}\E|n I_{n-1}(\widetilde{f}_{n}(\cdot,r,z,t,x))|^2 dz\leq \pi t n^2 \Gamma_t^{n-1} e^{2Mt}  \left(\frac2M\right)^{n-1} K_M^{n-1}.
\end{equation}

We now return to \eqref{eq:99}. We use the fact that $n \leq e^n$ for all $n \geq 1$.
We conclude that for any $M>0$, $t \in [0,T]$, $x \in \bR$ and $r \in [0,t]$,
$$\E\int_{\bR}|D_{r,z}u(t,x)|^2 dz \leq \pi T e^{2MT}e^2 \sum_{n\geq 1}\left(e^2 \Gamma_T \frac{2}{M} K_M \right)^{n-1}.$$
Choose $M=M_T>2$ large enough such that
$e^2 \Gamma_T K_{M_T} < 1/2$. Then
\begin{equation}
\label{L2-D-norm}
\E\int_{\bR}|D_{r,\cdot}u(t,x)|^2dz \leq \pi T e^{2M_T T}e^2 \sum_{n\geq 1}\left(\frac{1}{2} \right)^{n-1}.
\end{equation}
The conclusion follows with $C_T=2\pi T e^{2M_T T}e^2.$
\end{proof}

\begin{remark}
\label{old-method-P0}
{\rm Fix $0<r<t$ and $x \in \bR$.
Theorem \ref{th-D-in-L2} shows that with probability $1$, the functions $z \mapsto I_{n-1}(\widetilde{f}_{n}(\cdot,r,z,t,x))$ and $z \mapsto D_{r,z}u(t,x)$ belong to $L^2(\bR)$. We denote by $\cF_{z} [I_{n-1}(\widetilde{f}_{n}(\cdot,r,\cdot,t,x))]$ and $\cF_{z}[D_{r,\cdot}u(t,x)]$ the Fourier transforms of these functions. Similarly to the proof of relation \eqref{chaos-D2} below, it can be proved that for almost all $\xi \in \bR$,
$$\cF_z \big[D_{r,\cdot} u(t,x)\big](\xi)=\sum_{n\geq 1}n \cF_{z} \Big[I_{n-1}(\widetilde{f}_{n}(\cdot,r,\cdot,t,x))\Big](\xi)=\sum_{n\geq 1}n I_{n-1}\big(\cF_{z} \widetilde{f}_{n}(\cdot,r,\cdot,t,x)(\xi)\big),$$
which gives the chaos expansion of the Fourier transform of $D_{r,\cdot}u(t,x)$. Hence,
\begin{align*}
\E\|D_{r,\cdot}u(t,x)\|_{0}^2 &= \E \int_{\bR} |\cF_z [D_{r,\cdot}u(t,x)](\xi)|^2 \mu(d\xi) \\ & =\sum_{n \geq 1}n^2 (n-1)! \int_{\bR}\|\cF_{z} \widetilde{f}_{n}(\cdot,r,\cdot,t,x)(\xi)\|_{\cH^{\otimes (n-1)}}^2 \mu(d\xi)\\
& = \sum_{n \geq 1}n^2 (n-1)! \, \|\widetilde{f}_n(\cdot,r,\cdot,t,x) \|_{\cH^{\otimes (n-1)} \otimes \cP_0}^2,
\end{align*}
where for the last equality we expressed the norm in $\cH^{\otimes (n-1)}$ using the Fourier transform with respect to the variables $(t_1,x_1),\ldots,(t_{n-1},x_{n-1})$.
Similarly to the proof of Theorem \ref{H-L2-norm} below, it can be shown that
$\E\|D_{r,\cdot}u(t,x)\|_0^2<\infty$, and hence $D_{r,\cdot}u(t,x) \in \cP_0$ a.s. Moreover,
$$\sup_{(t,x) \in [0,T] \times \bR} \sup_{r \in [0,t]}\E\|D_{r,\cdot}u(t,x)\|_0^2 <\infty.$$
These facts will not be used in the present article.
}
\end{remark}

\section{The second order Malliavin derivative}
\label{section-D2}

In this section, we study the second order Malliavin derivative of the solution.

Note that
\begin{align}
\nonumber
D^2_{(\theta,w),(r,z)}u(t,x)&=D_{\theta,w}\big(D_{r,z}u(t,x)\big)=D_{\theta,w}\Big(\sum_{n\geq 1}n I_{n-1}(\widetilde{f}_n(\cdot,r,z,t,x)) \Big)\\
\label{def-D2}
&=\sum_{n\geq 2} n(n-1)I_{n-2}(\widetilde{f}_n(\cdot,\theta,w,r,z,t,x)).
\end{align}

We will show below that $\widetilde{f}_n(\cdot,\theta,w,r,z,t,x) \in \cH^{\otimes (n-2)}$ and the series above converges in $L^2(\Omega)$. First, note that for any $j=1,\ldots,n$
\begin{equation}
\label{hjn-sum}
h_{j}^{(n)}(\cdot,\theta,w,r,z,t,x)=\frac{1}{n-1}\sum_{i=1,i\not=j}^{n}
h_{ij}^{(n)}(\cdot,\theta,w,r,z,t,x),
\end{equation}
where
\begin{align*}
&h_{ij}^{(n)}(t_1,x_1,\ldots,t_{n-2},x_{n-2},\theta,w,r,z,t,x)=\\
& \quad \frac{1}{(n-2)!}\sum_{\rho \in S_{n-2}} f_n(t_{\rho(1)},x_{\rho(1)},\ldots,t_{\rho(i-1)},x_{\rho(i-1)},\theta,w,t_{\rho(i)},x_{\rho(i)},
\ldots,
t_{\rho(j-1)},x_{\rho(j-1)},r,z,\\
& \qquad \qquad \qquad \qquad \quad t_{\rho(j)},x_{\rho(j)},\ldots,t_{\rho(n-2)},x_{\rho(n-2)},t,x).
\end{align*}

In the definition of $h_{ij}^{(n)}(\cdot,\theta,w,r,z,t,x)$, $(\theta,w)$ is on position $i$ and $(r,z)$ is on position $j$, as arguments of the function $f_n(\cdot,t,x)$. Since the function $f_n(\cdot,t,x)$ contains the indicator of the set $\{0<t_1<\ldots<t_n<t\}$,
this means that if $\theta<r$, then $h_{ij}^{(n)}(\cdot,\theta,w,r,z,t,x)=0$ for all $i>j$; on the other hand, if $\theta>r$, then $h_{ij}^{(n)}(\cdot,\theta,w,r,z,t,x)=0$ for all $i<j$.

For $\theta<r$ and $i<j$, $h_{ij}^{(n)}(\cdot,\theta,w,r,z,t,x)$ is the symmetrization of the function $f_{ij}^{(n)}(\cdot,\theta,w,r,z,t,x)$ defined by:
\begin{align*}
& f_{ij}^{(n)}(t_1,x_1,\ldots,t_{n-2},x_{n-2},\theta,w,r,z,t,x)\\
&\quad = f_n(t_1,x_1,\ldots,t_{i-1},x_{i-1},\theta,w,t_i,x_i,\ldots,t_{j-1},x_{j-1},r,z,t_j,x_j,\ldots,
t_{n-2},x_{n-2},t,x)\\
&\quad = G(t-t_{n-2},x-x_{n-2}) \cdots G(t_i-\theta,x_i-w)G(\theta-t_{i-1},w-x_{i-1}) \cdots G(t_j-r,x_j-z) \\
& \quad \quad G(r-x_{j-1},z-x_{j-1}) \cdots G(t_2-t_1,x_2-x_1)
1_{\{0<t_1<\ldots<t_{i-1}<\theta<t_i<\ldots<t_{j-1}<r<t_j<\ldots<t_{n-2}<t\}}.
\end{align*}
If $r<\theta$ and $j<i$, the function $f_{ij}^{(n)}(\cdot,\theta,w,r,z,t,x)$ is defined similarly and contains the indicator of the set $\{0<t_1<\ldots<t_{j-1}<r<t_j<\ldots<t_{i-1}<\theta<t_i<\ldots<t_{n-2}<t\}$.


By \eqref{tilde-fn-sum} and \eqref{hjn-sum},
\begin{equation}
\label{fn-sum}
\widetilde{f}_n(\cdot,\theta,w,r,z,t,x)=\frac{1}{n(n-1)}\sum_{i,j=1,i\not=j}^{n}
h_{ij}^{(n)}(\cdot,\theta,w,r,z,t,x).
\end{equation}
If $\theta<r$, then $h_{ij}^{(n)}(\cdot,\theta,w,r,z,t,x)=0$ if $i>j$. So the previous sum has only $\frac{n(n-1)}{2}$ terms.

\begin{lemma}\label{lem:34}
For any $(r,z),(\theta,w) \in [0,t] \times \bR$, $\widetilde{f}_{n}(\cdot,\theta,w,r,z,t,x)
\in \cH^{\otimes (n-2)}$. Moreover, the series
$$\sum_{n \geq 2} n(n-1)I_{n-2}(\widetilde{f}_{n}(\cdot,\theta,w,r,z,t,x)) \quad converges
\ in \ L^2(\Omega),$$
uniformly in $(r,z),(\theta,w) \in [0,t] \times \bR$.
\end{lemma}

\begin{proof}
For the first statement, we fix $(r,z),(\theta,w) \in [0,t] \times \bR$. Say $\theta<r$. To prove that $\widetilde{f}_{n}(\cdot,\theta,w,r,z,t,x) \in \cH^{\otimes (n-2)}$, it suffices to show that $f_{ij}^{(n)}(\cdot,\theta,w,r,z,t,x) \in \cH^{\otimes (n-2)}$ for any $1 \leq i<j \leq n$. We proceed as in the proof of Lemma \ref{An-series}. In the computation of the squared $\cH^{\otimes (n-2)}$-norm of
$f_{ij}^{(n)}(\cdot,\theta,w,r,z,t,x)$, we bound by $1/2$ the terms $G(t_j-r,x_j-z),G(t_i-\theta,x_i-w),G(s_j-r,y_j-z),G(s_i-\theta,y_i-w)$. We obtain that:
\begin{align*}
& \|f_{ij}^{(n)}(\cdot,\theta,w,r,z,t,x)\|_{\cH^{\otimes (n-2)}}^2 \leq \\
& \quad \quad \quad \frac{1}{16} \|f_{i-1}(\cdot,\theta,w)\|_{\cH^{\otimes (i-1)}}^2 \|f_{j-i}(\cdot,r-\theta,z) \|_{\cH^{\otimes (j-i)}}^2 \|f_{n-j-1}(\cdot,t-r,x) \|_{\cH^{\otimes (n-j-1)}}^{2}<\infty.
\end{align*}

To prove the convergence of the series, we let
$$
B_n(\theta,w,r,z,t,x):=n(n-1)I_{n-2}(\widetilde{f}_n(\cdot,\theta,w,r,z,t,x))=
\sum_{i,j=1,i\not=j}^{n}I_{n-2}(h_{ij}^{(n)}(\cdot,\theta,w,r,z,t,x)).$$
Using the inequality $(\sum_{i=1}^N a_i)^2 \leq N \sum_{i=1}^N a_i^2$, we see that
\begin{align}
\nonumber
\E|B_n(\theta,w,r,z,t,x)|^2 & \leq n(n-1) \sum_{i,j=1,i\not=j}^{n}\E|I_{n-2}(h_{ij}^{(n)}(\cdot,\theta,w,r,z,t,x))|^2 \\
\label{bound-Bn}
&= n(n-1) \sum_{i,j=1,i\not=j}^{n} \frac{1}{(n-2)!}\|(n-2)!h_{ij}^{(n)}(\cdot,\theta,w,r,z,t,x))\|_{\cH^{\otimes (n-2)}}^2.
\end{align}
Proceeding as in the proof of Lemma \ref{An-series}, we see that for any $i<j$ and $\theta<r$,
\begin{align}
\label{bound-norm-hij}
& \|(n-2)!h_{ij}^{(n)}(\cdot,\theta,w,r,z,t,x))\|_{\cH^{\otimes (n-2)}}^2 \leq \\
\nonumber
&\quad \Gamma_t^{n-2}(n-2)! \int_{t_1<\ldots<t_{i-1}<\theta<t_i<\ldots<t_{j-1}<r<t_j<\ldots<t_{n-2}<t} \int_{\bR^{2(n-2)}} \prod_{k=1}^{n-2}f(x_k-y_k) \\
\nonumber
& \quad G(t-t_{n-2},x-x_{n-2}) \cdots G(t_j-r,x_j-z) G(r-t_{j-1},z-x_{j-1}) \cdots
G(t_i-\theta,x_i-w)\\
\nonumber
& \quad G(\theta-t_{i-1},w-x_{i-1}) \cdots G(t_2-t_1,x_2-x_1) G(t-t_{n-2},x-y_{n-2}) \cdots G(t_j-r,y_j-z)\\
\nonumber
& \quad G(r-t_{j-1},z-y_{j-1}) \cdots
G(t_i-\theta,y_i-w) G(\theta-t_{i-1},w-y_{i-1}) \cdots G(t_2-t_1,y_2-y_1)\\
& \qquad d{\bf x}d{\bf y}d{\bf t}
\end{align}
We bounded each of the terms $G(t_j-r,x_j-z),G(t_i-\theta,x_i-w),G(t_j-r,y_j-z),G(t_i-\theta,y_i-w)$ by $1/2$. The remaining integrals are separated into three integrals, one on $\{t_1<\ldots<t_{i-1}<\theta\} \times \bR^{i-1}$ which is bounded by $e^{Mr} \left(\frac{2}{M}\right)^{i-1} K_M^{i-1}$, the second on $\{\theta<t_i<\ldots<t_{j-1}<r\} \times \bR^{j-i}$ which is bounded by $e^{M(r-\theta)} \left(\frac{2}{M}\right)^{j-i} K_M^{j-i}$, and the last one on $\{r<t_j<\ldots<t_{n-2}<t\} \times \bR^{n-j-1}$ which is bounded by
$e^{M(t-r)} \left(\frac{2}{M}\right)^{n-j-1} K_M^{n-j-1}$, for any $M>0$.
A similar bound holds for $j<i$ and $r<\theta$.

Hence, for any $i,j=1, \ldots,n$ with $i\not=j$,
$$\|(n-2)!h_{ij}^{(n)}(\cdot,\theta,w,r,z,t,x))\|_{\cH^{\otimes (n-2)}}^2 \leq \frac{1}{16}\Gamma_t^{n-2}(n-2)!\,e^{Mt} \left(\frac{2}{M} \right)^{n-2}K_M^{n-2}.$$
Coming back to \eqref{bound-Bn}, we obtain that for any $(\theta,w),(r,z) \in [0,t] \in \bR$
$$\E|B_n(\theta,w,r,z,t,x)|^2 \leq \frac{1}{16}  [n(n-1)]^2 \Gamma_t^{n-2} e^{Mt} \left(\frac{2}{M} \right)^{n-2}K_M^{n-2}.$$

By choosing $M=M_t$ large enough, and using the orthogonality of the Wiener chaos spaces, we see that
$$\sup_{(\theta,w),(r,z) \in [0,t] \times \bR} \E\left|\sum_{k=n}^{m} B_k(\theta,w,r,z,t,x)\right|^2 \to 0 \quad \mbox{as} \quad n,m \to \infty.$$
\end{proof}

\begin{theorem}
\label{th-D2-L2}
For any $T>0$,
$$\sup_{(t,x) \in [0,T] \times \bR} \sup_{\theta,r \in [0,t]} \E \int_{\bR} \int_{\bR} |D^2_{(\theta,w),(r,z)}u(t,x)|^2dwdz\leq C_T',$$
where $C_T'>0$ is a constant depending on $T$.
\end{theorem}

\begin{proof} Fix $\theta<r<t$. Using \eqref{def-D2}, Fubini's theorem and orthogonality, we have:
\begin{equation}
\label{bound-D2-L2}
\E \int_{\bR} \int_{\bR} |D_{(\theta,w),(r,z)}u(t,x)|^2dwdz=\sum_{n\geq 2}\int_{\bR} \int_{\bR}\E\big|n(n-1)I_{n-2}(\widetilde{f}_n(\cdot,\theta,w,r,z,t,x))\big|^2dwdz.
\end{equation}

To estimate the second moment of $n(n-1)I_{n-2}(\widetilde{f}_n(\cdot,\theta,w,r,z,t,x))$, we will use \eqref{bound-Bn}. Let $1 \leq i<j \leq n$.
Using \eqref{bound-norm-hij} and Fubini's theorem,
\begin{align*}
& \int_{\bR} \int_{\bR} \|(n-2)!h_{ij}^{(n)}(\cdot,\theta,w,r,z,t,x))\|_{\cH^{\otimes (n-2)}}^2 dwdz \\
& \; \leq \Gamma_t^{n-2} (n-2)! \int_{t_1<\ldots<t_{i-1}<\theta<t_i<\ldots<t_{j-1}<r<t_j<\ldots<t_{n-2}<t}\\
 & \int_{\bR^2} \int_{\bR^{2(n-2)}} f_n(t_1,x_1,\ldots,t_{i-1},x_{i-1},\theta,w,t_i,x_i,\ldots,t_{j-1},x_{j-1},r,z,t_j,x_j,\ldots,
 t_{n-2},x_{n-2},t,x)\\
 & f_n(t_1,y_1,\ldots,t_{i-1},y_{i-1},\theta,w,t_i,y_i,\ldots,t_{j-1},y_{j-1},r,z,t_j,y_j,\ldots,
 t_{n-2},y_{n-2},t,x)d{\bf x} d{\bf y} dw dz d{\bf t}.
\end{align*}
We evaluate separately the inner integral $d{\bf x}d{\bf y}dwdz$ above, for fixed ${\bf t}=(t_1,\ldots,t_{n-2})$. This integral is equal to the norm in $\cP_0^{\otimes (i-1)} \otimes L^2(\bR) \otimes \cP_0^{j-i} \otimes L^2(\bR) \otimes \cP_0^{n-j-1}$ of the function
$f_n(t_1,\cdot,\ldots,t_{i-1},\cdot,\theta,\cdot,t_i,\cdot,\ldots,t_{j-1},\cdot,r,\cdot,
t_j,\cdot,\ldots,t_{n-2},\cdot,t,x)$, which in turn is equal to:
\begin{align*}
& \int_{\bR^n} |\cF G(t_2-t_1,\cdot)(\xi_1)|^2 \cdots |\cF G(\theta-t_{i-1},\cdot)(\xi_1+\ldots+\xi_{i-1})|^2 |\cF G(t_{i}-\theta,\cdot)(\xi_1+\ldots+\xi_{i-1}+\zeta)|^2\\
& \quad \cdots |\cF G(r-t_{j-1},\cdot)(\xi_1+\ldots+\xi_{j-1}+\zeta)|^2 |\cF G(t_{j}-r,\cdot)(\xi_1+\ldots+\xi_{j-1}+\zeta+\xi)|^2  \\
& \quad \cdots |\cF G(t-t_{n-2},\cdot)(\xi_1+\ldots+\xi_{n-2}+\zeta+\xi)|^2
\prod_{k=1}^{n-2}\mu(d\xi_k) d\zeta d\xi.
\end{align*}
To estimate this norm, we use Lemma \ref{max-principle} for the integrals $\mu(d\xi_k)$ with $k=1,\ldots,n-2$, the bound \eqref{bound-L2-F} for the $d\xi$ integral, and a similar bound for the $d\zeta$ integral:
$$\sup_{\eta \in \bR}\int_{\bR}|\cF G(t_i-\theta,\cdot)(\zeta+\eta)|^2 d\zeta \leq \pi T.$$
Therefore, the norm mentioned above is bounded by
\begin{align*}
& (\pi T)^2 \int_{\bR^{n-2}} \frac{\sin^2((t_2-t_1)|\xi_1|)}{|\xi_1|^2} \ldots
\frac{\sin^2((\theta-t_{i-1})|\xi_{i-1}|)}{|\xi_{i-1}|^2}
\ldots
\frac{\sin^2((r-t_{j-1})|\xi_{j-1}|)}{|\xi_{j-1}|^2} \ldots \\
& \quad \quad \quad \frac{\sin^2((t-t_{n-2})|\xi_{n-2}|)}{|\xi_{n-2}|^2} \prod_{k=1}^{n-2}\mu(d\xi_k).
\end{align*}
We now come back and take the integral $dt_1 \ldots dt_{n-2}$. We use the change of variable $u_1=t_2-t_1,\ldots,u_{i-1}=\theta-t_{i-1},\ldots,u_{j-1}=r-t_{j-1},
\ldots,u_{n-2}=t-t_{n-2}$. Inserting the terms $e^{-Mu_k}$ for $k=1,\ldots,n$ and using the fact that $e^{M(u_1+\ldots+u_{n-2})} \leq e^{3Mt}$, and proceeding similarly as in the proof of Theorem \ref{th-D-in-L2}, we see that
$$\int_{\bR} \int_{\bR}\|h_{ij}^{(n)}(\cdot,\theta,w,r,z,t,x) \|_{\cH^{\otimes(n-2)}}dw dz \leq \Gamma_t^{n-2}(n-2)! (\pi T)^2 e^{3Mt} \left(\frac{2}{M}\right)^{n-2} K_M^{n-2},$$
for any $M>0$. Using \eqref{bound-Bn},
it follows that
\begin{equation}
\label{bpund-Bn-L2}
\E\int_{\bR} \int_{\bR}|n(n-1)I_{n-2}(\widetilde{f}_n(\cdot,\theta,w,r,z,t,x))|^2 dwdz \leq [n(n-1)]^2 \Gamma_t^{n-2} (\pi T)^2 e^{3Mt}\left(\frac{2}{M}\right)^{n-2} K_M^{n-2}.
\end{equation}

Coming back to \eqref{bound-D2-L2}, it follows that
$$\E \int_{\bR} \int_{\bR}|D_{(\theta,w),(r,z)}^2 u(t,x)|^2 dwdz \leq (\pi T)^2 e^{3Mt} \sum_{n\geq 2} [n(n-1)]^2 \Gamma_t^{n-2} \left(\frac{2}{M}\right)^{n-2} K_M^{n-2}.$$
The conclusion follows as in the last part of the proof of Theorem \ref{th-D-in-L2}.
\end{proof}

\begin{remark}
{\rm Fix $0<r<t$ and $x \in \bR$. Taking the integral $d\theta$ on $[0,t]$ in \eqref{bpund-Bn-L2}, we infer that with probability $1$, for almost all $z\in\bR$, the map $(\theta,w) \mapsto I_{n-2}((\widetilde{f}_n(\cdot,\theta,w,r,z,t,x)))$ belongs to $L^2(\bR^2)$. We denote by $\cF_{\theta,w}[I_{n-2}((\widetilde{f}_n(\cdot,*,r,z,t,x)))]$ its Fourier transform with respect to the missing $(\theta,w)$-variable, denoted by $*$. 

Similarly, by Theorem \ref{th-D2-L2},  with probability $1$, for almost all $z\in\bR$, the map $(\theta,w) \mapsto D_{(\theta,w),(r,z)}^2 u(t,x)$ belongs to $L^2(\bR^2)$. We denote by
$\cF_{\theta,w}[D_{*,(r,z)}^2 u(t,x)]$ the Fourier transform of this function.
We claim that:
\begin{equation}
\label{chaos-Fourier-D2} \cF_{\theta,w} [D_{*,(r,z)}^2
u(t,x)]=\sum_{n\geq 2}n(n-1) \cF_{\theta,w} \big[I_{n-2}(
\widetilde{f}_n(\cdot,*,r,z,t,x))\big].
\end{equation}
Indeed, using the proof of Theorem \ref{th-D2-L2} one verifies
that the series on the right-hand side of \eqref{chaos-Fourier-D2}
converges in $L^2(\Omega;L^2(\bR))$; for this, it suffices to prove
that
\[
\sum_{n\geq 2}n(n-1) \|\cF_{\theta,w} \big[I_{n-2}(
\widetilde{f}_n(\cdot,*,r,z,t,x))\big]\|_{L^2(\Omega;L^2(\bR))}<\infty.
\]
Next, fix $N>2$ and observe that
\[
\cF_{\theta,w} \Big[ \sum_{n=2}^N n(n-1) I_{n-2}(
\widetilde{f}_n(\cdot,*,r,z,t,x))\Big] = \sum_{n=2}^N n(n-1)
\cF_{\theta,w} \big[I_{n-2}( \widetilde{f}_n(\cdot,*,r,z,t,x))\big].
\]
In order to conclude that \eqref{chaos-Fourier-D2} holds, it
suffices to prove that the left-hand side above converges to
$\cF_{\theta,w} [D_{*,(r,z)}^2 u(t,x)]$, as $N\to \infty$. This
follows by applying Plancherel's theorem and using the arguments in
the proof of Lemma \ref{lem:34}. }
\end{remark}

\begin{lemma}
\label{Fourier-lemma-D2}
For any $t>0$, $x \in \bR$ and $(t_1,x_1),\ldots,(t_{n-2},x_{n-2}),(r,z) \in [0,t] \times \bR$, the function $(\theta,w) \mapsto \widetilde{f}_n(t_1,x_1,\ldots,t_{n-2},x_{n-2},\theta,w,r,z,t,x)$ is in $L^2(\bR^2)$, If we denote by $\cF_{\theta,w}\widetilde{f}_n(t_1,x_1,\ldots,t_{n-2},x_{n-2},*,,r,z,t,x)$ the Fourier transform of this function, then
\begin{equation}
\label{Fourier-H}
\cF_{\theta,w}\widetilde{f}_n(\cdot,*,,r,z,t,x)(\tau,\xi) \in \cH^{\otimes (n-2)}
\end{equation}
for almost all $z \in \bR$ and for almost all $(\tau,\xi)\in \bR^2$.
\end{lemma}

\begin{proof}
By \eqref{fn-sum}, it is enough to prove that $(\theta,w)\mapsto f_{ij}^{(n)}(t_1,x_1,\ldots,t_{n-2},x_{n-2},\theta,w,r,z,t,x)$ is in $L^2(\bR^2)$ for any
$i,j=1,\ldots,n$ with $i\not=j$. This is clear since, if for instance
 $i<j$, $$\int_0^t \int_{\bR}G(t_j-\theta,x_j-w)G(\theta-t_{j-1},w-x_{j-1})dw d\theta<\infty.$$

We now prove \eqref{Fourier-H}. By \eqref{bpund-Bn-L2},
$\int_{0}^t \int_{\bR} \E|I_{n-2}(\widetilde{f}_n(\cdot,\theta,w,r,z,t,x))|^2 dw d\theta<\infty$, for almost all $z \in \bR$. This integral is equal to $(n-2)!$ multiplied by the following quantity:
\begin{align*}
&\int_{\bR} \int_{\bR}\|\widetilde{f}_n(\cdot,\theta,w,r,z,t,x)\|_{\cH^{\otimes (n-2)}}^2 dw d\theta\\
&\quad
=\int_{\bR^2}\int_{\bR^{2(n-2)}}|\cF_{t_1,x_1,\ldots,t_{n-2},x_{n-2}}
\widetilde{f}_n(\cdot,\theta,w,r,z,t,x)(\tau_1,\xi_1,\ldots,\tau_{n-2},\xi_{n-2})|^2\\
& \qquad \qquad \times
\prod_{i=1}^{n-2}\nu(d\tau_i)\prod_{i=1}^{n-2}\mu_i(d\xi_i) dw
d\theta
\end{align*}
where $\cF_{t_1,x_1,\ldots,t_{n-2},x_{n-2}}\widetilde{f}_n(\cdot,\theta,w,r,z,t,x)$ denotes the Fourier transform with respect to the variables $t_1,x_1,\ldots,t_{n-2},x_{n-2}$. We apply Fubini's theorem, followed by Plancherel's theorem:
\begin{align*}
&\int_{\bR^2}|\cF_{t_1,x_1,\ldots,t_{n-2},x_{n-2}}\widetilde{f}_n(\cdot,\theta,w,r,z,t,x)
(\tau_1,\xi_1,\ldots,\tau_{n-2},\xi_{n-2})|^2 dw d\theta\\
& \quad = \frac{1}{(2\pi)^2}
\int_{\bR^2}|\cF_{\theta,w}[\cF_{t_1,x_1,\ldots,t_{n-2},x_{n-2}}
\widetilde{f}_n(\cdot,*,r,z,t,x)
(\tau_1,\xi_1,\ldots,\tau_{n-2},\xi_{n-2})](\tau,\xi)|^2 d\tau d\xi,
\end{align*}
where $*$ denotes missing $(\theta,w)$-variable and $\cF_{\theta,w}$ is the Fourier transform with respect to this variable. Since $\widetilde{f}_n(\cdot,*,r,z,t,x) \in L^1(\bR^{n-1})$, we can switch the order of the variables $(t_1,x_1),\ldots,(t_{n-2},x_{n-2})$ and $(\theta,w)$ when calculating the Fourier transform, i.e.
\begin{align}
\nonumber
& \cF_{\theta,w}[\cF_{t_1,x_1,\ldots,t_{n-2},x_{n-2}}
\widetilde{f}_n(\cdot,*,r,z,t,x)
(\tau_1,\xi_1,\ldots,\tau_{n-2},\xi_{n-2})](\tau,\xi)\\
\label{two-Fourier}
& \quad \quad \quad =\cF_{t_1,x_1,\ldots,t_{n-2},x_{n-2}}[\cF_{\theta,w}
\widetilde{f}_n(\cdot,*,r,z,t,x)(\tau,\xi)](\tau_1,\xi_1,\ldots,\tau_{n-2},\xi_{n-2}) \\
\label{two-Fourier2}
& \quad \quad \quad = \cF \widetilde{f}_n(\cdot,r,z,t,x)(\tau_1,\xi_1,\ldots,\tau_{n-2},\xi_{n-2},\tau,\xi).
\end{align}
Applying Fubini's theorem again, we conclude that for almost all $z \in \bR$,
\begin{align*}
& \int_{\bR^2} \int_{\bR^{2(n-2)}}|\cF_{t_1,x_1,\ldots,t_{n-2},x_{n-2}}[\cF_{\theta,w}
\widetilde{f}_n(\cdot,*,r,z,t,x)(\tau,\xi)](\tau_1,\xi_1,\ldots,\tau_{n-2},\xi_{n-2})|^2 \\
& \qquad \qquad \qquad \times
\prod_{i=1}^{n-2}\nu(d\tau_i)\prod_{i=1}^{n-2}\mu_i(d\xi_i) d\tau
d\xi <\infty.
\end{align*}
Hence, for almost all $z \in \bR$ and for almost all $(\tau,\xi) \in \bR^2$, the integral
\begin{align}
\label{norm-H-Fourier}
&\|\cF_{\theta,w} \big[\tilde{f}_n(\cdot,*,r,z,t,x) \big]
(\tau,\xi)\big) \|^2_{\cH^{\otimes (n-2)}}\\
\nonumber & \quad =
\int_{\bR^{2(n-2)}}|\cF_{t_1,x_1,\ldots,t_{n-2},x_{n-2}}[\cF_{\theta,w}
\widetilde{f}_n(\cdot,*,r,z,t,x)(\tau,\xi)](\tau_1,\xi_1,\ldots,\tau_{n-2},\xi_{n-2})|^2
\\
& \qquad \qquad \times
\prod_{i=1}^{n-2}\nu(d\tau_i)\prod_{i=1}^{n-2}\mu_i(d\xi_i)
\end{align}
is finite. Relation \eqref{Fourier-H} follows using a characterization of the space $\cH^{\otimes (n-2)}$ (in terms of the integral of the squared Fourier transform) similar to the one given by Theorem \ref{jolis-th} for space $\cH$.
\end{proof}

The next result shows that in our context, we can interchange the order of the Fourier transform and the multiple integral, a fact which is highly non-trivial in the general case. Lemma \ref{Fourier-lemma-D2} shows that the integral $I_{n-2}\big(\cF_{\theta,w}\widetilde{f}_{n}(\cdot,*,r,z,t,x)(\tau,\xi)\big)$ is well-defined.

\begin{lemma}
\label{lem-G-eq} With probability $1$, for all $r\in [0,t]$ and for
almost all $z \in \bR$ and $(\tau,\xi)\in \bR^2$,
$$\cF_{\theta,w}[I_{n-2}(\widetilde{f}_{n}(\cdot,*,r,z,t,x))](\tau,\xi)=
I_{n-2}\big(\cF_{\theta,w}\widetilde{f}_{n}(\cdot,*,r,z,t,x)(\tau,\xi)\big).$$
\end{lemma}

\begin{proof} Let $h_1(\tau,\xi)=\cF_{\theta,w}[I_{n-2}(\widetilde{f}_{n}(\cdot,*,r,z,t,x))](\tau,\xi) $ and $h_2(\tau,\xi)=I_{n-2}(\cF_{\theta,w}\widetilde{f}_{n}(\cdot,*,\linebreak r,z,t,x)(\tau,\xi))$.
For an arbitrary function $g \in L^2(\bR^2)$, we denote
$$X_g =\int_{\bR^2} h_1(\tau,\xi)\overline{\cF g(\tau,\xi)}d\tau d\xi  \quad \mbox{and} \quad
Y_g =\int_{\bR^2} h_2(\tau,\xi)\overline{\cF g(\tau,\xi)}d\tau d\xi.$$
We will show that for any $g \in L^2(\bR^2)$ fixed,
\begin{equation}
\label{G-equality}
\E(G X_g)=\E(G Y_g) \quad \mbox{for all} \quad G \in L^2(\Omega),
\end{equation}
which implies that $X_g=Y_g$ a.s., the negligible set depending on $g$. Since $L^2(\bR^2)$ is separable, there exists a countable dense set $S$ in $L^2(\bR^2)$. Hence, there exists an event $\Omega_0$ of probability $1$ on which $X_g=Y_g$ for any $g \in S$. Since $\{\cF g;g \in S\}$ is dense in $L^2(\bR^2)$, it follows that on $\Omega_0$, $\langle h_1, \phi \rangle_{L^2(\bR^2)} =\langle h_2, \phi \rangle_{L^2(\bR^2)}$ for all $\phi \in L^2(\bR^2)$, and hence $h_1=h_2$ a.e.

First, we prove that $X_g \in \cH_{n-2}$. (A similar argument shows
that $Y_g \in \cH$.) For this, it suffices to prove that $\E
[I_m(h_m)X_g]=0$, for all $m\neq n-2$ and all symmetric $h_m\in
\cH^{\otimes m}$. By Plancherel's theorem,
\begin{equation}
\label{def-Xg}
X_g=(2\pi)\int_{\bR^2}I_{n-2}(\widetilde{f}_{n}(\cdot,\theta,w,r,z,t,x))g(\theta,w) dw d\theta,
\end{equation}
and hence,
\begin{equation}\label{eq:82}
 \E [I_m(h_m)X_g] = (2\pi)\E \left[\int_{\bR^2} I_m(h_m) I_{n-2}(\widetilde{f}_{n}(\cdot,\theta,w,r,z,t,x))g(\theta,w) dw d\theta\right].
\end{equation}
Now, we would like to apply Fubini theorem, for which we need to check that
\[
 \E \int_{\bR^2} |I_m(h_m) I_{n-2}(\widetilde{f}_{n}(\cdot,\theta,w,r,z,t,x))g(\theta,w)| dw d\theta<\infty.
\]
But we have, by Fubini theorem (because the integrand is now non-negative) and Cauchy-Schwarz inequality (applied twice),
\begin{align*}
 & \E \int_{\bR^2} |I_m(h_m) I_{n-2}(\widetilde{f}_{n}(\cdot,\theta,w,r,z,t,x))g(\theta,w)| dw d\theta \\
 &\quad = \int_{\bR^2} \E[|I_m(h_m) I_{n-2}(\widetilde{f}_{n}(\cdot,\theta,w,r,z,t,x))|] \, |g(\theta,w)| dw d\theta\\
 & \quad \leq \int_{\bR^2}  \big(\E[I_m(h_m)^2]\big)^\frac12 \big(\E[I_{n-2}(\widetilde{f}_{n}(\cdot,\theta,w,r,z,t,x))^2]\big)^\frac12 \, |g(\theta,w)| dw d\theta\\
 & \quad = (m!)^{1/2} \|h_m\|_{\cH^{\otimes m}} \int_{\bR^2}  \big(\E[I_{n-2}(\widetilde{f}_{n}(\cdot,\theta,w,r,z,t,x))^2]\big)^\frac12 \, |g(\theta,w)| dw d\theta\\
 & \quad \leq  (m!)^{1/2} \|h_m\|_{\cH^{\otimes m}} \left(\int_{\bR^2} \E[I_{n-2}(\widetilde{f}_{n}(\cdot,\theta,w,r,z,t,x))^2] dw d\theta \right)^\frac12 \, \|g\|_{L^2(\bR^2)}<\infty,
 \end{align*}
where we have applied \eqref{bpund-Bn-L2}.
By Fubini's theorem, we obtain that for any $m\not=n-2$,
$$\E [I_m(h_m)X_g] =  \int_{\bR^2} \E[I_m(h_m) I_{n-2}(\widetilde{f}_{n}(\cdot,\theta,w,r,z,t,x))] g(\theta,w) dw d\theta=0.$$

We now prove \eqref{G-equality}. Since both $X_g$ and $Y_g$ are in $\cH_{n-2}$,
it is enough to assume that $G=I_{n-2}(g_{n-2})$ for some symmetric $g_{n-2} \in \cH^{\otimes (n-2)}$.
Using Fubini's theorem and \eqref{def-Xg},
\begin{align*}
\E(G X_g)&=(2\pi)\int_{\bR^2}\E\left[I_{n-2}(g_{n-2})
I_{n-2}(\widetilde{f}_{n}(\cdot,\theta,w,r,z,t,x))\right]g(\theta,w)d\theta dw \\
&=(2\pi)\int_{\bR^2}(n-2)!\langle g_{n-2},\widetilde{f}_{n}(\cdot,\theta,w,r,z,t,x) \rangle_{\cH^{\otimes (n-2)}}g(\theta,w)d\theta dw
\end{align*}
and
\begin{align*}
& \E(G Y_g)\\
&\; =\int_{\bR^2}\E [I_{n-2}(g_{n-2})I_{n-2}(\cF_{\theta,w}\widetilde{f}_{n}(\cdot,*, r,z,t,x)(\tau,\xi))]\overline{\cF g(\tau,\xi)}d\tau d\xi  \\
& \; = \int_{\bR^2} (n-2)! \langle g_{n-2}, \cF_{\theta,w}\widetilde{f}_{n}(\cdot,*, r,z,t,x)(\tau,\xi))\rangle_{\cH^{\otimes (n-2)}}
\overline{\cF g(\tau,\xi)}d\tau d\xi  \\
& \; =(n-2)! \int_{\bR^2} \int_{\bR^{2(n-2)}} \cF_{t_1,x_1,\ldots,t_{n-2},x_{n-2}} (\cF_{\theta,w}\widetilde{f}_{n}(\cdot,*, r,z,t,x)(\tau,\xi))(\tau_1,\xi_1,\ldots,\tau_{n-2},\xi_{n-2}) \\
 & \quad  \qquad \overline{\cF g_{n-2}(\tau_1,\xi_1,\ldots,\tau_{n-2},\xi_{n-2})}\,\,
\overline{\cF g(\tau,\xi)}\,\prod_{i=1}^{n-2}\nu(d\tau_i) \prod_{i=1}^{n-2}\mu(d\xi_i)d\tau d\xi.
\end{align*}
We use \eqref{two-Fourier}, followed by Fubini's theorem. For the $d\tau d\xi$ integral, we use Plancherel's theorem to eliminate the Fourier transform in the $(\theta,w)$-variable. We obtain:
\begin{align*}
\E(G Y_g)&= (n-2)!  \int_{\bR^{2(n-2)}}\Big(2\pi\int_{\bR^2} \cF_{t_1,x_1,\ldots,t_{n-2},x_{n-2}} \widetilde{f}_{n}(\cdot,\theta,w, r,z,t,x)(\tau_1,\xi_1,\ldots,\tau_{n-2},\xi_{n-2}) \Big.\\
& \qquad \qquad \qquad \qquad \Big.g(\theta,w)d\theta dw\Big) \, \overline{\cF g_{n-2}(\tau_1,\xi_1,\ldots,\tau_{n-2},\xi_{n-2})} \, \prod_{i=1}^{n-2}\nu(d\tau_i) \prod_{i=1}^{n-2}\mu(d\xi_i) \\
&=(n-2)!(2\pi)\int_{\bR^2} \langle g_{n-2},\widetilde{f}_n(\cdot,\theta,w,r,z,t,x) \rangle_{\cH^{\otimes (n-2)}} g(\theta,w)d\theta dw.
\end{align*}
This proves \eqref{G-equality}.
\end{proof}


In summary, relation \eqref{chaos-Fourier-D2} and Lemma
\ref{lem-G-eq} show that for all $r\in [0,t]$ and almost all $z \in
\bR$ and for almost all $(\tau,\xi) \in \bR^2$,
\begin{align}
\nonumber
\cF_{\theta,w} \big[
D^2_{*,(r,z)} u(t,x)\big](\tau,\xi)&=\sum_{n\geq 2}n(n-1)\cF_{\theta,w}
\Big[ I_{n-2}(\widetilde{f}_n(\cdot,*,r,z,t,x))  \Big] (\tau,\xi)\\
\label{chaos-D2}
&=\sum_{n\geq 2}n(n-1)I_{n-2}\big(\cF_{\theta,w} \widetilde{f}_n(\cdot,*,r,z,t,x)
 (\tau,\xi)\big),
\end{align}
which gives the chaos expansion of the Fourier transform of
$D^2_{*,(r,z)} u(t,x)$.

\begin{theorem}
\label{H-L2-norm}
Let $D_r^2 u(t,x)$ denote $D_{*,(r,\cdot)}^2$,  where $*$ is the missing $(\theta,w)$-variable and $\cdot$ is the missing $z$-variable.
For any $T>0$,
$$\sup_{(t,x) \in [0,T] \times \bR}\sup_{r \in [0,t]}\E\|D_r^2u(t,x)\|_{\cH \otimes L^2(\bR)}^2\leq C_T'',$$
where $C_T''>0$ is a constant depending on $T$.
\end{theorem}

\begin{proof}
By \eqref{chaos-D2}, we have:
\begin{align*}
& \E \|D^2_r u(t,x)\|^2_{\cH\otimes L^2(\bR)}= \E \int_\bR
\|D_{*,(r,z)}^2u(t,x)\|_\cH^2 dz \\
& \; = \E \int_\bR \int_{\bR^2} |\cF_{\theta,w} \big[
D^2_{*,(r,z)} u(t,x)\big] (\tau,\xi) |^2
\nu(d\tau)\mu(d\xi) dz\\
& \; =\sum_{n\geq 2} n^2(n-1)^2 (n-2)!\int_\bR \int_{\bR^2} \|
\cF_{\theta,w} \widetilde{f}_n(\cdot,*,r,z,t,x)
(\tau,\xi)\big) \|^2_{\cH^{\otimes (n-2)}} \nu(d\tau)\mu(d\xi) dz.
\end{align*}
Using relation \eqref{norm-H-Fourier} for
expressing $\|\cF_{\theta,w} \widetilde{f}_n(\cdot,*,r,z,t,x)
(\tau,\xi)\big) \|^2_{\cH^{\otimes (n-2)}}$, followed by \eqref{two-Fourier2}, we see that
$$\int_{\bR^2} \|
\cF_{\theta,w} \widetilde{f}_n(\cdot,*,r,z,t,x)
(\tau,\xi)\big) \|^2_{\cH^{\otimes (n-2)}} \nu(d\tau)\mu(d\xi)=\|\widetilde{f}_n(\cdot,r,z,t,x)\|_{\cH^{\otimes (n-1)}}^2.$$
Hence,
$$\E \|D^2_r u(t,x)\|^2_{\cH\otimes L^2(\bR)}\leq \sum_{n\geq 2} n^2(n-1)^2 (n-2)!
\int_{\bR}\|\widetilde{f}_n(\cdot,r,z,t,x)\|_{\cH^{\otimes (n-1)}}^2 dz.$$
By \eqref{2-moment-I}, for any $t \in [0,T]$, $x \in \bR$, $r \in [0,t]$ and $M>0$,
$$n^2 (n-1)!\int_{\bR}\|\widetilde{f}_n(\cdot,r,z,t,x))\|_{\cH^{\otimes (n-1)}}^2 dz \leq \pi t n^2 \Gamma_t^{n-1} e^{2M t} \left(\frac{2}{M} \right)^{n-1}K_{M}^{n-1},$$
and hence,
$$\E \|D^2_r u(t,x)\|^2_{\cH\otimes L^2(\bR)}\leq \pi t e^{2M t} \sum_{n\geq 2}n^2 (n-1) \Gamma_t^{n-1} \left(\frac{2}{M} \right)^{n-1}K_{M}^{n-1}.$$

We use the fact that $n \leq e^n$ for all $n \geq 1$.
We conclude that for any $M>0$, $t \in [0,T]$, $x \in \bR$ and $r \in [0,t]$,
$$\E \|D^2_r u(t,x)\|^2_{\cH\otimes L^2(\bR)}\leq \pi T e^{2M T} e^3 \sum_{n\geq 2} \left( e^3 \Gamma_T  \frac{2}{M} K_{M}\right)^{n-1}.$$
Choose $M=M_T'>2$ large enough such that
$e^3 \Gamma_T K_{M_T'} < 1/2$. Then
\begin{equation}
\label{L2-D2-norm}
\E\|D^2_r u(t,x)\|^2_{\cH\otimes L^2(\bR)}\leq \pi T e^{2M_T' T}e^3 \sum_{n\geq 1}\left(\frac{1}{2} \right)^{n-1}.
\end{equation}
The conclusion follows with $C_T''=2\pi T e^{2M_T' T}e^3$.
\end{proof}

\section{Equation satisfied by $D_{r,\cdot}u(t,x)$}
\label{section-eq-D}

In this section, we show that $D_{r,\cdot}u(t,x)$ satisfies an integral equation.

By Theorem 6.7 of \cite{balan12}, we know that Malliavin derivative $Du$ of the solution satisfies the following equation in $L^2(\Omega;\cH)$:
\begin{equation}
\label{Malliavin-eq}
D_{\cdot}u(t,x)=G(t-\cdot,x-\cdot)u(\cdot)+\int_0^t \int_{\bR}U(s,y)W(\delta^* s,\delta^* y)
\end{equation}
where $U=U^{(t,x)} \in L^2(\Omega;\cH \otimes \cH)$ belongs to the domain of the $\delta^*$ and is given by:
$$U((r',z'),(s,y))=G(t-s,x-y) D_{r',z'} u(s,y).$$
Here $\delta^*$ is the Skorohod integral with values in $\cH$.

In what follows we will show that $D_{r,\cdot}u$ satisfies an equation similar to \eqref{Malliavin-eq}, but in $L^2(\Omega;L^2(\bR))$ for fixed $r \in [0,t]$.

 Let $\overline{\delta}$ be the $L^2(\bR)$-valued Skorohod integral defined in Section 6 of \cite{balan12} with $\cA=L^2(\bR)$. By Proposition 6.2 of \cite{balan12}, $\bD^{1,2}(\cH \otimes L^2(\bR)) \subset {\rm Dom} \ \overline{\delta}$.
Let $(u_n)_{n\geq 0}$ be the sequence of Picard iterations given by: $u_0(t,x)=1$ and
$$u_{n}(t,x)=1+\sum_{k=1}^{n}I_k(f_k(\cdot,t,x)). $$

We fix $t>0$, $x \in \bR^d$ and $r \in [0,t]$. For any $s \in [0,t]$, $y \in \bR^d$ and $z \in \bR^d$, we let
\begin{align*}
K^{(r)}\big((s,y),z\big)&=1_{[0,t]}(s)G(t-s,x-y)D_{r,z}u(s,y),\\
K_n^{(r)}\big((s,y),z\big)&=1_{[0,t]}(s)G(t-s,x-y)D_{r,z}u_n(s,y).
\end{align*}

We denote
\begin{equation}
\label{def-phi} \phi(t):=\int_{[0,t]^2}
\int_{\bR^{2}}G(s,y)G(s',y')\gamma(s-s')f(y-y')dy dy' ds
ds'=\alpha_1(t),
\end{equation}
where we recall that $\alpha_1(t)$ has been defined in
\eqref{def-alpha-n}.

The following result is similar to Lemma 6.5 of \cite{balan12}.

\begin{lemma}
\label{K-in-D}
For any $r \in [0,t]$, $K^{(r)} \in \bD^{1,2}(\cH \otimes L^2(\bR))$ and $K_n^{(r)} \in \bD^{1,2}(\cH \otimes L^2(\bR))$. Hence,
$K^{(r)}$ and $K_n^{(r)}$ belong to ${\rm Dom} \ \overline{\delta}$, for any $r \in [0,t]$.
\end{lemma}

\begin{proof} To prove that $K^{(r)} \in \bD^{1,2}(\cH \otimes L^2(\bR))$ it suffices to show that
$\|K^{(r)}\|_{\bD^{1,2}(\cH \otimes L^2(\bR))}^2<\infty$, where
$$\|K^{(r)}\|_{\bD^{1,2}(\cH \otimes L^2(\bR))}^2=\E\|K^{(r)}\|_{\cH \otimes L^2(\bR)}^2+\E\|DK^{(r)}\|_{\cH \otimes \cH \otimes L^2(\bR)}^2.$$
It suffices to show that both terms above are finite.
Note that
$\|K^{(r)}\|_{\cH\otimes L^2(\bR)}\leq \|K^{(r)}\|_{|\cH|\otimes
L^2(\bR)}$.
Applying Cauchy-Schwarz inequality twice and Theorem \ref{th-D-in-L2}, we have
\begin{align}
\nonumber
& \E\|K^{(r)}\|_{|\cH|\otimes L^2(\bR)}^2 \\
\nonumber
& \; = \E \int_{[0,t]^2} \int_{\bR^{2}}G(t-s,x-y)G(t-s',x-y')
\big|\langle D_{r,\cdot}u(s,y), D_{r,\cdot}u(s',y')\rangle_{L^2(\bR)}\big|\\
\nonumber
& \qquad \qquad \qquad \times \gamma(s-s')f(y-y')dydy'dsds' \\
\nonumber
& \; \leq \E \int_{[0,t]^2}
\int_{\bR^{2}}G(t-s,x-y)G(t-s',x-y')
\|D_{r,\cdot}u(s,y)\|_{L^2(\bR)}  \|D_{r,\cdot}u(s',y')\|_{L^2(\bR)}\\
\nonumber
& \qquad \qquad \qquad \times \gamma(s-s')f(y-y')dydy'dsds' \\
\nonumber
& \; \leq \int_{[0,t]^2} \int_{\bR^{2}}G(t-s,x-y)G(t-s',x-y')\Big(\E \| D_{r,\cdot}u(s,y)\|_{L^2(\bR)}^2 \Big)^{1/2}  \Big(\E \| D_{r,\cdot}u(s',y')\|_{L^2(\bR)}^2 \Big)^{1/2}\\
\nonumber
& \qquad \qquad \qquad \times \gamma(s-s')f(y-y')dydy'dsds' \\
\label{bound-K-r}
& \; \leq C_t \phi(t)<\infty.
\end{align}
A similar calculation shows that $E\|K_n^{(r)}\|_{|\cH| \otimes
L^2(\bR)}^2<C_t \phi(t)$.
For the second term, we proceed similarly. First, note that
$\|DK^{(r)}\|_{\cH \otimes \cH\otimes L^2(\bR)}\leq \|DK^{(r)}\|_{|\cH|\otimes \cH \otimes
L^2(\bR)}$.
Applying Cauchy-Schwarz inequality twice and Theorem \ref{H-L2-norm}, we have
\begin{align}
\nonumber
& \E \|DK^{(r)}\|_{|\cH|\otimes \cH \otimes L^2(\bR)}^2 \\
\nonumber
& \; = \E \int_{[0,t]^2} \int_{\bR^{2}}G(t-s,x-y)G(t-s',x-y')
\big|\langle D_{r}^2 u(s,y), D_{r}^2 u(s',y')\rangle_{\cH \otimes L^2(\bR)}\big|\\
\nonumber
& \qquad \qquad \qquad \times \gamma(s-s')f(y-y')dydy'dsds' \\
\nonumber
& \; \leq \E \int_{[0,t]^2}
\int_{\bR^{2}}G(t-s,x-y)G(t-s',x-y')
\|D_{r}^2 u(s,y)\|_{\cH \otimes L^2(\bR)}  \|D_{r}^2 u(s',y')\|_{\cH \otimes L^2(\bR)}\\
\nonumber
& \qquad \qquad \qquad \times \gamma(s-s')f(y-y')dydy'dsds' \\
\nonumber
& \; \leq \int_{[0,t]^2} \int_{\bR^{2}}G(t-s,x-y)G(t-s',x-y')
\Big(\E \| D_{r}^2 u(s,y)\|_{\cH \otimes L^2(\bR)}^2 \Big)^{1/2}\\
\nonumber
& \qquad \qquad \times \Big(\E \| D_{r}^2 u(s',y')\|_{\cH \otimes L^2(\bR)}^2 \Big)^{1/2} \gamma(s-s')f(y-y')dydy'dsds' \\
\label{bound-DK-r} & \; \leq C_t'' \phi(t)<\infty.
\end{align}
A similar calculation shows that $\E\|DK_n^{(r)}\|_{|\cH| \otimes \cH \otimes L^2(\bR)}^2<\infty$.
\end{proof}

We use the following more suggestive notation:
$$\overline{\delta}(K^{(r)})=\int_0^t \int_{\bR}G(t-s,x-y)D_{r,\cdot}u(s,y)W(\overline{\delta} s, \overline{\delta}y)$$
$$\overline{\delta}(K_{n}^{(r)})=\int_0^t \int_{\bR}G(t-s,x-y)D_{r,\cdot}u_{n}(s,y)W(\overline{\delta} s, \overline{\delta}y).$$
By Lemma \ref{K-in-D}, these stochastic integrals are well-defined.

\medskip

The following result establishes a recursive relation involving the
Malliavin derivative of the Picard iteration scheme (see Proposition
6.6 of \cite{balan12} for a similar result).

\begin{proposition}
\label{analogue-prop6-6}
For any $r \in [0,t]$, the following equality holds in $L^2(\Omega;L^2(\bR))$:
\begin{equation}
\label{eq:9}
D_{r,\cdot}u_n(t,x)=G(t-r,x-\cdot)u_{n-1}(r,\cdot)+\int_0^t \int_{\bR}G(t-s,x-y)D_{r,\cdot}u_{n-1}(s,y)W(\overline{\delta} s, \overline{\delta}y).
\end{equation}
\end{proposition}

\begin{proof} For clarity, we split the proof into several steps.

{\em Step 1.} To prove \eqref{eq:9}, we will show (in {\em Step 2} below) that for
all $\varphi\in L^2(\bR)$ fixed, the following equality holds in $L^2(\Omega)$:
\begin{align}
\langle \varphi,D_{r,\cdot} u_n(t,x)\rangle_{L^2(\bR)} & = \langle
\varphi,G(t-r,x-\cdot) u_{n-1}(r,\cdot)\rangle_{L^2(\bR)} \nonumber \\
& \qquad + \int_0^t \int_{\bR} G(t-s,x-y) \langle \varphi, D_{r,\cdot}
u_{n-1}(s,y)\rangle_{L^2(\bR)} W({\delta} s,{\delta} y). \label{eq:10}
\end{align}

The integral on the right-hand side of \eqref{eq:10} is a
real-valued Skorohod integral. In this step of the proof, we show
that this integral is well-defined. For this, let
$$v_n(s,y)=G(t-s,x-y)\langle \varphi, D_{r,\cdot}
u_{n-1}(s,y)\rangle_{L^2(\bR)}.$$ Since $\bD^{1,2}(\cH) \subset {\rm
Dom} \ \delta$, it is enough to prove that $v_n \in \bD^{1,2}(\cH)$,
i.e. $\E\|v_n\|_{\cH}^2<\infty$ and $\E\|D v_n\|_{\cH \otimes
\cH}^2<\infty$. By Cauchy-Schwarz inequality
and  Theorem \ref{th-D-in-L2}, it follows that
\begin{align*}
\E\|v_n\|_{\cH}^2  & \leq
\E\|v_n\|_{|\cH|}^2 \\
& =\E \int_{([0,t] \times \bR)^2} \gamma(s-s')f(y-y') G(t-s,x-y)G(t-s',x-y')\\
 & \qquad \qquad \quad \times |\langle \varphi, D_{r,\cdot} u_{n-1}(s,y)\rangle_{L^2(\bR)}| \cdot | \langle \varphi, D_{r,\cdot} u_{n-1}(s',y')\rangle_{L^2(\bR)}| dydy' dsds'\\
&
\leq \|\varphi\|_{L^2(\bR)}^2 C_t \phi(t)<\infty.
\end{align*}

We now prove that:
\begin{equation}
\label{D-vn-finite} \E\|D v_n\|_{\cH \otimes \cH}^2<\infty.
\end{equation}
Note that
\begin{align*}
D_{\theta,w}v_n(s,y) & =
 \int_\bR G(t-s,x-y) \varphi(z) D^2_{(\theta,w),(r,z)}
u_{n-1}(s,y) dz.
\end{align*}
Then, we have
\begin{align*}
\E\|D v_n\|_{\cH \otimes \cH}^2 & \leq \E\|D v_n\|_{|\cH| \otimes
\cH}^2\\
& = \E \int_{([0,t] \times \bR)^2} \gamma(s-s')f(y-y')
G(t-s,x-y)G(t-s',x-y') \\
& \quad\times \left|\left\langle \int_\bR \varphi(z)
D^2_{*,(r,z)}u_{n-1}(s,y) dz,\int_\bR \varphi(z)
D^2_{*,(r,z)}u_{n-1}(s',y') dz\right\rangle_\cH \right| ds ds' dy'
dy'.
\end{align*}
By Cauchy-Schwarz inequality, in order to have \eqref{D-vn-finite}
it suffices to check that
\[
\sup_{(s,y)\in [0,t]\times \bR} \E \left\|\int_\bR \varphi(z)
D^2_{*,(r,z)}u_{n-1}(s,y) dz\right\|^2_\cH <\infty.
\]
Observe that
\begin{align*}
& \E \left\|\int_\bR \varphi(z) D^2_{*,(r,z)}u_{n-1}(s,y)
dz\right\|^2_\cH \\
& \qquad = \E \int_{\bR^2} \big|\cF_{\theta,w} \left[\int_\bR
\varphi(z) D^2_{*,(r,z)}u_{n-1}(s,y) dz\right](\tau,\xi) \big|^2
\nu(d\tau)\mu(d\xi).
\end{align*}
One verifies that, with probability 1, for all $r\in [0,t]$ and
$(s,y)\in [0,t]\times \bR$, the function $(\theta,w)\to \int_\bR
\varphi(z) D^2_{(\theta,w),(r,z)}u_{n-1}(s,y) dz$ belongs to
$L^2(\bR^2)$ and
\[
\cF_{\theta,w} \left[\int_\bR \varphi(z) D^2_{*,(r,z)}u_{n-1}(s,y)
dz\right] = \int_\bR
\varphi(z)\cF_{\theta,w}\big[D^2_{*,(r,z)}u_{n-1}(s,y)\big] dz.
\]
Indeed, for the latter equality to be fulfilled, it suffices to show
the equality of the respective inner products against an arbitrary
element of $L^2(\bR^2)$. Thus, by Cauchy-Schwarz inequality and
Fubini theorem,
\begin{align*}
& \E \left\|\int_\bR \varphi(z) D^2_{*,(r,z)}u_{n-1}(s,y)
dz\right\|^2_\cH \\
& \quad  = \E \int_{\bR^2} \Big| \int_\bR \varphi(z) \cF_{\theta,w}
\big[D^2_{*,(r,z)}u_{n-1}(s,y)\big](\tau,\xi) dz \Big|^2
\nu(d\tau)\mu(d\xi)\\
& \quad \leq \|\varphi\|^2_{L^2(\bR)} \E \int_{\bR^2}\int_\bR
\big|\cF_{\theta,w}
\big[D^2_{*,(r,z)}u_{n-1}(s,y)\big](\tau,\xi)\big|^2 dz \nu(d\tau)\mu(d\xi)\\
& \quad = \|\varphi\|^2_{L^2(\bR)} \E \int_\bR
\|D^2_{*,(r,z)}u_{n-1}(s,y)\|^2_\cH dz \\
& \quad \leq \|\varphi\|^2_{L^2(\bR)} \sup_{(t,x)\in [0,t]\times
\bR} \sup_{r\in [0,t]} \E \|D^2_r u_{n-1}(t,x)\|^2_{\cH\otimes
L^2(\bR)},
\end{align*}
and this is finite by Theorem \ref{H-L2-norm}. This concludes the
proof of \eqref{D-vn-finite}.

\medskip

{\em Step 2.} We now give the proof of \eqref{eq:10}. By \eqref{series-D} and \eqref{def-An},
$D_{r,z}u_n(t,x)=\sum_{k=1}^{n}A_k(r,z,t,x)$
with $A_n(r,z,t,x)=\sum_{j=1}^{n}A_{j}^{(n)}(r,z,t,x)$ and
\begin{align*}
A_{j}^{(n)}(r,z,t,x)&=\int_{0<t_1<\ldots<t_{j-1}<r<t_j<\ldots<t_{n-1}<t}
 G(t-t_{n-1},x-x_{n-1})\ldots  G(t_j-r,x_j-z) \\
&G(r-t_{j-1},z-x_{j-1})\ldots G(t_2-t_1,x_2-x_1)W(dt_1,dx_1)\ldots
W(dt_{n-1},dx_{n-1}),
\end{align*}
where the integrals are taken on the set $([0,t]\times \bR)^n$.

We take $n=3$ for simplicity. The general case is similar. We drop $(t,x)$ from the notation. We proceed as in Step 2 in the proof of Proposition 6.6 in \cite{balan12}. We have
\begin{align*}
\langle \varphi, D_{r,\cdot}u_3(t,x) \rangle_{L^2(\bR)} &=\int_{\bR} D_{r,z}u_3(t,x)\varphi(z)dz\\
&= \int_{\bR} \Big(A_1(r,z)+\sum_{j=1}^{2}A_j^{(2)}(r,z)+\sum_{j=1}^{3}A_j^{(3)}(r,z)\Big)\varphi(z)dz.
\end{align*}
In this sum, we put together the terms which have the common factor $G(t-z,x-z)$:
\begin{align*}
A_1(r,z)&=G(t-r,x-z),\\
A_2^{(2)}(r,z)&=G(t-r,x-z) \int_{t_1<r} G(r-t_1,z-x_1)W(dt_1,dx_1)
\end{align*}
and
\begin{align*}
& A_3^{(3)}(r,z)\\
&\quad =G(t-r,x-z) \int_{t_1<t_2<r}
G(r-t_2,z-x_2)G(t_2-t_1,x_2-x_1)W(dt_1,dx_1)W(dt_2,dx_2).
\end{align*}
It follows that for every $z \in \bR$ fixed,
\begin{align*}
A_1(r,z)+A_2^{(2)}(r,z)+A_3^{(3)}(r,z)=G(t-r,x-z)u_2(r,z).
\end{align*}
Multiplying by $\varphi(z)$ and taking the integral $dz$ on $\bR$, we see that
$$\int_{\bR^{2}}\Big(A_1^{(1)}(r,z)+A_2^{(2)}(r,z)+A_3^{(3)}(r,z)\Big) \varphi(z)dz=
\langle \varphi, G(t-r,x-\cdot)u_2(r,\cdot)\rangle_{L^2(\bR)}.$$

For the remaining terms, we perform some changes of notation. If we denote $(t_1,x_1)$ by $(s,y)$, then
\begin{align*}
A_1^{(2)}(r,z)&= \int_{r<t_1<t} G(t-t_1,x-x_1) G(t_1-r,x_1-z) W(dt_1,dx_1)\\
&= \int_{r<s<t} G(t-s,x-y) G(s-r,y-z) W(ds,dy).
\end{align*}

Similarly, if we denote $(t_2,x_2)$ by $(s,y)$, then
\begin{align*}
& A_1^{(3)}(r,z)\\
&\; = \int_{r<t_1<t_2<t} G(t-t_2,x-x_2) G(t_2-t_1,x_2-x_1)G(t_1-r,x_1-z) W(dt_1,dx_1)W(dt_2,dx_2)\\
&\;=\int_{r<t_1<s<t} G(t-s,x-y) G(s-t_1,y-x_1)G(t_1-r,x_1-z) W(dt_1,dx_1)W(ds,dy),\\
& A_2^{(3)}(r,z)\\
&\; = \int_{t_1<r<t_2<t} G(t-t_2,x-x_2) G(t_2-r,x_2-z)G(r-t_1,z-x_1) W(dt_1,dx_1)W(dt_2,dx_2)\\
&\;=\int_{t_1<r<s<t} G(t-s,x-y) G(s-r,y-z)G(r-t_1,z-t_1)
W(dt_1,dx_1)W(ds,dy).
\end{align*}

With this change of notation, $G(t-s,x-y)$ becomes a common factor in the integrals $W(ds,dy)$ above. The remaining factors in these integrals are easily recognized as $A_1(r,z,s,y)$, $A_1^{(2)}(r,z,s,y)$, respectively $A_2^{(2)}(r,z,s,y)$. We obtain that, for any $z \in \bR$ fixed,
\begin{align*}
& A_1^{(2)}(r,z)+A_1^{(3)}(r,z)+A_2^{(3)}(r,z)
\\
&=\int_r^t G(t-s,x-y) \Big(A_1(r,z,s,y)+A_1^{(2)}(r,z,s,y)+A_2^{(2)}(r,z,s,y)\Big) W(ds,dy) \\
&= \int_r^t G(t-s,x-y) D_{r,z}u_2(s,y)W(ds,dy)\\
&= \int_r^t G(t-s,x-y) D_{r,z}u_2(s,y)W(\delta s,\delta y).
\end{align*}
We multiply by $\varphi(z)$ and we take the integral $dz$ on $\bR$. Using stochastic Fubini theorem, we obtain:
\begin{align*}
& \int_{\bR}\Big(A_1^{(2)}(r,z)+A_1^{(3)}(r,z)+A_2^{(3)}(r,z)\Big) \varphi(z)dz\\
& \quad = \int_r^t G(t-s,x-y) \langle \varphi, D_{r,\cdot} u_2(s,y)
\rangle_{L^2(\bR)} W(\delta s,\delta y).
\end{align*}
This concludes the proof of \eqref{eq:10} for $n=3$.

{\em Step 3.} Finally, we give the proof of
\eqref{eq:9}, based on \eqref{eq:10}. By the duality relation characterizing $\overline{\delta}$, we have to prove that for any $F \in \bD^{1,2}(L^2(\bR))$,
$$\E \langle F, D_{r,\cdot} u_n(t,x)- G(t-r,x-\cdot) u_{n-1}(r,\cdot)\rangle_{L^2(\bR)}
= \E \langle DF,K^{(r)}_{n-1}\rangle_{\cH\otimes L^2(\bR)}.$$
Using a density argument, it is enough to show that this holds for
$F=F_0 \varphi$ with $\varphi\in L^2(\bR)$ and $F_0$ a smooth random
variable. By \eqref{eq:10}, and the duality
relation characterizing $\delta$,
\begin{align*}
& \E \langle F, D_{r,\cdot} u_n(t,x)- G(t-r,x-\cdot)
u_{n-1}(r,\cdot)\rangle_{L^2(\bR)} \\
& \quad = \E[F_0 \langle \varphi , D_r u_n(t,x)- G(t-r,x-\cdot)
u_{n-1}(r,\cdot \rangle_{L^2(\bR)} ] \\
& \quad = \E\left[F_0 \delta(v_n) \right]= \E [\langle DF_0,v_n \rangle_\cH ]\\
& \quad = \E \int_{([0,t] \times \bR)^2}  v_n(s,y) (D_{s',y'}F_0) \gamma(s-s')f(y-y') dydy'dsds'\\
& \quad = \E \int_{[0,t]^2} \int_{\bR^2}\int_{\bR}G(t-s,x-y) D_{r,z}u_{n-1}(s,y) \varphi(z)
 (D_{s',y'}F_0) \\
 & \qquad \qquad \qquad \times \gamma(s-s')f(y-y')dyd' dsds' dz \\
& \quad = \E \langle DF,K^{(r)}_{n-1}\rangle_{\cH\otimes L^2(\bR)}.
\end{align*}
For the third last equality we used the fact that $v_n \in |\cH|$ a.s. (since $E\|v_n\|_{|\cH|}^2<\infty$) and we assumed (without loss of generality) that $DF_0 \in |\cH|$, so that the inner product in $\cH$ can be expressed as an integral on $([0,t] \times \bR)^2$, according to Lemma \ref{L1-lemma} (Appendix \ref{appendix-Parseval}).
\end{proof}

In the following theorem, we prove that the Malliavin derivative
$Du(t,x)$ satisfies an equation in the space $L^2(\bR)$.

\begin{theorem}
\label{analogue-th-6-7}
For any $r \in [0,t]$, the following equality holds in $L^2(\Omega;L^2(\bR))$:
\begin{equation}
\label{eq-Malliavin-P}
D_{r,\cdot}u(t,x)=G(t-r,x-\cdot)u(r,\cdot)+\int_0^t \int_{\bR}G(t-s,x-y)D_{r,\cdot}u(s,y)W(\overline{\delta} s, \overline{\delta}y).
\end{equation}
\end{theorem}

\begin{proof} Recall that the stochastic integral on the right-hand side of \eqref{eq-Malliavin-P} is $\overline{\delta}(K^{(r)})$. By duality, it suffices to prove that for any $F \in \bD^{1,2}(L^2(\bR))$,
\begin{equation}
\label{analogue-80}
\E\langle D_{r,\cdot}u(t,x)-G(t-r,x-\cdot)u(r,\cdot), F\rangle_{L^2(\bR)}=\E \langle DF, K^{(r)}\rangle_{\cH \otimes L^2(\bR)}.
\end{equation}
By Proposition \ref{analogue-prop6-6}, we know that for any $F \in \bD^{1,2}(L^2(\bR))$,
\begin{equation}
\label{analogue-81}
\E\langle D_{r,\cdot}u_n(t,x)-G(t-r,x-\cdot)u_{n-1}(r,\cdot), F\rangle_{L^2(\bR)}=\E \langle DF, K_{n-1}^{(r)}\rangle_{\cH \otimes L^2(\bR)}.
\end{equation}
Relation \eqref{analogue-80} is obtained by taking $n \to \infty$ in
\eqref{analogue-81}. We justify this below.

On the right-hand side, using duality and Cauchy-Schwarz inequality,
we have:
\begin{align*}
\E \langle DF, K_{n-1}^{(r)}-K^{(r)}\rangle_{\cH \otimes L^2(\bR)} &
=\E \langle \overline{\delta}(K_{n-1}^{(r)}-K^{(r)}),F \rangle_{L^2(\bR)} \\
& \leq
\big(\E\|\overline{\delta}(K_{n-1}^{(r)}-K^{(r)})\|_{L^2(\bR)}^2\big)^{1/2}
\big(\E\|F\|_{L^2(\bR)}^2 \big)^{1/2}.
\end{align*}
By applying Proposition 6.2 of \cite{balan12} with ${\cal A}=L^2(\bR)$,
$$\E\|\overline{\delta}(K_{n-1}^{(r)}-K^{(r)})\|_{L^2(\bR)}^2 \leq \E \|K_{n-1}^{(r)}-K^{(r)}\|_{\cH \otimes L^2(\bR)}^2+
\E\|DK_{n-1}^{(r)}-DK^{(r)}\|_{\cH \otimes \cH \otimes L^2(\bR)}^2.$$

Repeating the same calculations as in the proofs of Theorems \ref{th-D-in-L2} and \ref{H-L2-norm} for the series $u(t,x)-u_{n-1}(t,x)=\sum_{k\geq n}I_k(f_k(\cdot,t,x))$, we arrive at the following estimates (similar to \eqref{L2-D-norm} and \eqref{L2-D2-norm} with $t=T$): for any $r \in [0,t]$,
\begin{align}
\label{bound-D-n}
\E\|D_{r,\cdot}u(t,x)-D_{r,\cdot} u_{n-1}(t,x)\|_{L^2(\bR)}^2  &
\leq \pi t e^{2M_t t}e^2 \sum_{k \geq n}\left(\frac{1}{2} \right)^{k-1}
=\frac{1}{2}C_t a_n, \\
\nonumber
\E \|D_r^2 u(t,x)-D_r^2 u_{n-1}(t,x)\|_{\cH \otimes L^2(\bR)}^2 & \leq \pi t e^{2M_t' t} e^3 \sum_{k\geq n}\left( \frac{1}{2}\right)^{k-1}=\frac{1}{2}C_t'' a_n,
\end{align}
where $a_n=\sum_{k\geq n}\left(\frac{1}{2} \right)^{k-1}$.
Using these bounds, and proceeding as in the proof of Lemma 3.8 for relations \eqref{bound-K-r} and \eqref{bound-DK-r}, we obtain that:
$$\E\|K_{n-1}^{(r)}-K^{(r)}\|_{\cH \otimes L^2(\bR)}^2 \leq \frac{1}{2}C_t a_n \phi(t) \to 0 \quad \mbox{as} \quad n \to \infty$$
and
$$\E\|DK_{n-1}^{(r)}-DK^{(r)}\|_{\cH \otimes \cH \otimes L^2(\bR)}^2 \leq \frac{1}{2}C_t'' a_n \phi(t) \to 0 \quad \mbox{as} \quad n \to \infty.$$

We now treat the left-hand side of \eqref{analogue-81}, which is a difference of two terms. For the first term, by the Cauchy-Schwarz inequality,
$$\E\langle D_{r,\cdot} u_n(t,x)-D_{r,\cdot}u(t,x),F\rangle_{L^2(\bR)} \leq \Big(\E \|D_{r,\cdot} u_n(t,x)-D_{r,\cdot}u(t,x) \|_{L^2(\bR)}^2 \Big)^{1/2} \Big( \E\|F\|_{L^2(\bR)}^2 \Big)^{1/2},$$
which converges to $0$ as $n\to \infty$, by \eqref{bound-D-n}.
For the second term, we use again Cauchy-Schwarz inequality, and we observe that
\begin{align*}
& \E\|G(t-r,x-\cdot)(u_{n-1}(r,\cdot)-u(r,\cdot))\|_{L^2(\bR)}^2 =\int_{\bR}G^2(t-r,x-z)
\E|u_{n-1}(r,z)-u(r,z)|^2dz
\end{align*}
which converges to $0$ as $n \to \infty$. This follows by the dominated convergence theorem, using the fact that $\int_\bR G^2(t-r,x-z)dz<\infty$ and
$$\sup_{(s,y) \in [0,t] \times \bR^d}\E|u_n(s,y)-u(s,y)|^2 \to 0,$$
which was shown in the proof of Theorem 7.1 of \cite{balan-song17}.
\end{proof}

\begin{remark}
\label{remark-eq-M}
{\rm
In what follows, we will use the following important consequence of Theorem  \ref{analogue-th-6-7}. If we denote by $B(r,\cdot)$ the right-hand side of \eqref{eq-Malliavin-P}, then obviously
$$\E\int_0^t \|D_{r,\cdot}u(t,x)-B(r,\cdot)\|_{L^2(\bR)}^2dr=0.$$
Hence, there exists a measurable set $N \subset \Omega \times [0,t]$ with $(P \times \lambda) (N)=0$ (where $\lambda$ is the Lebesgue measure on $\bR$), such that for all
$(\omega,r) \in (\Omega \times [0,t]) \verb2\2 N$
\begin{equation}
\label{D-equality}
\|D_{r,\cdot}u(t,x)(\omega)\|_{L^2(\bR)}=\|B(r,\cdot)(\omega)\|_{L^2(\bR)}.
\end{equation}
}
\end{remark}



\section{Proof of Theorem \ref{thm:density}}
\label{sec:density}

 In this section, we give the proof of Theorem \ref{thm:density},
which is based on Corollary \ref{new-cor} and Lemma \ref{lem:1} below.


We begin with a general result about absolute continuity of the law of a random variable, which corresponds to the remark right after the proof of Theorem 2.1.1 in \cite{nualart06}. We include the proof for the sake
of completeness.

\begin{lemma}
\label{lem:9}
If $F$ is a random variable such that $F\in \bD^{2,p}$ for some
$p>1$, then the measure $(\|DF\|^2_\cH \cdot \bP)\circ F^{-1}$ is
absolutely continuous with respect to the Lebesgue measure on $\bR$. In particular, for any Borel set $B \subset \bR$ with Lebesgue measure zero,
\begin{equation}
\label{P-zero} \bP(F \in B, \|DF \|_{\cH}>0)=0.
\end{equation}
\end{lemma}

\begin{proof}
Let $\varphi\in C^\infty_b(\bR)$ be arbitrary. By the chain rule for
the Malliavin derivative (see Proposition 1.2.3 of
\cite{nualart06}), $\varphi(F)\in \bD^{1,2}$ and
$D(\varphi(F))=\varphi'(F) DF$. Taking the scalar product with $DF$,
we obtain that
\[
\langle D(\varphi(F)),DF\rangle_\cH = \varphi'(F) \|DF\|^2_\cH.
\]
Note that $F\in \bD^{2,p}$ implies that $DF\in \bD^{1,p}(\cH)$.
(Recall that $\bD^{1,p}(\cH)$ is the completion of the space
$\cS(\cH)$ of ``smooth'' $\cH$-valued random variables $F$ with
respect to the norm $\|F\|_{\bD^{1,p}(\cH)}=\big(\E\|F\|_{\cH}^p+
\E\|DF\|_{\cH \otimes \cH}^p\big)^{1/p}$.) Since $\bD^{1,2}(\cH)
\subset  {\rm Dom } \, \delta$, it follows that
 $DF\in {\rm Dom } \, \delta$. By definition \eqref{def-delta} of
${\rm Dom} \, \delta$, we have
\begin{align*}
\big|\E \big(\|DF\|^2_\cH \varphi'(F)\big)\big| & = \big| \E \big(
\langle D\varphi(F),DF\rangle_\cH\big)\big| \leq c \big(\E |\varphi(F)|^2\big)^\frac12
\leq c \|\varphi\|_\infty.
\end{align*}
We have thus proved that, for all $\varphi\in C^\infty_b(\bR)$,
\[
\left| \int_\bR \varphi'(x) \big((\|DF\|^2_\cH \cdot \bP)\circ
F^{-1}\big) (dx)\right| \leq c \|\varphi\|_\infty.
\]

By Lemma 2.1.1 of \cite{nualart06}, the measure $\mu=(\|DF\|_{\cH}^2 \cdot P) \circ F^{-1}$ is absolutely continuous with respect to the Lebesgue measure.

To prove the last statement, let $B \subset \bR$ be a Borel set with Lebesgue measure zero. Then
$\int_{F^{-1}(B)} \|DF\|_{\cH}^2dP=\mu(B)=0$,
which implies \eqref{P-zero}.
\end{proof}

We continue with a result about existence of density
of a truncated random variable. First note that if $\Gamma$ is a
Borel set in $\bR$ such that $0 \not \in \Gamma$, then the law of
the truncated variable $G=F 1_{\{F \in \Gamma\}}$ is a probability
measure $\bP_G$ on $\Gamma \cup \{0\}$ given by $\bP_G(A)=\bP(F \in
A)$ for any Borel set $A \subset \Gamma$, and
$\bP_G(\{0\})=\bP(G=0)=\bP(F \not\in \Gamma)$.

\begin{lemma}
\label{lemma-Gamma}
Let $\Gamma$ be a Borel set in $\bR$ such that $0 \not \in \Gamma$. Let $F \in \bD^{2,p}$ for some $p>1$ be such that
\begin{equation}
\label{Gamma-set}
\|DF\|_\cH >0 \quad \text{a.s. on} \quad \{F\in \Gamma\}.
\end{equation}
Then, the restriction of the law of the variable $G=F 1_{\{F \in \Gamma\}}$ to the set $\Gamma$ is absolutely continuous with respect to the Lebesgue measure on $\Gamma$.
\end{lemma}

\begin{proof}
Let $B \subset \bR$ be a Borel set with zero Lebesgue measure. By \eqref{Gamma-set}, we have
\begin{equation}
\label{P-zero-2}
\bP(F\in \Gamma, \|DF\|_{\cH}=0)=0.
\end{equation}
By \eqref{P-zero} and \eqref{P-zero-2}, we obtain:
\begin{align*}
\bP(F\in B, F\in \Gamma) & = \bP(F\in B, F\in \Gamma,
\|DF\|_{\cH}>0)+\bP(F\in B, F\in \Gamma,
\|DF\|_{\cH}=0)=0.
\end{align*}

\end{proof}

\begin{lemma}
\label{lem:1}
 Let $(\Gamma_m)_{m\geq 1}$ be a sequence of open sets in $\bR$ such that $0 \not \in \Gamma_m$ and
 $\Gamma_m\subset \Gamma_{m+1}$, for all $m\geq 1$. Let $\Gamma=\cup_{m\geq 1}\Gamma_m$.
 Let $F\in \bD^{2,p}$ for some $p>1$ be such that, for all $m\geq 1$,
$$\|DF\|_\cH >0 \quad \text{a.s. on} \quad \{F\in \Gamma_m\}.$$
Then, the restriction of the law of the variable $F 1_{\{F \in \Gamma\}}$ to the set $\Gamma$ is absolutely continuous with respect to the Lebesgue measure on $\Gamma$.
\end{lemma}

\begin{proof}
Let $B \subset \bR$ be a Borel set with zero Lebesgue
measure. 
By Lemma \ref{lemma-Gamma},
$\bP(F \in B,F \in \Gamma_m)=0$ for all $m \geq 1$. By taking $m \to \infty$, we infer that $\bP(F \in
B,F \in \Gamma)=0$.
\end{proof}

In the proof of Theorem \ref{thm:density}, we will use the following notation:
\begin{equation}
\psi(t) :=  \int_0^t \int_\bR G(r,z)^2 dzdr
 = \frac14 \int_0^t \int_\bR 1_{\{|z|<r\}} dz dr
 = \frac{1}{4}t^2
 \label{eq:66}
\end{equation}
and
$$\psi_0(t):=
\int_0^t \int_{\bR^2} G(r,z)G(r,z') f(z-z')dz dz'=\int_0^t \int_\bR |\cF G(r,\cdot)(\xi)|^2
\mu(d\xi)dr.$$
By relation (A.3) in \cite{qs1}, for any $t \in (0,1)$,
\begin{equation}
\psi_0(t) \leq c_0 t,
 \label{eq:67}
\end{equation}
where $c_0=\frac{4}{3}\int_{\bR}\frac{1}{1+|\xi|^2}\mu(d\xi)$.

Recall that $\phi(t)$ is defined by \eqref{def-phi} and $\Gamma_t=2\int_0^t \gamma(s)ds$.
We will use the following lemma.

\begin{lemma}
\label{phi-psi-lem} For any $t>0$,
$$\phi(t) \leq \Gamma_t
\psi_0(t).$$
\end{lemma}

\begin{proof}
By Cauchy-Schwarz inequality and the inequality $ab \leq
\frac{1}{2}(a^2+b^2)$, $$\langle G(s,\cdot),G(s',\cdot)\rangle _0
\leq \|G(s,\cdot)\|_0 \|G(s',\cdot)\|_0 \leq
\frac{1}{2}\Big(\|G(s,\cdot)\|_0^2 +\|G(s',\cdot)\|_0^2 \Big).$$ By
symmetry, it follows that
$$\phi(t)=\int_0^t \int_0^t \gamma(s-s') \langle G(s,\cdot),G(s',\cdot) \rangle_0 dsds'\leq \int_0^t \int_0^t \gamma(s-s') \|G(s,\cdot)\|_0^2  dsds'.$$
The last integral is bounded by $\Gamma_t \psi_0(t)$ by Lemma
4.3 of \cite{balan-song17} (with $n=1$).
\end{proof}

\bigskip

{\bf Proof of Theorem \ref{thm:density}}: We will apply Lemma 4.3 
in the case $F=u(t,x)$ and $\Gamma_m=\{v\in \bR; |v|>\frac1m\}$. For any $m\geq 1$, let $\Omega_m:=\{|u(t,x)|>\frac1m\} \cap \widetilde{\Omega}$, where $\widetilde{\Omega}$ is an event of probability $1$ which will be defined below. By Lemma 3.3 
$u(t,x) \in \bD^{2,p}$ for any $p>1$.
We will prove that
$$\|D u(t,x)\|_{\cH}>0 \quad \mbox{a.s. on} \ \Omega_m.$$
In view of Corollary \ref{new-cor} (applied to $S=Du(t,x)(\omega)$ for fixed $\omega$), it is enough to prove that
\begin{equation}
\label{integral-positive}
\int_{0}^t \int_{\bR} |D_{r,z} u(t,x)|^2 dz dr >0  \quad \mbox{a.s. on} \ \Omega_m.
\end{equation}
By Remark 3.7 
the map $(r,z) \mapsto D_{r,z}u(t,x)(\omega)$ is measurable on $\bR_{+} \times \bR$, for any $\omega \in \Omega$.

We use Remark \ref{remark-eq-M}. By Fubini's theorem, $(\bP \times \lambda)(N)=\int_{\Omega}\lambda(N_{\omega})\bP(d\omega)$, where
$N_{\omega}:=\{r \in [0,t]; (\omega,r) \in N \}$ is the section of $N$ at point $\omega \in \Omega$.
Since $(\bP \times \lambda)(N)=0$, $\lambda(N_{\omega})=0$ for $\bP$-almost all $\omega \in \Omega$. Say this happens on the event $\widetilde{\Omega}$ of probability $1$.

Let $A=\big(\Omega \times [0,t] \big) \verb2\2 N$. Note that the section of $A$ at $\omega \in \Omega$ is the set
$$A_{\omega}:=\{r \in [0,t]; (\omega,r) \in A \}=\{r \in [0,t]; (\omega,r) \not \in N\}=[0,t]- N_{\omega}.$$

Fix $\omega \in \widetilde{\Omega}$. For the remaining part of the proof, we use an argument similar to the proof of Theorem 5.2 of \cite{NQ07} (for the white noise in time).

For any  $r \in A_{\omega}$, equality \eqref{D-equality} holds.
Using the inequality $\|a+b\|_{L^2(\bR)}^2 \geq \frac{1}{2}\|a\|_{L^2(\bR)}^2-\|b\|_{L^2(\bR)}^2$, we obtain:
\begin{align*}
& \|D_{r \cdot}u(t,x)\|_{L^2(\bR)}^2= \left\|G(t-r,x-\cdot)u(r,\cdot)+\int_0^t \int_{\bR}G(t-s,x-y)D_{r\cdot } u(s,y)W(\overline{\delta} s, \overline{\delta}y) \right\|_{L^2(\bR)}^2\\
& \quad \quad \geq \frac{1}{2}\|G(t-r,x-\cdot)u(r,\cdot)\|_{L^2(\bR)}^2-\left\|\int_0^t \int_{\bR}G(t-s,x-y)D_{r\cdot} u(s,y)
W(\overline{\delta} s, \overline{\delta}y) \right\|_{L^2(\bR)}^2.
\end{align*}

Let $\delta\in (0,1)$ be arbitrary. Taking the integral with respect to $r$ on $A_{\omega}$ (and using the fact that $\lambda(N_{\omega})=0$), we obtain that on $\widetilde{\Omega}$,
\begin{align}
 \int_{0}^t \|D_{r \cdot} u(t,x)\|^2_{L^2(\bR)}\, dr & \geq \int_{t-\delta}^t \|D_{r \cdot} u(t,x)\|^2_{L^2(\bR)}\, dr \nonumber \\
 & \geq  \frac{1}{2} \int_{t-\delta}^t \|G(t-r,x-\cdot) u(r,\cdot)\|^2_{L^2(\bR)} \, dr - I(\delta),
 \label{eq:33}
\end{align} where
\begin{align*}
 I(\delta) & = \int_{t-\delta}^t \left\| \int_{0}^t \int_{\bR} G(t-s,x-y) D_{r \cdot} u(s,y)
 W(\overline{\delta} s, \overline{\delta} y) \right\|^2_{L^2(\bR)} \, dr\\
 & = \int_{t-\delta}^t \left\| \int_{t-\delta}^t \int_{\bR} G(t-s,x-y) D_{r \cdot} u(s,y)
 W(\overline{\delta} s, \overline{\delta} y)\right\|^2_{L^2(\bR)} \, dr.
\end{align*}
The second equality is due to the fact that $D_{r \cdot}u(s,y)=0$ if $r>s$, and so we must have $s \geq r \geq t-\delta$.

On the event $\Omega_m$,
\begin{align*}
\int_{t-\delta}^t \int_\bR |G(t-r,x-z) u(r,z)|^2 dz dr & \geq
\frac{1}{m^2} \int_{t-\delta}^t \int_\bR G(t-r,x-z)^2 dz dr -
J(\delta) \\
& = \frac{1}{m^2} \psi(\delta) - J(\delta) =\frac{1}{4m^2} \delta^2 - J(\delta),
\end{align*}
where
\begin{align*}
 J(\delta) & = \int_{t-\delta}^t \int_\bR G^2(t-r,x-z) \big( u(t,x)^2
 - u(r,z)^2\big) dz dr.
\end{align*}
So, by \eqref{eq:33}, we have that
\begin{equation}
\label{LB-D}
\int_0^t \int_\bR |D_{r,z}u(t,x)|^2 dz dr \geq \frac{1}{8 m^2}\delta^2 - \frac12
J(\delta) - I(\delta).
\end{equation}

Let us now estimate the first moment of $J(\delta)$ and
$I(\delta)$, respectively. To start with, we have
\[
\E[|J(\delta)|] \leq\int_{t-\delta}^t \int_\bR G^2(t-r,x-z)
\E \big[| u(t,x)^2
 - u(r,z)^2|\big] dz dr,
\]
and we have
\begin{align*}
\E \big[| u(t,x)^2 - u(r,z)^2|\big] & = \E \big[|u(t,x)+u(r,z)|
\times |u(t,x)-u(r,z)| \big] \\
& \leq \left(\E \big[|u(t,x)+u(r,z)|^2\big]\right)^\frac12 \left(\E
\big[|u(t,x)-u(r,z)|^2\big]\right)^\frac12\\
& \leq 2 \sup_{(s,y)\in [0,t]\times \bR} \left(\E
\big[|u(s,y)|^2\big]\right)^\frac12  \left(\E
\big[|u(t,x)-u(r,z)|^2\big]\right)^\frac12 \\
& =2C_t^* \left(\E \big[|u(t,x)-u(r,z)|^2\big]\right)^\frac12,
\end{align*}
using the fact that $C_t^*:=\sup_{(s,y) \in [0,t] \times \bR}\E[|u(s,y)|^2]<\infty$, by Theorem 7.1 of \cite{balan-song17}.
Thus,
\begin{align}
\nonumber
\E[|J(\delta)|] & \leq 2 C_t^* \int_{t-\delta}^t \int_\bR G^2(t-r,x-z)
\left(\E \big[|u(t,x)-u(r,z)|^2\big]\right)^\frac12 dz dr \\
\nonumber
& \leq 2 C_t^* \int_{t-\delta}^t \int_\bR G^2(t-r,x-z) \sup_{t-\delta<s<t}
\sup_{|x-y|<t-s} \left(\E \big[|u(t,x)-u(s,y)|^2\big]\right)^\frac12
dzdr\\
\nonumber
& \leq 2 C_t^* g_{t,x}(\delta)\int_{t-\delta}^t \int_\bR G^2(t-r,x-z) dz
dr\\
\label{estimate-J}
& = \frac{1}{2}C_t^* g_{t,x}(\delta) \, \delta^2,
\end{align}
with
\[
g_{t,x}(\delta):= \sup_{|t-s|<\delta} \sup_{|x-y|<\delta} \left(\E
\big[|u(t,x)-u(s,y)|^2\big]\right)^\frac12.
\]
Note that $\lim_{\delta \to 0}g_{t,x}(\delta) =0$ because $u$ is
continuous in $L^2(\Omega)$, by Theorem 7.1 of \cite{balan-song17}.

Now we proceed to bound $E[I(\delta)]$. By Fubini's theorem,
$$\E[I(\delta)] =
 \int_{t-\delta}^t \E \left\| \int_{t-\delta}^t \int_{\bR} G(t-s,x-y) D_{r,\cdot} u(s,y)
 W(\overline{\delta} s, \overline{\delta} y)\right\|^2_{L^2(\bR)} \, dr.$$

Note that the integrand of the divergence operator $\overline{\delta}$ appearing above is the $L^2(\bR)$-valued process $K_{\delta}^{(r)}$ defined by $$K_{\delta}^{(r)}((s,y),\cdot)=1_{[t-\delta,t]}(s)G(t-s,x-y)D_{r,\cdot}u(s,y).$$
This process is similar to $K^{(r)}$, but contains the indicator of $[t-\delta,t]$ instead of $1_{[0,t]}$. Similarly to Lemma \ref{K-in-D}, it can be proved that $K_{\delta}^{(r)} \in {\rm Dom} \ \overline{\delta}$.

At this point, we apply Proposition 6.2 of \cite{balan12} (with ${\cal A}=L^2(\bR)$) to get that:
$$E\|\overline{\delta}(K_{\delta}^{(r)})\|^2 \leq \E\|K_{\delta}^{(r)}\|_{\cH \otimes L^2(\bR)}^2+\E\|DK_{\delta}^{(r)}\|_{\cH \otimes L^2(\bR) \otimes L^2(\bR)}^2.$$
Hence
$$\E[I(\delta)] \leq I_1(\delta)+I_2(\delta)$$
where
$$I_1(\delta)=\int_{t-\delta}^t \E\|K_{\delta}^{(r)}\|_{\cH \otimes L^2(\bR)}^2dr \quad \mbox{and} \quad I_2(\delta)=\int_{t-\delta}^t \E\|DK_{\delta}^{(r)}\|_{\cH \otimes L^2(\bR) \otimes L^2(\bR)}^2dr.$$

Similarly to \eqref{bound-K-r} and \eqref{bound-DK-r}, using Theorems \ref{th-D-in-L2} and \ref{H-L2-norm}, we obtain:
\begin{align*}
& \E\|K_{\delta}^{(r)}\|_{|\cH|\otimes L^2(\bR)}^2 \\
& \; \leq \int_{[t-\delta,t]^2} \int_{\bR^{2}}G(t-s,x-y)G(t-s',x-y')
\Big(\E \| D_{r,\cdot}u(s,y)\|_{L^2(\bR)}^2 \Big)^{1/2}  \\
& \qquad \qquad \times \Big(\E \| D_{r,\cdot}u(s',y')\|_{L^2(\bR)}^2
\Big)^{1/2}
 \gamma(s-s')f(y-y')dydy'dsds' \\
& \; \leq C_t \phi(\delta),
\end{align*}
and
\begin{align*}
& \E \|DK_{\delta}^{(r)}\|_{|\cH|\otimes \cH \otimes L^2(\bR)}^2 \\
\nonumber
& \; \leq \int_{[t-\delta,t]^2} \int_{\bR^{2}}G(t-s,x-y)G(t-s',x-y')
\Big(\E \| D_{r}^2 u(s,y)\|_{\cH \otimes L^2(\bR)}^2 \Big)^{1/2} \nonumber \\
& \qquad \qquad \times  \Big(\E \| D_{r}^2 u(s',y')\|_{\cH \otimes
L^2(\bR)}^2 \Big)^{1/2}
\gamma(s-s')f(y-y')dydy'dsds' \\
& \; \leq C_t'' \phi(\delta).
\end{align*}
Here $D_r^2 u(t,x)$ denotes $D_{*,(r,\cdot)}^2$, where $*$ is the
missing $(\theta,w)$-variable and $\cdot$ is the missing
$z$-variable.
By Lemma \ref{phi-psi-lem} and relation \eqref{eq:67}, $\phi(\delta) \leq \Gamma_{\delta}
\psi(\delta)\leq c_0 \Gamma_{\delta} \delta$, and hence
$$I_1(\delta) \leq c_0 C_t \delta^2  \Gamma_{\delta} \quad \mbox{and} \quad I_2(\delta) \leq c_0 C_t'' \delta^2 \Gamma_{\delta} .$$

We obtain:
\begin{equation}
\label{estimate-I}
\E[I(\delta)]\leq c_0(C_t+C_t'') \delta^2 \Gamma_\delta,
\end{equation}

Taking into account the later estimate and the one obtained for $E[|J(\delta)|]$ (see \eqref{estimate-J}), we will be able to conclude the proof, as follows. Using \eqref{LB-D} and Markov's inequality, for any $n\geq 1$, we have
\begin{align*}
& \bP\left( \left\{\int_{0}^t \int_\bR |D_{r,z} u(t,x)|^2 \, dz dr <\frac1n\right\} \cap
 \Omega_m \right) \leq
 \bP\left(I(\delta) + \frac{1}{2} J(\delta) > \frac{1}{8m^2}\delta^2 - \frac1n \right)\\
 & \quad \leq \left( \frac{1}{8m^2} \delta^2 - \frac1n\right)^{-1} \Big(\E[I(\delta)] + \frac{1}{2} \E[|J(\delta)|]\Big)\\
 & \quad \leq \left( \frac{1}{8m^2} \delta^2 - \frac1n\right)^{-1} \delta^2
 \Big(c_0(C_t+C_t'') \Gamma_\delta + \frac{1}{4}C_t^* g_{t,x}(\delta)\Big).
\end{align*}
Taking $n\rightarrow \infty$, one gets
\[
 \bP\left( \left\{\int_{0}^t \int_\bR |D_{r,z} u(t,x)|^2 \,dz dr =0\right\} \cap \Omega_m \right)\leq
 \left(\frac{1}{8m^2} \delta^2 \right)^{-1}
  \delta^2 \Big(c_0(C_t+C_t'') \Gamma_\delta + \frac{1}{4}C_t^* g_{t,x}(\delta)\Big).
\]
Next, we take $\delta\rightarrow 0$. Since $\lim_{\delta \to 0}\Gamma_{\delta}= 0$ and $\lim_{\delta \to 0}g_{t,x}(\delta)=0$, we obtain:
$$\bP\left( \left\{\int_{0}^t \int_\bR |D_{r,z} u(t,x)|^2 \, dzdr =0\right\} \cap \Omega_m \right)=0.$$
This concludes the proof of \eqref{integral-positive} and the proof of Theorem \ref{thm:density}.


\appendix

\section{A Parseval-type identity}
\label{appendix-Parseval}

In this section, we give the Parseval-type identity which is used in
the proof of {Lemma \ref{max-principle}. We begin by recalling a
remarkable result of \cite{KX09}, and comment on a small correction
of a related result from the same paper.

Let $f:\bR^d \to [0,\infty]$ be a kernel of {\em positive type}, i.e. $f$ is locally integrable and its Fourier transform in $\cS'(\bR^d)$ is a function $g$ which is non-negative almost everywhere. In addition, we suppose that $f$ is continuous, symmetric and $f(x)<\infty$ if and only if $x\not=0$.

By Lemma 5.6 of \cite{KX09}, for any Borel probability measure $\mu$ on $\bR^d$, we have:
$$\int_{\bR^d}\int_{\bR^d}f(x-y)\mu(dx) \mu(dy)=\frac{1}{(2\pi)^d}\int_{\bR^d}|\cF \mu(\xi)|^2 g(\xi)d\xi=:\cE_f(\mu).$$

In particular, if $\mu(dx)=\varphi(x)dx$ where $\varphi$ is a density function on $\bR^d$,
\begin{equation}
\label{identity1}
\int_{\bR^d}\int_{\bR^d}f(x-y)\varphi(x) \varphi(y)dx dy=\frac{1}{(2\pi)^d}\int_{\bR^d}|\cF \varphi(\xi)|^2 g(\xi)d\xi=:\cE_f(\varphi).
\end{equation}

It follows that relation \eqref{identity1} holds for any {\em non-negative} function $\varphi \in L^1(\bR^d)$.

Relation (5.37) of \cite{KX09} says that
\begin{equation}
\label{KX1}
\int_{\bR^d}\int_{\bR^d}f(x-y)\mu(dx) \nu(dy)=\frac{1}{(2\pi)^d}\int_{\bR^d}\cF \mu(\xi) \overline{\cF \nu(\xi)} g(\xi)d\xi
\end{equation}
for any Borel probability measures $\mu$ and $\nu$ on $\bR^d$. When $\nu=\delta_0$ and $\mu=\delta_x$ for $x \in \bR$ arbitrary, this relation becomes
$$f(x)=\frac{1}{(2\pi)^d}\int_{\bR^d}e^{-i \xi \cdot x}g(\xi)d\xi, \quad x \in \bR,$$
which is not true if $g$ is not integrable. The problem is caused by the fact that in the proof of (5.37), on the right hand side of (5.39), we may have $\infty-\infty$.

But \eqref{KX1} {\em does hold} for any Borel probability measures $\mu$ and $\nu$ on $\bR^d$ such that $\cE_f(\mu)<\infty$ and $\cE_f(\nu)<\infty$.
In particular, 
\begin{equation}
\label{identity2}
\int_{\bR^d}\int_{\bR^d}f(x-y)\varphi(x) \psi(y)dx dy=\frac{1}{(2\pi)^d}\int_{\bR^d}\cF \varphi(\xi) \overline{\cF \psi(\xi)} g(\xi)d\xi=:\cE_f(\varphi,\psi)
\end{equation}
for any density functions $\varphi$ and $\psi$ on $\bR^d$ with $\cE_f(\varphi)<\infty$ and $\cE_f(\psi)<\infty$.

It follows that relation \eqref{identity2} holds for any {\em non-negative} functions $\varphi,\psi \in L^1(\bR^d)$ with $\cE_f(\varphi)<\infty$ and $\cE_f(\psi)<\infty$. (To see this, we write \eqref{identity2} for the density functions $\varphi/\|\varphi\|_1$ and  $\psi/\|\psi\|_1$ and then we multiply by $\|\varphi\|_1 \|\psi\|_1$.) Moreover, in this case $|\cE_f(\varphi,\psi)|<\infty$ since by the Cauchy-Schwarz inequality
$$\left|\int_{\bR^d}\cF \varphi(\xi) \overline{\cF \psi(\xi)} g(\xi)d\xi \right|\leq \left( \int_{\bR^d}|\cF \varphi(\xi)|^2 g(\xi)d\xi \right)^{1/2} \left( \int_{\bR^d}|\cF \psi(\xi)|^2 g(\xi)d\xi \right)^{1/2}$$

In fact, we have the following more general result.

\begin{lemma}
\label{L1-lemma}
Relation \eqref{identity2} holds for any functions $\varphi,\psi \in L^1(\bR^d)$ with $\cE_f(|\varphi|)<\infty$ and $\cE_f(|\psi|)<\infty$, and in this case $|\cE_f(\varphi,\psi)|<\infty$.
\end{lemma}

\noindent {\bf Proof:} We write $\varphi=\varphi^{+}-\varphi^{-}$,
where $\varphi^{+}=\max(\varphi,0)$ and
$\varphi^{-}=\max(-\varphi,0)$. Similarly, $\psi=\psi^{+}-\psi^{-}$,
where $\psi^{+}=\max(\psi,0)$ and $\psi^{-}=\max(-\psi,0)$. Then
$|\varphi|=\varphi^{+}+\varphi^{-}$. Since $\varphi^{+}(x)\leq
|\varphi(x)|$ for any $x \in \bR$, we have
\begin{align*}
\cE_{f}(\varphi^{+}) &
=\int_{\bR^d}\int_{\bR^d}f(x-y)\varphi^{+}(x)\varphi^{+}(y)dxdy
\\
& \leq \int_{\bR^d}\int_{\bR^d}f(x-y)|\varphi(x)| |\varphi(y)|dxdy
=\cE_{f}(|\varphi|)<\infty.
\end{align*}
 Similarly, we obtain that
$\cE_{f}(\varphi^{-})<\infty$, $\cE_{f}(\psi^{+})<\infty$ and
$\cE_{f}(\psi^{-})<\infty$.

Hence relation \eqref{identity2} holds for the pairs $(\varphi^+,\psi^+)$, $(\varphi^-,\psi^-)$, $(\varphi^+,\psi^-)$ and $(\varphi^-,\psi^+)$:
\begin{eqnarray*}
\int_{\bR^d}\int_{\bR^d}f(x-y)\varphi^+(x) \psi^+(y)dx dy&=&\frac{1}{(2\pi)^d}\int_{\bR^d}\cF \varphi^+(\xi) \overline{\cF \psi^+(\xi)} g(\xi)d\xi\\
\int_{\bR^d}\int_{\bR^d}f(x-y)\varphi^-(x) \psi^-(y)dx dy&=&\frac{1}{(2\pi)^d}\int_{\bR^d}\cF \varphi^-(\xi) \overline{\cF \psi^-(\xi)} g(\xi)d\xi \\
\int_{\bR^d}\int_{\bR^d}f(x-y)\varphi^+(x) \psi^-(y)dx dy&=&\frac{1}{(2\pi)^d}\int_{\bR^d}\cF \varphi^+(\xi) \overline{\cF \psi^-(\xi)} g(\xi)d\xi \\
\int_{\bR^d}\int_{\bR^d}f(x-y)\varphi^-(x) \psi^+(y)dx dy&=&\frac{1}{(2\pi)^d}\int_{\bR^d}\cF \varphi^-(\xi) \overline{\cF \psi^+(\xi)} g(\xi)d\xi
\end{eqnarray*}
and all the integrals appearing above are finite. We take the sum of the first two relations above and from this, we subtract the sum of the last two. We obtain
$$\int_{\bR^d}\int_{\bR^d}f(x-y)(\varphi^+(x)-\varphi^-(x))(\psi^+(y)-\psi^-(y))dx dy=$$
$$\frac{1}{(2\pi)^d}\int_{\bR^d}(\cF \varphi^+(\xi) -\cF \varphi^{-}(\xi))\overline{(\cF \psi^+(\xi)-\cF \psi^-(\xi))} g(\xi)d\xi,$$
which is exactly relation \eqref{identity2}. Finally, we note that
$$|\cE_f(\varphi,\psi)|\leq \int_{\bR^d} \int_{\bR^d} f(x-y)|\varphi(x)||\psi(y)|dxdy =\cE_f(|\varphi|,|\psi|)\leq \cE_f(|\varphi|)^{1/2}\cE_f(|\psi|)^{1/2}<\infty. $$
$\Box$

\vspace{3mm}

For complex-valued functions, we have the following result.

\begin{lemma}
\label{L1C-lemma}
For any functions $\varphi,\psi \in L_{\bC}^1(\bR^d)$ with $\cE_f(|\varphi|)<\infty$ and $\cE_f(|\psi|)<\infty$, we have:
\begin{equation}
\label{identity3}
\int_{\bR^d}\int_{\bR^d}f(x-y)\varphi(x) \overline{\psi(y)}dx dy=\frac{1}{(2\pi)^d}\int_{\bR^d}\cF \varphi(\xi) \overline{\cF \psi(\xi)} g(\xi)d\xi=:\cE_f(\varphi,\psi),
\end{equation}
and in this case, $|\cE_f(\varphi,\psi)|<\infty$.
\end{lemma}

\noindent {\bf Proof:} We write $\varphi=\varphi_1+i\varphi_2$ and
$\psi=\psi_1+i\psi_2$, where $\varphi_1,\varphi_2,\psi_1,\psi_2 \in
L^1(\bR^d)$. Note that $|\varphi|^2=|\varphi_1|^2+|\varphi_2|^2$. It
follows that $|\varphi_1(x)|\leq |\varphi(x)|$ for all $x \in \bR$,
and hence
\begin{align*}
\cE_{f}(|\varphi_1|)&
=\int_{\bR^d}\int_{\bR^d}f(x-y)|\varphi_{1}(x)||\varphi_{1}(y)|dxdy\\
& \leq \int_{\bR^d}\int_{\bR^d}f(x-y)|\varphi(x)| |\varphi(y)|dxdy
=\cE_{f}(|\varphi|)<\infty.
\end{align*}
 Similarly,
$\cE_{f}(|\varphi_2|)<\infty$, $\cE_{f}(|\psi_1|)<\infty$ and
$\cE_{f}(|\psi_2|)<\infty$.

Note that
\begin{eqnarray*}
\varphi(x) \overline{\psi(y)}&=&[\varphi_1(x)+i\varphi_2(x)][\psi_1(y)-i\psi_2(y)]\\
&=& [\varphi_1(x)\psi_1(y)+\varphi_2(y)\psi_2(y)]+
i[\varphi_2(x)\psi_1(y)-\varphi_1(y)\psi_2(y)]
\end{eqnarray*}
and
\begin{eqnarray*}
\cF \varphi(\xi) \overline{\cF \psi(\xi)}&=&[\cF \varphi_1(\xi) +i \cF \varphi_2(\xi) ][\overline{\cF \psi_1(\xi)} -i\overline{ \cF \psi_2(\xi)}]\\
&=&\cF \varphi_1(\xi)\overline{\cF \psi_1(\xi)}+\cF \varphi_2(\xi)\overline{\cF \psi_2(\xi)}+i [\cF \varphi_2(\xi)\overline{\cF \psi_1(\xi)}-\cF \varphi_1(\xi)\overline{\cF \psi_2(\xi)}].
\end{eqnarray*}

We apply Lemma \ref{L1-lemma} to the pairs of functions $(\varphi_1,\psi_1)$, $(\varphi_2,\psi_2)$, $(\varphi_2,\psi_1)$, $(\varphi_1,\psi_2)$:
\begin{eqnarray*}
\int_{\bR^d}\int_{\bR^d}f(x-y)\varphi_1(x) \psi_1(y)dx dy&=&\frac{1}{(2\pi)^d}\int_{\bR^d}\cF \varphi_1(\xi) \overline{\cF \psi_1(\xi)} g(\xi)d\xi\\
\int_{\bR^d}\int_{\bR^d}f(x-y)\varphi_2(x) \psi_2(y)dx dy&=&\frac{1}{(2\pi)^d}\int_{\bR^d}\cF \varphi_2(\xi) \overline{\cF \psi_2(\xi)} g(\xi)d\xi \\
\int_{\bR^d}\int_{\bR^d}f(x-y)\varphi_2(x) \psi_1(y)dx dy&=&\frac{1}{(2\pi)^d}\int_{\bR^d}\cF \varphi_2(\xi) \overline{\cF \psi_1(\xi)} g(\xi)d\xi \\
\int_{\bR^d}\int_{\bR^d}f(x-y)\varphi_1(x) \psi_2(y)dx dy&=&\frac{1}{(2\pi)^d}\int_{\bR^d}\cF \varphi_1(\xi) \overline{\cF \psi_2(\xi)} g(\xi)d\xi,
\end{eqnarray*}
where all the integrals above are finite. We take the sum of the first two relations and then we add the difference between the third and fourth relations multiplied by $i$. We obtain:
$$
\int_{\bR^d}\int_{\bR^d}f(x-y)\{\varphi_1(x) \psi_1(y)+\varphi_2(x) \psi_2(y)-
i[\varphi_2(x) \psi_1(y)+\varphi_1(x) \psi_2(y)]\}dx dy=$$
$$\frac{1}{(2\pi)^d}\int_{\bR^d}\{\cF \varphi_1(\xi) \overline{\cF \psi_1(\xi)}+
\cF \varphi_2(\xi) \overline{\cF \psi_2(\xi)}+i[\cF \varphi_2(\xi) \overline{\cF \psi_1(\xi)}-
\cF \varphi_1(\xi) \overline{\cF \psi_2(\xi)} ]\}g(\xi)d\xi,$$
which is exactly relation \eqref{identity3}.
The fact that $|\cE_f(\varphi,\psi)|<\infty$ follows as in Lemma \ref{L1-lemma}.
$\Box$

\section{Existence of measurable modifications}
\label{appendix-meas}

In this section, we prove a result about existence of measurable modifications of random fields, which is
used frequently in the literature on SPDEs using Walsh' approach \cite{walsh86}. Parts a) and b) of this result are extensions to random fields of Theorem 30 in Chapter IV of \cite{dellacherie-meyer75}, respectively Proposition 3.21 of \cite{PZ07}.

Part a) of this result is used in the proof of Theorem \ref{D-measurable-th}. Part b) of this result is not needed in the present article, but we include it here since its proof requires only a minor modification of the proof of part a). We include the proof for the sake of completeness.

A random field is a collection $X=\{X(t,x);t \geq 0,x \in \bR^d\}$ of random variables defined on the same probability space $(\Omega,\cF,P)$. We say that $X$ is {\em stochastically continuous} if it is continuous in probability, i.e. $X(t_n,x_n) \stackrel{P}{\to} X(t,x)$ if $t_n \to t$ and $x_n \to x$.
$X$ is {\em predictable} with respect to
a filtration $(\cF_t)_{t \geq 0}$ if the map $(\omega,t,x) \mapsto X(\omega,t,x)$ is measurable with respect to the {\em predictable $\sigma$-field} on $\Omega \times \bR_{+} \times \bR^d$, which is the $\sigma$-field generated by elementary processes. (An {\em elementary process} is a process of the form $Y(\omega,t,x)=Y_0(\omega)1_{(a,b]}(t)1_{A}(x)$, where $Y_0$ is $\cF_a$-measurable, $0<a<b$ and $A \in \cB(\bR^d)$ is a bounded set.) $X$ is {\em adapted} with respect to
a filtration $(\cF_t)_{t \geq 0}$ if $X(t,x)$ is $\cF_t$-measurable, for any $t \geq 0$ and $x\in \bR^d$.

\begin{proposition}
\label{exist-meas}
a) Any stochastically continuous random field has a measurable modification.
b) Any stochastically continuous random field which is adapted with respect to
a filtration $(\cF_t)_{t \geq 0}$ has a predictable modification, with respect to the same filtration.
\end{proposition}

\begin{proof} a) Let $X=\{X(t,x);t \geq 0,x \in \bR^d\}$ be a stochastically continuous random field defined on a probability space $(\Omega,\cF,P)$.
Let $(E_m)_{m \geq 1}$ be an increasing sequence of compact sets in $\bR^d$ such that $\cup_m E_m=\bR^d$. Fix $m \geq 1$ and let $I=[0,m] \times E_m$. Since $X$ is stochastically continuous, it is uniformly stochastically continuous on $I$, i.e. for any $\varepsilon>0$ there exists $\delta_{\varepsilon}>0$ such that for any $(t,x),(s,y) \in I$ with $|t-s|^2+|x-y|^2 \leq \delta_m^2$,
$$\bP(|X(t,x)-X(s,y)|>2^{-m}) \leq 2^{-m}.$$

Let $0=t_0^{(m)}<t_1^{(m)}< \ldots<t_{n_m}^{(m)}=m$ be a partition of $[0,m]$ into subintervals of length smaller than $\delta_m$ and $(U_{l}^{(m)})_{l=1, \ldots,K_m}$ be a partition of $E_m$ into Borel sets of diameter smaller than $\delta_m$.
(The diameter of a set $S$ is defined as $\sup\{|x-y|; x,y \in S\}$.) Let $x_{l}^{(m)} \in U_l^{(m)}$ be arbitrary. For any $t \in (t_k^{(m)}, t_{k+1}^{(m)}]$ and $x \in U_l^{(m)}$,
\begin{equation}
\label{prob-Am} \bP(|X(t_k^{(m)},x_{l}^{(m)})-X(t,x)|>2^{-m}) \leq
2^{-m}.
\end{equation}

Define
$$X_m(\omega,t,x)=\sum_{k=0}^{n_m-1}\sum_{l=1}^{K_m}X(\omega,t_{k}^{(m)},x_{l}^{(m)})
1_{(t_k^{(m)},t_{k+1}^{(m)}]}(t)1_{U_l^{(m)}}(x).$$
Note that $X_m$ is a measurable process, since $X(t_{k}^{(m)},x_{l}^{(m)})$ is $\cF$-measurable.

Let $A=\{(\omega,t,x) \in \Omega \times \bR_{+} \times \bR^d; (X_{m}(\omega,t,x))_{m \geq 0} \ \mbox{converges}\}$. Then $A\in \cF \times \cB(\bR_+) \times \cB(\bR^d)$ and the process $\widetilde{X}$ defined by
$$\widetilde{X}(\omega,t,x)=1_{A}(\omega,t,x)\lim_{m \to \infty}X_{m}(\omega,t,x)$$
is measurable.

We now show that $\widetilde{X}$ is a modification of $X$. Let
$(t,x) \in \bR_{+} \times \bR^d$ be arbitrary. Then $(t,x) \in I_m$
for $m$ large enough, and $t \in (t_k^{(m)},t_{k+1}^{(m)}]$ for some
$k=k_m$ and $x \in U_{l}^{(m)}$ for some $l=l_m$. Let
$$A_m=\{|X(t_k^{(m)},x_{l}^{(m)})-X(t,x)|>2^{-m}\}.$$ By
(\ref{prob-Am}), $\sum_{m}\bP(A_m)<\infty$. By the Borel-Cantelli
Lemma, $\bP(\Omega_{t,x})=1$ where
$\Omega_{t,x}=\liminf_{m}A_m^c=\cup_{m_0} \cap_{m \geq m_0}A_m^c$.
Let $\omega \in \Omega_{t,x}$. Then there exists $m_0=m_0(\omega)$
such that
$$|X(\omega,t_k^{(m)},x_l^{(m)})-X(\omega,t,x)| \leq 2^{-m} \quad \forall m \geq m_0(\omega).$$
By the definition of $X_m$, it follows that
$$|X_m(\omega,t,x)-X(\omega,t,x)| \leq 2^{-m} \quad \forall m \geq m_0(\omega).$$
From this, we infer that $\lim_{m \to
\infty}X_m(\omega,t,x)=X(\omega,t,x)$, i.e. $(\omega,t,x) \in A$.
This shows that $\widetilde{X}(\omega,t,x)=X(\omega,t,x)$ for any
$\omega \in \Omega_{t,x}$. Since $\bP(\Omega_{t,x})=1$, we infer
that $\bP(\widetilde{X}(t,x)=X(t,x))=1$, as required.

b) In this part, we assume, in addition, that $X$ is adapted. Then $X_m$ is predictable, being a linear combination of elementary processes. Hence $A$ and $\widetilde{X}$ are predictable.
\end{proof}

\begin{remark}
{\rm Proposition \ref{exist-meas}.a) follows from a
general result of Cohn: see Theorem 3 of \cite{cohn72}, and the
(first) Remark on page 164 of \cite{cohn72}. }
\end{remark}

\section*{Acknowledgement}
The first author is grateful to Mohammud Foondun who asked her (in a personal communication) about
the existence of a predictable modification of a random field. The answer to his question is given by Proposition \ref{exist-meas}.b). The authors would like to thank David Nualart for useful discussions, and for sending them the preliminary version of the preprint \cite{huang-nualart-viitasaari}.



\begin{thebibliography}{99}


\bibitem{balan12} Balan, R. M. (2012). The stochastic wave equation with multiplicative fractional noise: a Malliavin calculus approach. {\em Potential Anal.} {\bf 36}, 1-34.

\bibitem{balan-chen} Balan, R. M. and Chen, L. (2018). Parabolic Anderson model with space-time homogeneous Gaussian noise and rough initial conditions. {\em J. Theor. Probab.} 31, 2216-2265.

\bibitem{balan-song17} Balan, R. M. and Song, J. (2017). Hyperbolic Anderson Model with space-time homogeneous Gaussian noise.
{\em Latin Amer. J. Probab. Math. Stat.} {\bf 14}, 799-849.

\bibitem{balan-song17b} Balan, R. M. and Song, J. (2017). Second order Lyapunov exponents for parabolic and hyperbolic Anderson models. Preprint available on arxiv:1704.02411.

\bibitem{bally-pardoux98} Bally, V. and Pardoux, E. (1998). Malliavin calculus for white noise driven parabolic SPDEs. {\em Potential Anal.} {\bf 9}, 27-64.

\bibitem{BGP12} Basse-O'Connor, A., Graversen, S.-E. and Pedersen, J. (2012). Multiparameter processes
with stationary increments. Spectral representation and intregration. {\em Electr. J. Probab.} {\bf 17}, paper no. 74, 21 pages.

\bibitem{carmona-nualart88} Carmona, R. and Nualart, D. (1988). Random non-linear wave equations: smoothness of the solutions. {\em Probab. Th. Rel. Fields} {\bf 79}, 469-508.

\bibitem{cohn72} Cohn, D. L. (1972). Measurable choice of limit points and the existence of separable and measurable processes. {\em Z. Wahr. verw. Geb.} {\bf 22}, 161-165.

\bibitem{dalang99} Dalang R.C. (1999) Extending martingale measure stochastic integral with applications to spatially homogeneous spde's. {Electr. J. Probab.} {\bf 4}.

\bibitem{dalang-frangos98} Dalang, R. C. and Frangos, N. (1998). The stochastic wave equation in two spatial dimensions. {\em Ann. Probab.} {\bf 26}, 187-212.

\bibitem{dalang-sanz05} Dalang, R. C. and Sanz-Sol\'e, M. (2005). Regularity of the sample paths of a class of second-order spde's. {\em J. Funct. Anal.} {\bf 227}, 304-337.

\bibitem{delgato-nualart-zheng} Delgado-Vences, F.,
Nualart, D. and Zheng, G. (2018). A central limit theorem for the
stochastic wave equation with fractional noise. Preprint available
on arXiv:1812.05019.

\bibitem{dellacherie-meyer75} Dellacherie, C. and Meyer, P.A. (1975). {\em Probabilit\'es et Potentiel.} Hermann, Paris.

\bibitem{HHNT} Hu, Y., Huang, J., Nualart, D. and Tindel, S. (2015). Stochastic heat equations with general multiplicative Gaussian noises: H\"{o}lder continuity and intermittency. {\em Electr. J. Probab.} {\bf 20}, no. 55, 1-50.

\bibitem{hu-le18} Hu, Y. and L\^e, K. (2018). Asumptotics of the density of parabolic Anderson random fields. Preprint available on arxiv:1801.03386

\bibitem{hu-nualart-song11} Hu, Y., Nualart, D. and Song, J. (2011).
Feynman-Kac formula for heat equation driven by fractional white
noise. {\em Ann. Probab.} {\bf 39}, 291-326.

\bibitem{huang-le-nualart17} Huang, J., L\^e, K. and Nualart, D. (2017). Large time asymptotics for the Parabolic Anderson model driven by spatially correlated noise. {\em Ann. Inst. Henri Poincar\'e} {\bf 53}, 1305-1340.

\bibitem{huang-nualart-viitasaari} Huang, J., Nualart, D. and Viitasaari, L. (2018). A central limit theorem for the stochastic heat equation. Preprint available on arXiv:1810.09492.


\bibitem{Ito} It\^o, K. (1954) Stationary random distributions. {\em Mem. Coll. Sci. Univ. Kyoto Ser. A Math.}
{\bf 28}, 209-223.

\bibitem{jolis10} Jolis, M. (2010). The Wiener integral with respect to second-order processes with stationary increments. {\em J. Math. Anal. Appl.} {\bf 366}, 607-620.

\bibitem{KX09} Khoshnevisan, D. and Xiao, Y. (2009).
 Harmonic analysis of additive L\'evy processes.
 {\em Probab. Th. Rel. Fields} {\bf 145}, 459-515.

\bibitem{mms} M\'arquez-Carreras, D., Mellouk, M., Sarr\`a, M.
(2001). On stochastic partial differential equations with spatially
correlated noise: smoothness of the law. {\em Stoch. Proc. Appl.}
{\bf 93}, 269-284.

\bibitem{millet-sanz99} Millet, A. and Sanz-Sol\'e, M. (1999). A stochastic wave equation in two space dimension: smoothness of the law. {\em Ann. Probab.} {\bf 27}, 803-844.

\bibitem{mn} Mueller, C. and Nualart, D. (2008). Regularity of the density for the stochastic heat
equation. {\em Electr. J. Probab.} {\bf 13}, Paper no. 74,
2248-2258.

\bibitem{nualart98} Nualart, D. (1998). Analysis on Wiener space and anticipative stochastic calculus. {\em Lecture Notes Math.} {\bf 1690}. Ecole d'ete de probabilit\'es de Saint Flour XXV-1995, P. Bernard Ed. 123-227.

\bibitem{nualart06} Nualart, D. (2006). {\em Malliavin Calculus and Related Topics}. Second Edition. Springer.

\bibitem{NQ07} Nualart, D. and Quer-Sardanyons, L. (2007).
Existence and smoothness of the density for spatially homogeneous
SPDEs. {\em Potential Anal.} {\bf 27}, 281-299.

\bibitem{pardoux-zhang93} Pardoux, E. and Zhang, T. (1993). Absolute continuity for the law of the solution of a parabolic SPDE. {\em J. of Funct. Anal.} {\bf  112}, 447-458.

\bibitem{peszat02} Peszat, S. (2002). The Cauchy problem for a nonlinear stochastic wave equation in any dimension. {\em J. Evol. Equ.} {\bf 2}, 383-394.

\bibitem{PZ00} Peszat, S. and Zabczyk, J. (2000). Nonlinear stochastic wave and heat equations.
{\em Probab. Th. Rel. Fields} {\bf 116}, 421-443.

\bibitem{PZ07} Peszat, S. and Zabczyk, J. (2007). {\em Stochastic partial differential equations with L\'evy noise}. Cambridge University Press.

\bibitem{qs1} Quer-Sardanyons, L., Sanz-Sol\'e, M. (2004). Absolute continuity of the law of the
solution to the 3-dimensional stochastic wave equation. {\em J.
Funct. Anal.} {\bf{206}}, No.1, 1-32.

\bibitem{qs2} Quer-Sardanyons, L., Sanz-Sol\'e, M. (2004) A stochastic wave equation in dimension 3:
Smoothness of the law. {\em Bernoulli} {\bf{10}}, No.1, 165-186.

\bibitem{sanz-sus13} Sanz-Sol\'e, M. and S\"u\ss, A. (2013). The stochastic wave equation in high dimensions: Malliavin differentiability and absolute continuity. {\em Electr. J. Probab.} {\bf 18}, no. 64, 1-28

\bibitem{sanz-sus15} Sanz-Sol\'e, M. and S\"u\ss, A. (2015). Absolute continuity for SPDEs with irregular fundamental solution. {\em Electr. Comm. Probab.} {\bf 20}, no.14, 1-11.

\bibitem{walsh86}  Walsh, J. B. (1986). An introduction to stochastic
partial differential equations. Ecole d'Et\'{e} de
Probabilit\'{e}s de Saint-Flour XIV. {\em Lecture Notes in Math.} {\bf 1180}, 265-439, Springer-Verlag. 

\end{thebibliography}
\end{document}